\documentclass[a4paper,11pt]{amsart}
\usepackage{mathptmx}
\usepackage{amsmath}
\usepackage{amscd}
\usepackage{amssymb, bm}
\usepackage{amsthm}
\usepackage{xspace}
\usepackage[all,tips]{xy}
\usepackage[dvips]{graphicx}
\usepackage{verbatim}
\usepackage{syntonly}
\usepackage{hyperref}
\usepackage{graphics, fullpage,color, epsfig,url}
\usepackage{indentfirst}
\usepackage{esint}
\usepackage{enumitem}
\usepackage[dvipsnames]{xcolor}
\usepackage{soul, amsaddr}
\usepackage{tensor}
\usepackage{amsrefs, eucal}
\usepackage{bbm}
\usepackage{ulem}

\providecommand{\MR}{\relax\ifhmode\unskip\space\fi MR }

\providecommand{\href}[2]{#2}

\theoremstyle{plain}
\newtheorem{thm}{Theorem}[section]
\newtheorem{lem}[thm]{Lemma}
\newtheorem{prop}[thm]{Proposition}
\newtheorem{defn}[thm]{Definition}
\newtheorem{cor}[thm]{Corollary}
\newtheorem{clm}[thm]{Claim}
\newtheorem{ex}[thm]{Example}

\newtheorem*{defin}{Definition}
\theoremstyle{remark}
\newtheorem{rem}[thm]{Remark}

\newcommand{\disp}{\displaystyle}

\DeclareMathOperator{\dist}{dist}
\DeclareMathOperator{\diam}{diam}

\DeclareMathOperator{\osc}{osc}
\DeclareMathOperator{\ess}{ess}

\DeclareMathOperator{\tr}{tr}

\DeclareMathOperator{\supp}{supp}
\DeclareMathOperator{\di}{div}

\DeclareMathOperator{\loc}{loc}

\newcommand{\eps}{\varepsilon}
\newcommand{\vp}{\varphi}

\newcommand{\al}{\alpha}
\newcommand{\be}{\beta}
\newcommand{\ga}{\gamma}
\newcommand{\de}{\delta}

\newcommand{\Ga}{\Gamma}
\newcommand{\te}{\theta}
\newcommand{\la}{\lambda}
\newcommand{\La}{\Lambda}

\newcommand{\Om}{\Omega}

\newcommand{\iny}{\infty}

\newcommand{\su}{\subset}
\newcommand{\LP}{\Delta}
\newcommand{\gr}{\nabla}

\newcommand{\inrn}{\ensuremath{\int_{\R^n}}}
\newcommand{\inom}{\ensuremath{\int_{\Om}}}
\newcommand{\Bp}{{\MC{B}_p}}
\newcommand{\Ai}{{\MC{A}_\iny}}
\newcommand{\Atwi}{{\MC{A}_{2,\iny}}}
\newcommand{\Api}{{\MC{A}_{p,\iny}}}
\newcommand{\sBp}{{\text{B}_p}}
\newcommand{\sAi}{{\text{A}_\iny}}
\newcommand{\um}{\underline{m}}
\newcommand{\ovm}{\overline{m}}
\newcommand{\ud}{\underline{d}}
\newcommand{\ovd}{\overline{d}}
\newcommand{\ND}{\MC{ND}}
\newcommand{\NC}{\MC{NC}}
\newcommand{\Lz}{\mathcal{L}_0}
\newcommand{\LV}{\mathcal{L}_V}
\newcommand{\LGa}{\mathcal{L}_\La}
\newcommand{\BV}{\mathcal{B}_V}

\newcommand{\ur}{\underline{r}}
\newcommand{\BLa}{\mathcal{B}_{\Lambda}}
\newcommand{\WV}{W_{V, 0}^{1, 2} (\Rn)}
\newcommand{\WL}{W_{\Lambda, 0}^{1, 2} (\Rn)}
\newcommand{\Y}{Y_{0}^{1, 2} (\Rn)}
\newcommand{\WLd}{\pr{W_{\Lambda, 0}^{1, 2} (\Rn)}'}
\newcommand{\WVd}{\pr{W_{V, 0}^{1, 2} (\Rn)}'}
\newcommand{\Yd}{\pr{Y_{0}^{1, 2} (\Rn)}'}
\newcommand\RBM{\ensuremath{\mathcal{R}_{\text{BM}}}}

\newcommand{\norm}[1]{\left\| #1\right\|}
\newcommand{\innp}[1]{\left< #1 \right>}
\newcommand{\abs}[1]{\left\vert#1\right\vert}

\newcommand{\set}[1]{\left\{#1\right\}}
\newcommand{\brac}[1]{\left[#1\right]}
\newcommand{\ceil}[1]{\lceil #1\rceil}
\newcommand{\pr}[1]{\left( #1 \right) }
\newcommand{\pb}[1]{\left( #1 \right] }
\newcommand{\brp}[1]{\left[ #1 \right) }

\newcommand{\WT}[1]{\ensuremath{\widetilde{#1}}}
\newcommand{\V}[1]{\ensuremath{\vec{#1}}}

\newcommand{\MC}[1]{\ensuremath{\mathcal{#1}}}
\newcommand{\T}[1]{\ensuremath{\text{#1}}}

\newcommand{\N}{\ensuremath{\mathbb{N}}}

\newcommand{\R}{\ensuremath{\mathbb{R}}}
\newcommand{\Z}{\ensuremath{\mathbb{Z}}}
\newcommand{\C}{\ensuremath{\mathbb{C}}}
\newcommand{\Sd}{\ensuremath{\mathbb{S}^{d-1}}}
\newcommand{\Cd}{\ensuremath{\C^d}}
\newcommand{\Rd}{\ensuremath{\R^d}}
\newcommand{\Rn}{\ensuremath{\R^n}}

\def\XXint#1#2#3{{\setbox0=\hbox{$#1{#2#3}{\int}$}
\vcenter{\hbox{$#2#3$}}\kern-.5\wd0}}

\date{}

\begin{document}

\title{Exponential Decay Estimates for Fundamental Matrices of Generalized Schr\"odinger Systems}
\author[Davey, Isralowitz]{Blair Davey \and  Joshua Isralowitz}

\address{Blair Davey, Department of Mathematical Sciences, Montana State University}
\email{blairdavey@montana.edu}
\thanks{Davey is supported in part by the Simons Foundation Grant 430198 and the National Science Foundation DMS - 2137743.}

\address{Joshua Isralowitz, Department of Mathematics \& Statistics, University at Albany}
\email{jisralowitz@albany.edu}
\thanks{Isralowitz is supported in part by the Simons Foundation Grant 427196.}


\begin{abstract}
In this article, we investigate systems of generalized Schr\"odinger operators and their fundamental matrices.
More specifically, we establish the existence of such fundamental matrices and then prove sharp upper and lower exponential decay estimates for them.
The Schr\"odinger operators that we consider have leading coefficients that are bounded and uniformly elliptic, while the zeroth-order terms are assumed to be nondegenerate and belong to a reverse H\"older class of matrices.
In particular, our operators need not be self-adjoint.
The exponential bounds are governed by the so-called upper and lower Agmon distances associated to the reverse H\"older matrix that serves as the potential function. Furthermore, we thoroughly discuss the relationship between this new reverse H\"{o}lder class of matrices, the more classical matrix $\Api$ class, and the matrix $\Ai$ class introduced in \cite{Dall15}.
\end{abstract}

\maketitle

\hypersetup{linktocpage}
  \setcounter{tocdepth}{1}
\tableofcontents

\section{Introduction}
In this article, we undertake the study of fundamental matrices associated to systems of generalized Schr\"odinger operators; we
establish the existence of such fundamental matrices and we prove sharp upper and lower exponential decay estimates for them.
Our work is strongly motivated by the papers \cite{She99} and \cite{MP19} in which similar exponential decay estimates were established for fundamental solutions associated to Schr\"odinger operators of the form
$$- \LP + \mu,$$
and generalized magnetic Schr\"odinger operators of the form
$$- \pr{\gr - i {\bf a}}^T A \pr{\gr - i {\bf a}} + V,$$
respectively.
In \cite{She99}, $\mu$ is assumed to be a nonnegative Radon measure, whereas in \cite{MP19}, $A$ is bounded and uniformly elliptic, while ${\bf a}$ and $V$ satisfy a number of reverse H\"older conditions.
Here we consider systems of generalized electric Schr\"odinger operators of the form
\begin{equation}
\label{formalEPDE}
\LV = -D_\al\pr{A^{\al \be} D_\be } + V,
\end{equation}
where $A^{\al \be} = A^{\al \be}\pr{x}$, for each $\al, \be \in \set{ 1, \dots, n}$, is a $d \times d$ matrix with bounded measurable coefficients defined on $\R^n$ that satisfies boundedness and ellipticity conditions as described by \eqref{Abd} and \eqref{ellip}, respectively.
Moreover, the zeroth order potential function $V$ is assumed to be a matrix $\Bp$ function.
We say that $V$ is in the matrix $\Bp$ class if and only if $\innp{V \V{e}, \V{e}} := \V{e}^T V \V{e}$ is uniformly a scalar $\sBp$ function for any $\V{e} \in \R^d$.
As such, the operators that we consider in this article fall in between the generality of those that appear in \cite{She99} and \cite{MP19}, but are far more general in the sense that they are for elliptic \textit{systems} of equations.

Many of the ideas in Shen's prior work \cite{She94, She95, She96} have contributed to this article.
In particular, we have built on some of the framework used to prove power decay estimates for fundamental solutions to Schr\"odinger operators $-\LP + V$, where $V$ belongs to the scalar reverse H\"older class $\sBp$, for $p = \iny$ in \cite{She94} and $p \ge \frac n 2$ in \cite{She95}, along with the exponential decay estimates for eigenfunctions of more general magnetic operators as in \cite{She96}.

As in both \cite{She99} and \cite{MP19}, Fefferman-Phong inequalities (see \cite{Fef83}, for example) serve as one of the main tools used to establish both the upper and lower exponential bounds that are presented in this article.
However, since the Fefferman-Phong inequalities that we found in the literature only apply to scalar weights, we state and prove new matrix-weighted Fefferman-Phong inequalities (see Lemma \ref{FPml} and Corollary \ref{FPmlCor}) that are suited to our problem.
To establish our new Fefferman-Phong inequalities, we build upon the notion of an auxiliary function associated to a scalar $\sBp$ function that was introduced by Shen in \cite{She94}.
More specifically, given a matrix function $V \in \Bp$, we introduce a pair of auxiliary functions: the upper and lower auxiliary functions.
(Section \ref{MaxFun} contains precise definitions of these functions and examines their properties.)
Roughly speaking, we can, in some settings, interpret these quantities as the auxiliary functions associated to the largest and smallest eigenvalues of $V$.
The upper and lower auxiliary functions are used to produce two versions of the Fefferman-Phong inequalities.
Using these auxiliary functions, we also define upper and lower Agmon distances (also defined in Section \ref{MaxFun}), which then appear in our lower and upper exponential bounds for the fundamental matrix, respectively.
We remark that the original Agmon distance appeared in \cite{Agm82}, where exponential upper bounds for $N$-body Schr\"odinger operators first appeared.

Given the elliptic operator $\LV$ as in \eqref{formalEPDE} that satisfies a suitable set of conditions, there exists a fundamental matrix function associated to $\LV$, which we denote by $\Ga^V$.
The fundamental matrix generalizes the notion of a fundamental solution to the systems setting; see for example \cite{HK07}, where the authors generalized the results of \cite{GW82} to the systems setting.
To make precise the notion of the fundamental matrix for our systems setting, we rely upon the constructions presented in \cite{DHM18}.
In particular, we define our bilinear form associated to \eqref{formalEPDE}, and introduce a well-tailored Hilbert space that is used to establish the existence of weak solutions to PDEs of the form $\LV \V{u} = \V{f}$.
We then assume that our operator $\LV$ satisfies a natural collection of de Giorgi-Nash-Moser estimates.
This allows us to confirm that the framework from \cite{DHM18} holds for our setting, thereby verifying the existence of the fundamental matrix $\Ga^V$.
Section \ref{FundMat} contains these details.

In Section \ref{ellipExamples}, assuming very mild conditions on $V$, we verify that the class of systems of ``weakly coupled" elliptic operators of the form
\begin{equation}
- \di \pr{A \gr} + V \label{WC}
\end{equation}
satisfy the de Giorgi-Nash-Moser estimates that are mentioned in the previous paragraph (see the remark at the end of Section \ref{ellipExamples} for details).
Consequently, this implies that the fundamental matrices associated to weakly coupled elliptic systems exist and satisfy the required estimates.
In fact, this additionally shows that Green's functions associated to these elliptic systems exist and satisfy weaker interior estimates, though we will not need this fact.
Further, we establish local H\"{o}lder continuity of bounded weak solutions $\V{u}$ to
\begin{equation}
- \di \pr{A \gr \V{u}} + V \V{u} = 0  \label{WCEq}
\end{equation}
under even weaker conditions on $V$.
Specifically, $V$ doesn't have to be positive semidefinite a.e. or even symmetric, see Proposition \ref{HolderContThm} and Remark \ref{HolderRem}.
Finally, although we will not pursue this line of thought in this paper, note that the combination of Proposition \ref{HolderContThm} and Remark \ref{HolderRem} likely leads to new Schauder estimates for  bounded weak solutions $\V{u}$ to \eqref{WCEq}.
We remark that this section on elliptic theory for weakly coupled Schr\"odinger systems could be of independent interest beyond the theory of fundamental matrices.

Assuming the set-up outlined above, we now describe the decay results for the fundamental matrices.
We show that there exists a small constant $\eps > 0$ so that for any $\V{e} \in \Sd$,
\begin{equation}
\label{boundsSummary}
\frac{e^{-\eps \ovd(x, y, V)}}{|x-y|^{n-2}} \lesssim \abs{\innp{\Ga^V (x, y) \V{e}, \V{e}}} \lesssim \frac{ e^{-\eps \ud(x, y, V)}}{|x-y|^{n-2}},
\end{equation}
where $\ovd$ and $\ud$ denote the upper and lower Agmon distances associated to the potential function $V \in \Bp$ (as defined in Section \ref{MaxFun}).
That is, we establish an exponential upper bound for the norm of the fundamental matrix in terms of the lower Agmon distance function, while the fundamental matrix is always exponentially bounded from below in terms of the upper Agmon distance function.
The precise statements of these bounds are described by Theorems \ref{UppBoundThm} and \ref{LowerBoundThm}.
For the upper bound, we assume that $V \in \Bp$ along with a noncommutivity condition $\NC$ that will be made precise in Subsection \ref{NCCondition}.
On the other hand, the lower bound requires the scalar condition $\abs{V} \in \sBp$ and that the operator $\LV$ satisfies some additional properties -- notably a scale-invariant Harnack-type condition.
In fact, the term $\ovd(x, y, V)$ in the lower bound of \eqref{boundsSummary} can be replaced by $d(x, y, \abs{V})$, see Remark \ref{differentDistance}.

Interestingly, \eqref{boundsSummary} can be used to provide a beautiful connection between our upper and lower auxiliary functions and Landscape functions that are similar to those defined in \cite{FM12}.  
Note that this connection was previously found in \cite{Po21} for scalar elliptic operators with nonnegative potentials.  
We will briefly discuss these ideas at the end of Section \ref{LowBds}, see Remark \ref{LandscapeRem}.

To further understand the structure of the bounds stated in \eqref{boundsSummary}, we consider a simple example.
For some scalar functions $0 < v_1 \le v_2 \in \sBp$, define the matrix function
$$V = \begin{bmatrix} v_1 & 0 \\ 0 & v_2 \end{bmatrix}.$$
A straightforward check shows that $V \in \Bp$ and satisfies a nondegeneracy condition that will be introduced below.
Moreover, the upper and lower Agmon distances satisfy $\ud\pr{\cdot, \cdot, V} = d\pr{\cdot, \cdot, v_1}$ and $\ovd\pr{\cdot, \cdot, V} = d\pr{\cdot, \cdot, v_2}$, where $d\pr{x, y, v}$ denotes the standard Agmon distance from $x$ to $y$ that is associated to a scalar function $v \in \sBp$.
We then set
$$\LV = - \LP + V.$$
Since $\V{u}$ satisfies $\LV \V{u} = \V{f}$ if and only if $u_i$ satisfies $- \LP u_i + v_i u_i = f_i$ for $i = 1, 2$, then $\LV$ satisfies the set of elliptic assumptions required for our operator.
Moreover, the fundamental matrix for $\LV$ has a diagonal form given by
$$\Ga^V = \begin{bmatrix} \Ga^{v_1} & 0 \\ 0 & \Ga^{v_2} \end{bmatrix},$$
where each $\Ga^{v_i}$ is the fundamental solution for $-\LP + v_i$.
The results of \cite{She99} and \cite{MP19} show that for $i = 1, 2$, there exists $\eps_i > 0$ so that
$$\frac{e^{-\eps_i d(x, y, v_i)}}{|x-y|^{n-2}} \lesssim \Ga^{v_i}(x, y) \lesssim \frac{ e^{-\eps_i d(x, y, v_i)}}{|x-y|^{n-2}}.$$
Restated, for $i = 1, 2$, we have
\begin{align*}
\frac{e^{-\eps_i d(x, y, v_i)}}{|x-y|^{n-2}} \lesssim \innp{\Ga^V \V{e}_i, \V{e}_i} \lesssim \frac{ e^{-\eps_i d(x, y, v_i)}}{|x-y|^{n-2}},
\end{align*}
where $\set{\V{e}_1, \V{e}_2}$ denotes the standard basis for $\R^2$.
Since $v_1 \le v_2$ implies that $\ud\pr{x, y, V} = d\pr{x, y, v_1} \le d\pr{x, y, v_2} = \ovd\pr{x, y, V}$, then we see that there exists $\eps > 0$ so that for any $\V{e} \in \mathbb{S}^1$,
\begin{align*}
\frac{e^{-\eps \ovd(x, y, V)}}{|x-y|^{n-2}} \lesssim \innp{\Ga^V \V{e}, \V{e}} \lesssim \frac{ e^{-\eps \ud(x, y, V)}}{|x-y|^{n-2}}.
\end{align*}
Compared to estimate \eqref{boundsSummary} that holds for our general operators, this example shows that our results are sharp up to constants.
In particular, the best exponential upper bound we can hope for will involve the lower Agmon distance function, while the best exponential lower bound will involve the upper Agmon distance function.

As stated above, the Fefferman-Phong inequalities are crucial to proving the exponential upper and lower bounds of this article.
The classical Poincar\'e inequality is one of the main tools used to prove the original Fefferman-Phong inequalities.
Since we are working in a matrix setting, we use a new matrix-weighted Poincar\'e inequality.
Interestingly, a fairly straightforward argument based on the scalar Poincar\'e inequality from \cite{She99} can be used to prove this matrix version of the Poincar\'e inequality, which is precisely what is needed to prove the main results described above.

Although the main theorems in this article may be interpreted as vector versions of the results in \cite{She99} and \cite{MP19}, many new ideas (that go well beyond the technicalities of working with systems) were required and developed to establish our results.
We now describe these technical innovations.

First, the theory of matrix weights was not suitably developed for our needs.
For example, we had to appropriately define the matrix reverse H\"older classes, $\Bp$.
And while the scalar versions of $\sBp$ and $\sAi$ have a well-understood and very useful correspondence (namely, a scalar weight $v \in \text{B}_p$ iff $v^p \in \text{A}_\infty$), this relationship was not known in the matrix setting.
In order to arrive at a setting in which we could establish interesting results, we explored the connections between the matrix classes $\Bp$ that we develop, as well as $\Ai$ and $\Api$ that were introduced in \cite{Dall15} and \cite{NT96}, \cite{Vol97}, respectively.
The matrix classes are introduced in Section \ref{MWeights}, and many additional relationships (including a matrix version of the previously mentioned correspondence between A${}_\infty$ and B${}_p$) are explored in Appendices \ref{Examples} and \ref{AiApp}.

Given that we are working in a matrix setting, there was no reason to expect to work with a single auxiliary function.
Instead, we anticipated that our auxiliary functions would either be matrix-valued, or that we would have multiple scalar-valued ``directional" Agmon functions.
We first tried to work with a matrix-valued auxiliary function based on the spectral decomposition of the matrix weight.
However, since this set-up assumed that all eigenvalues belong to $\sBp$, and it is unclear when that assumption holds, we decided that this approach was overly restrictive.
As such, we decided to work with a pair of scalar-valued auxiliary functions that capture the upper and lower behaviors of the matrix weight.
Once these functions were defined and understood, we could associate Agmon distances to them in the usual manner.
These notions are introduced in Section \ref{MaxFun}.

Another virtue of this article is the verification of elliptic theory for a class of elliptic \textit{systems} of the form \eqref{WC}.
By following the ideas of Caffarelli from \cite{Caf82}, we prove that under standard assumptions on the potential matrix $V$, the solutions to these systems are bounded and H\"older continuous.
That is, instead of simply assuming that our operators are chosen to satisfy the de Giorgi-Nash-Moser estimates, we prove in Section \ref{ellipExamples} that these results hold for this class of examples.
In particular, we can then fully justify the existence of their corresponding fundamental matrices.
To the best of our knowledge, these ideas from \cite{Caf82} have not been used in the linear setting.

A final challenge that we overcame has to do with the fact there are two distinct and unrelated Agmon distance functions associated to the matrix weight $V$.
In particular, because these distance functions aren't related, we had to modify the scalar approach to proving exponential upper and lower bounds for the fundamental matrix associated to the operator $\LV := \Lz + V$.
The first bound that we prove for the fundamental matrix is an exponential upper bound in terms of the lower Agmon distance.
In the scalar setting, this upper bound is then used to prove the exponential lower bound.
But for us, the best exponential lower bound that we can expect is in terms of the \textit{upper} Agmon distance.
If we follow the scalar proof, we are led to a standstill since the upper and lower Agmon distances of $V$ aren't related.
We overcame this issue by introducing $\LGa := \Lz + \abs{V} I_d$, an elliptic operator whose upper and lower Agmon distances agree and are equal to the upper Agmon distance associated to $\LV$.
In particular, the upper bound for the fundamental matrix of $\LGa$ depends on the \textit{upper} Agmon distance of $V$.
This observation, along with a clever trick, allows us to prove the required exponential lower bound.
These ideas are described in Section \ref{LowBds}, using results from the end of Section \ref{UpBds}.

The motivating reasons for studying \textit{systems} of elliptic equations are threefold, as we now describe.

First, real-valued systems may be used to describe complex-valued equations and systems.
To illuminate this point, we consider a simple example.
Let $\Om \su \R^n$ be an open set and consider the complex-valued Schr\"odinger operator given by
$$L_x = - \di \pr{c \gr } + x,$$
where $c = \pr{c^{\al \be}}_{\al, \be = 1}^n$ denotes the complex-valued coefficient matrix and $x$ denotes the complex-valued potential function.
That is, for each $\al, \be = 1, \ldots, n$,
$$c^{\al \be} = a^{\al \be} + i b^{\al \be},$$
where both $a^{\al \be}$ and $b^{\al \be}$ are $\R$-valued functions defined on $\Om \su \R^n$, while
$$x = v + i w,$$
where both $v$ and $w$ are $\R$-valued functions defined on $\Om$.
To translate our complex operator into the language of systems, we define
$$A = \begin{bmatrix} a & - b \\ b & a \end{bmatrix}, \quad \quad V = \begin{bmatrix} v & -w \\ w & v \end{bmatrix}.$$
That is, each of the entries of $A$ is an $n \times n$ matrix function:
$$A_{11} = A_{22} = a, \quad  A_{12} = -b, \quad A_{21} = b,$$
while each of the entries of $V$ is a scalar function:
$$V_{11} = V_{22} = v, \quad  V_{12} = -w, \quad V_{21} = w.$$
Then we define the systems operator
\begin{equation}
\label{LVDef0}
\LV = -D_\al\pr{A^{\al \be} D_\be } + V.
\end{equation}
If $u = u_1 + i u_2$ is a $\C$-valued solution to $L_x u = 0$, where both $u_1$ and $u_2$ are $\R$-valued, then $\V{u} = \begin{bmatrix} u_1 \\ u_2 \end{bmatrix}$ is an $\R^2$-valued vector solution to the elliptic system described by $\LV \V{u} = \V{0}$.

This construction also works with complex systems.
Let $C = A + i B$, where each $A^{\al\be}$ and $B^{\al\be}$ is $\R^{d \times d}$-valued, for $\al, \be \in \set{1, \ldots, n}$.
If we take $X = V + i W$, where $V$ and $W$ are $\R^{d \times d}$-valued, then the operator
$$L_X = - D_\al \pr{C^{\al \be} D_\be } + X$$
describes a complex-valued system of $d$ equations.
Following the construction above, we get a real-valued system of $2d$ equations of the form described by \eqref{LVDef0}, where now
$$A = \begin{bmatrix} A & - B \\ B & A \end{bmatrix}, \qquad
V = \begin{bmatrix} V & - W \\ W & V \end{bmatrix}.$$
In particular, if $X$ is assumed to be a $d \times d$ Hermitian matrix (meaning that $X = X^*$), then $V$ is a $2d \times 2d$ real, symmetric matrix.
This shows that studying systems of equations with Hermitian potential functions (as is often done in mathematical physics) is equivalent to studying real-valued systems with symmetric potentials, as we do in this article.
Moreover, $X$ is positive (semi)definite iff $V$ is positive (semi)definite.
In conclusion, because there is much interest in complex-valued elliptic operators, we believe that it is very meaningful to study real-valued elliptic systems of equations.

Our second motivation comes from physics and molecular dynamics.
Schr\"{o}dinger operators with complex Hermitian matrix potentials $V$ naturally arise when one seeks to solve the Schr\"{o}dinger eigenvalue problem for a molecule with Coulomb interactions between electrons and nuclei.
More precisely,  it is sometimes useful to convert the eigenvalue problem associated to the above (scalar) Schr\"{o}dinger operator into a simpler eigenvalue problem associated to a Schr\"{o}dinger operator with a matrix potential and Laplacian taken with respect to only the nuclear coordinates.
See the classical references \cite[p. 335 - 342]{Tan07} and \cite[p. 148 - 153]{WC04} for more details.
Note that this potential is self-adjoint and is often assumed to have eigenvalues that are bounded below, or even approaching infinity as the nuclear variable approaches infinity.
See for example \cite{BHKPSS15,KPSS18, KPSS18b, PSS19,HS20}, where various molecular dynamical approximation errors and asymptotics are computed utilizing the matrix Schr\"{o}dinger eigenvalue problem stated above as their starting point.

With this in mind, we are hopeful that the results in this paper might find applications to the mathematical theory of molecular dynamics.
Moreover, it would be interesting to know whether the results of Sections \ref{FundMat} and \ref{ellipExamples} are true for ``Schr\"{o}dinger operators" with a matrix potential and nonzero first order terms.
Note that such operators also appear naturally when one solves the same Schr\"{o}dinger eigenvalue problem for a molecule with Coulomb interactions between electrons and nuclei, but only partially performs the ``conversion" described in the previous paragraph.
We again refer the reader to \cite[p. 335 - 342]{Tan07} for additional details. 
It would also be  interesting to determine whether defining a landscape function in terms of a Green's function of a Schr\"{o}dinger operator with a matrix potential would provide useful pointwise eigenfunction bounds.

Third, studying elliptic systems of PDEs with a symmetric nonnegative matrix potential provides a beautiful connection between the theory of matrix-weighted norm inequalities and the theory of elliptic PDEs.
In particular, classical scalar reverse H\"{o}lder and Muckenhoupt A${}_\infty$ assumptions on the scalar potential of elliptic equations are very often assumed (see \cite{She94, She95, She96} for example).
On the other hand, while various matrix versions of these conditions have appeared in the literature (see for example \cite{NT96,Vol97,Gol03,Dall15,Ros16}), the connections between elliptic systems of PDEs with a symmetric  nonnegative matrix potential and the theory of matrix-weighted norm inequalities is a mostly unexplored area (with the exception of \cite{Dall15}, which provides a Shubin-Maz'ya type sufficiency condition for the discreteness of the spectrum of a Schr\"{o}dinger operator with complex Hermitian positive-semidefinite matrix potential $V$ on $\Rn$).
This project led to the systematic development of the theory of matrix reverse H\"older classes, $\Bp$, as well as an examination of the connections between $\Bp$, $\Ai$, and $\Atwi$.
By going beyond the ideas from \cite{Dall15}, \cite{NT96}, \cite{Vol97}, we carefully study $\Ai$ and prove that $\Ai = \Atwi$.

Unless otherwise stated, we assume that our $d \times d$ matrix weights (which play the role of the potential in our operators) are real-valued, symmetric, and positive semidefinite.
As described above, real symmetric potentials are equivalent to complex Hermitian potentials through a ``dimension doubling" process.
In fact, because of this equivalence, our results can be compared with those in mathematical physics, where systems with complex, Hermitian matrix potentials are considered.
To reiterate, we assume throughout the body of the article that $V$ is real-valued and symmetric.
However, in Appendix \ref{AiApp}, we follow the matrix weights community convention and assume that our matrix weights are complex-valued and Hermitian.

\subsection{Organization of the article}

The next three sections are devoted to matrix weight preliminaries with the goal of stating and proving our matrix version of the Fefferman-Phong inequality, Lemma \ref{FPml}.
In Section \ref{MWeights}, we present the different classes of matrix weights that we work with throughout this article, including the aforementioned matrix reverse H\"{o}lder condition $\Bp$ for $p > 1$ and the (non-Muckenhoupt) noncommutativity condition $\NC$ which will be crucial to the proof of Lemma \ref{FPml}.
These two classes will be discussed in relationship to the existing matrix weight literature in Appendix \ref{AiApp}.
Section \ref{MaxFun} introduces the auxiliary functions and their associated Agmon distance functions.
The Fefferman-Phong inequalities are then stated and proved in Section \ref{FPI}.
Section \ref{FPI} also contains the Poincar\'e inequality that is used to prove one of our new Fefferman-Phong inequalities.
The following three sections are concerned with elliptic theory.
Section \ref{EllOp} introduces the elliptic systems of the form \eqref{formalEPDE} discussed earlier.
The fundamental matrices associated to these operators are discussed in Section \ref{FundMat}.
In Section \ref{ellipExamples}, we show that the elliptic systems of the form \eqref{WC} satisfy the assumptions from Section \ref{FundMat}.
The last two sections, Section \ref{UpBds} and Section \ref{LowBds}, are respectively concerned with the upper and lower exponential bounds for our fundamental matrices. 
Futher, we discuss the aforementioned connection between our upper and lower auxiliary functions and Landscape functions at the end of Section \ref{LowBds}.

Finally, in our first two appendices, we state and prove a number of results related to the theory of matrix weights that are interesting in their own right, but are not needed for the proofs of our main results.
In Appendix \ref{Examples}, we explore the noncommutativity class $\NC$ in depth, providing examples and comparing it to our other matrix classes.
In Appendix \ref{AiApp}, we systematically develop the theory of the various matrix classes that are introduced in Section \ref{MWeights}.
In particular, we provide a comprehensive discussion of the matrix $\Bp$ class and characterize this class of matrix weights in terms of the more classical matrix weight class $\Api$  from \cite{NT96, Vol97}.
This discussion nicely complements a related matrix weight characterization from \cite[Corollary $3.8$]{Ros16}.
We also discuss how the matrix $\Ai$ class introduced in \cite{Dall15} relates to the other matrix weight conditions discussed in this paper.
In particular, we establish that $\Ai = \Atwi$.
Further, we provide a new characterization of the matrix $\Atwi$ condition in terms of a new reverse Brunn-Minkowski type inequality.
We hope that Appendix \ref{AiApp} will appeal to the reader who is interested in the theory of matrix-weighted norm inequalities in their own right.
The last appendix contains the proofs of technical results that we skipped in the body.

We have attempted to make this article as self-contained as possible, particularly for the reader who is not an expert in elliptic theory or matrix weights.
In Appendix \ref{AiApp}, we have not assumed any prior knowledge of matrix weights.
As such, we hope that this section can serve as a reference for the $\Ai$ theory of matrix weights.

\subsection{Notation}

As is standard, we use $C$, $c$, etc. to denote constants that may change from line to line.
We may use the notation $C(n, p)$ to indicate that the constant depends on $n$ and $p$, for example.
The notation $a \lesssim b$ means that there exists a constant $c > 0$ so that $a \le c b$.
If $c = c\pr{d, p}$, for example, then we may write $a \lesssim_{(d, p)} b$.
We say that $a \simeq b$ if both $a \lesssim b$ and $b \lesssim a$ with dependence denoted analogously.

Let $\innp{\cdot, \cdot}_d : \R^d \times \R^d \to \R$ denote the standard Euclidean inner product on $d$-dimensional space.
When the dimension of the underlying space is understood from the context, we may drop the subscript and simply write $\innp{\cdot, \cdot}$.
For a vector $\V{v} \in \R^d$, its scalar length is $\abs{\V{v}} = \innp{\V{v}, \V{v}}^{\frac 1 2}$.
The sphere in $d$-dimensional space is $\Sd = \set{\V{v} \in \R^d : \abs{\V{v}} = 1}$.

For a $d \times d$ real-valued matrix $A$, we use the $2$-norm, which is given by
$$\abs{A} = \abs{A}_2 = \sup \set{\abs{Ax} : x \in \Sd} = \sup\set{\innp{Ax, y} : x, y \in \Sd}.$$
Alternatively, $\abs{A}$ is equal to its largest singular value, the square root of the largest eigenvalue of $AA^T$.
For symmetric positive semidefinite $d \times d$ matrices $A$ and $B$, we say that $A \le B$ if $\innp{A \vec e, \vec e} \le \innp{B \vec e, \vec e}$ for every $\vec e \in \R^d$.
Note that both $\abs{\V{v}}$ and $\abs{A}$ are scalar quantities.

If $A$ is symmetric and positive semidefinite, then $\abs{A}$ is equal to $\la$, the largest eigenvalue of $A$.
Let $\V{v} \in \Sd$ denote the eigenvector associated to $\la$ and let $\set{\V{e}_i}_{i=1}^d$ denote the standard basis of $\R^d$.
Observe that $\disp \V{v} = \sum_{i=1}^d \innp{\V{v}, \V{e}_i} \V{e}_i$ and for each $j$, $\disp \innp{\V{v}, \V{e}_j}^2 \le \sum_{i=1}^d \innp{\V{v}, \V{e}_i}^2 = 1$.
Then, since $A^{\frac 1 2}$ is well-defined, an application of Cauchy-Schwarz shows that
\begin{equation}
\label{normProp}
\begin{aligned}
\abs{A}
&= \la
= \innp{A\V{v}, \V{v}}
\le d \sum_{i=1}^d \innp{ A\V{e}_i,\V{e}_i}.
\end{aligned}
\end{equation}

Let $\Om \su \R^n$ and $p \in \brp{1,\iny}$.
For any $d$-vector $\V{v}$, we write $\disp \norm{\V{v}}_{L^p(\Om)} = \pr{\int_{\Om} \abs{\V{v}\pr{x}}^p dx}^{\frac 1 p}$.
Similarly, for any $d \times d$ matrix $A$, we use the notation $\disp \norm{A}_{L^p(\Om)} = \pr{\int_{\Om} \abs{A\pr{x}}^p dx}^{\frac 1 p}$.
When $p = \iny$, we write $\disp \norm{\V{v}}_{L^\iny(\Om)} = \ess \sup_{x \in \Om} \abs{\V{v}\pr{x}}$ and $\disp \norm{A}_{L^\iny(\Om)} = \ess \sup_{x \in \Om} \abs{A\pr{x}}$.
We say that a vector $\V{v}$ belongs to $L^p\pr{\Om}$ iff the scalar function $\abs{\V{v}}$ belongs to $L^p\pr{\Om}$.
Similarly, for a matrix function $A$ defined on $\Om$, $A \in L^p\pr{\Om}$ iff $\abs{A} \in L^p\pr{\Om}$.
In summary, we use the notation $\abs{\cdot}$ to denote norms of vectors and matrices, while we use the notations $\norm{\cdot}_p$, $\norm{\cdot}_{L^p}$, or $\norm{\cdot}_{L^p\pr{\Om}}$ to denote $L^p$-space norms.

We let $C^\iny_c(\Om)$ denote the set of all infinitely differentiable functions with compact support in $\Om$.
If $\V{\vp} : \Om \to \R^d$ is a vector-valued function for which each component function $\vp_i \in C^\iny_c(\Om)$, then $\V{\vp} \in C^\iny_c(\Om)$.

For $x \in \R^n$ and $r > 0$, we use the notation $B\pr{x, r}$ to denote a ball of radius $r > 0$ centered at $x \in \R^n$.
We let $Q(x, r)$ be the ball of radius $r$ and center $x \in \R^n$ in the $\ell^\infty$ norm on $\R^n$ (i.e. the cube with side length $2r$ and center $x$).
We write $Q$ to denote a generic cube.
We use the notation $\ell$ to denote the sidelength of a cube.
That is, $\ell\pr{Q\pr{x, r}} = 2r$.

We will assume throughout that $n \ge 3$ and $d \ge 2$.
In general, $1 < p < \infty$, but we may further specify the range of $p$ as we go.

\subsection*{Acknowledgements.}
The authors would like to thank Svitlana Mayboroda and Bruno Poggi for interesting discussions and useful feedback.

\section{Matrix Classes}
\label{MWeights}

Within this section, we define the classes of matrix functions that we work with, then we collect a number of observations about them.

\subsection{Reverse H\"older Matrices}

Recall that for a nonnegative scalar-valued function $v$, we say that $v$ belongs to the reverse H\"older class $\sBp$ if $v \in L^p_{\loc}\pr{\R^n}$ and there exists a constant $C_v$ so that for every cube $Q$ in $\R^n$,
\begin{equation}
\label{BpDefOne}
\pr{\fint_Q \brac{v\pr{x}}^p dx}^{\frac 1 p} \le C_v \fint_Q v\pr{x} dx.
\end{equation}

Let $V$ be a $d \times d$ {\bf matrix weight} on $\R^n$.
That is, $V$ is a $d \times d$ real-valued, symmetric, positive semidefinite matrix.
For such matrices, we define $\Bp$, the class of reverse H\"older matrices, via quadratic forms as follows.

\begin{defn}[Matrix $\Bp$]
For matrix weights, we say that $V$ belongs to the class of {\bf reverse H\"older matrices}, $V \in \Bp$, if $V \in L^p_{\loc}\pr{\R^n}$ and there exists a constant $C_V$ so that for every cube $Q$ in $\R^n$ and every $\vec e \in \R^d$ (or $\Sd$),
\begin{equation}
\label{BpDefTwo}
\pr{\fint_Q \innp{V\pr{x} \vec e, \vec e}^p dx}^{\frac 1 p}
\le C_V \innp{\pr{ \fint_Q V\pr{x} dx } \vec e, \vec e}.
\end{equation}
This constant is independent of $Q$ and $\V{e}$, so that the inequality holds uniformly in $Q$ and $\V{e}$.
We call $C_V$ the (uniform) $\Bp$ constant of $V$.
Note that $C_V$ depends on $V$ as well as $p$; to indicate this, we may use the notation $C_{V,p}$.
\end{defn}

\begin{rem}
This definition may be generalized as follows.
For any $q, p > 1$, we say that $V \in \mathcal{B}_{q, p}$ if there exists a constant $C_V$ so that for every cube $Q$ in $\R^n$ and every $\vec e \in \R^d$, it holds that
\begin{equation*}
\pr{\fint_Q \abs{V\pr{x}^{\frac 1 q} \vec e}^{qp} dx}^{\frac 1 p}
\le C_V \fint_Q \abs{V\pr{x}^{\frac 1 q} \vec e}^q dx.
\end{equation*}
Notice that $\Bp = \mathcal{B}_{2, p}$.
This matrix class is unexplored, but its theory is probably similar to what we develop here for $\Bp$.
One might expect a relationship between $\mathcal{B}_{q, p}$ and the $q$-Laplacian.
\end{rem}

Now we collect some observations about such matrix functions.
The first result regards the norm of a matrix $\Bp$ function.

\begin{lem}[Largest eigenvalue is scalar $\sBp$]
\label{normVBp}
If $V \in \Bp$, then $\abs{V} \in \sBp$ with $C_{\abs{V}} \lesssim_{(d, p)} C_V$.
\end{lem}

\begin{proof}
Let $\V{e}_1, \ldots, \V{e}_d$ denote the standard basis for $\Rd$.
Using that $\disp \abs{V} \le d \sum_{i=1}^d \innp{V \V{e}_i, \V{e}_i}$ (as explained in the notation section, see \eqref{normProp}) combined with the H\"older and Minkowski inequalities shows that
\begin{align*}
\pr{\fint_Q \abs{V(x)}^p \, dx}^\frac{1}{p}
&\le \brac{\fint_Q \pr{ d \sum_{j=1}^d \innp{V\pr{x} \V{e}_j, \V{e}_j}}^p \, dx}^\frac{1}{p}
\le d^{2 - \frac 1 p} \sum_{j = 1}^d \pr{\fint_Q \innp{V(x) \V{e}_j, \V{e}_j}^p \, dx }^\frac{1}{p} \\
&\le d^{2 - \frac 1 p} C_V  \sum_{j = 1}^d \innp{\pr{\fint_Q V(x) dx } \V{e}_j, \V{e}_j} 
\le d^{3 - \frac 1 p} C_V \fint_Q \abs{V(x)} \, dx,
\end{align*}
where we have used the reverse H\"older inequality \eqref{BpDefTwo} in the fourth step.
\end{proof}

\begin{lem}[Gehring's Lemma]
\label{GehringLemma}
If $V \in \Bp$, then there exists $\eps\pr{p, C_V} > 0$ so that $V \in \mathcal{B}_{p+\eps}$.
In particular, $V \in \mathcal{B}_q$ for all $q \in \brac{1, p + \eps}$.
Moreover, if $q \le s$, then $C_{V, q} \le C_{V, s}$.
\end{lem}

\begin{proof}
Since $\innp{V(x)\V{e}, \V{e}}$ is a scalar $\sBp$ weight (with $\sBp$ constant uniform in $\V{e} \in \R^d$), then it follows from the proof of Gehring's Lemma (see for example the dyadic proof in \cite{Per01}) that $V \in \Bp$ implies that there exists $\epsilon > 0$ such that $V \in \MC{B}_{p + \epsilon}$.
Let $q \le p + \eps$, $\V{e} \in \R^d$.
Then by H\"older's inequality,
\begin{align*}
\pr{\fint_Q \innp{V\pr{x} \vec e, \vec e}^q dx}^{\frac 1 q}
&\le \frac{1}{\abs{Q}^{\frac 1 q}} \brac{\pr{\int_Q \innp{V\pr{x} \vec e, \vec e}^{p+\eps} dx}^{\frac q {p+\eps}} \abs{Q}^{1 - \frac q {p+\eps}}}^{\frac 1 q} \\
&=\pr{\fint_Q \innp{V\pr{x} \vec e, \vec e}^{p+\eps} dx}^{\frac 1 {p+\eps}}
\le C_{V, p+\eps} \innp{\pr{ \fint_Q V\pr{x} dx } \vec e, \vec e},
\end{align*}
showing that $V \in \mathcal{B}_q$.
If $q \le s \le p + \eps$, then the same argument holds with $C_{V,s}$ in place of $C_{V, p+\eps}$.
\end{proof}

Now we introduce an averaged version of $V$ that will be extensively used in our arguments.

\begin{defn}[Averaged matrix]
Let $V$ be a function, $x \in \R^n$, $r > 0$.
We define the {\bf averaged matrix} as
\begin{align}
\Psi\pr{x, r; V} = \frac{1}{r^{n-2}} \int_{Q\pr{x,r}} V\pr{y} dy.
\label{eqB.2}
\end{align}
\end{defn}

These averages have a controlled growth.

\begin{lem}[Controlled growth, cf. Lemma 1.2 in \cite{She95}]
\label{BasicShenLem}
If $V \in \Bp$, then for any $0 < r < R  < \iny$,
\begin{align*}
\Psi\pr{x, r; V} \le C_V \pr{\frac{r}{R}}^{2 - \frac{n}{p}} \Psi\pr{x, R; V},
\end{align*}
\label{lB.1}
where $C_V$ is the uniform $\Bp$ constant for $V$.
\end{lem}

\begin{proof}
Let $0 < r < R$.
Then for any $\vec e \in \R^n$, applications of the H\"older inequality and the reverse H\"older inequality described by \eqref{BpDefTwo} show that
\begin{align*}
\innp{\pr{\fint_{Q\pr{x, r}} V\pr{y} dy} \vec e, \vec e}
&
\le \pr{\frac 1 {\abs{Q\pr{x, r}}} \int_{Q\pr{x, r}} \innp{V\pr{y} \vec e, \vec e}^p dy}^{\frac 1 p}
\le \pr{\frac {\abs{Q\pr{x, R}}} {\abs{Q\pr{x, r}}} \fint_{Q\pr{x, R}} \innp{V\pr{y} \vec e, \vec e}^p dy}^{\frac 1 p} \\
&= \pr{\frac R r}^{\frac n p} \pr{ \fint_{Q\pr{x, R}} \innp{V\pr{y} \vec e, \vec e}^p dy}^{\frac 1 p}
\le C_V \pr{\frac R r}^{\frac n p} \innp{\pr{ \fint_{Q\pr{x, R}} V\pr{y} dy } \vec e, \vec e}.
\end{align*}
As $\vec{e} \in \R^n$ was arbitrary, then it follows that $\disp \fint_{Q\pr{x, r}} V\pr{y} dy \le C_V \pr{\frac R r}^{\frac n p}  \fint_{Q\pr{x, R}} V\pr{y} dy$, which leads to the conclusion of the lemma.
\end{proof}

Furthermore, the $\Bp$ matrices serve as doubling measures.

\begin{lem}[Doubling result]
\label{Vdbl}
If $V \in \Bp$, then $V$ is a doubling measure.
That is, there exists a doubling constant $\ga = \ga\pr{n, p, C_V} > 0$ so that for every $x \in \R^n$ and every $r > 0$,
\begin{align*}
\int_{Q\pr{x, 2r}} V\pr{y} dy \le \ga \int_{Q\pr{x, r}} V\pr{y} dy.
\end{align*}
\end{lem}

\begin{proof}
Since each $\innp{V \V{e}, \V{e}}$ belongs to $\sBp$, then by the scalar result, $\innp{V \V{e}, \V{e}}$ defines a doubling measure.
Moreover, since the $\sBp$ constant is independent of $\V{e} \in \Sd$, then so too is the doubling constant associated to each measure defined by $\innp{V \V{e}, \V{e}}$.
It follows that $V$ defines a doubling measure.
\end{proof}

\subsection{Nondegenerate matrices}

Next, we define a very natural class of nondegenerate matrices.

\begin{defn}[Nondegeneracy class]
We say that $V$ belongs to the {\bf nondegeneracy class}, $V \in \ND$, if $V$ is a matrix weight that satisfies the following very mild nondegeneracy condition:
For any measurable $|E| > 0$, we have (in the usual sense of semidefinite matrices) that
\begin{equation}
V(E) := \int_E V(y) \, dy  > 0.
\label{NDCond}
\end{equation}
\end{defn}

First we give an example of a matrix function in $\Bp$ but not in $\ND$.

\begin{ex}[$\Bp \setminus \ND$ is nonempty]
\label{degenMatrix}
Take $v: \R^n \to \R$ in $\sBp$ and define
$$V = \brac{\begin{array}{cccc} v & 0 & \ldots & 0 \\ 0 & 0 & \ldots & 0 \\ \vdots & \vdots & \ddots & \vdots \\ 0 & 0 & \ldots & 0 \end{array}}.$$
It is clear that $V \in \Bp$.
However, since $V$ and its averages all have zero eigenvalues, then $V \notin \ND$.
\end{ex}

Now we produce a number of examples in both $\Bp$ and $\ND$.

\begin{ex}[$\Bp \cap \ND$ polynomial matrices]
\label{polyEx}
Let $V$ be a polynomial matrix.
\begin{itemize}
\item[(a)] If $V$ is symmetric, nontrivial along the diagonal, and positive semidefinite, then $V$ satisfies \eqref{NDCond}.
It follows from Corollary \ref{ExampleCor} that $V \in \Bp$ as well.
\item[(b)] If for every $\V{e} \in \R^d$, there exists $i \in \set{1, \ldots, d}$ so that $\disp \sum_{j = 1}^d V_{ij}e_j \ne 0$, then $V^T V$ satisfies \eqref{NDCond}.
Since $V^T V$ is symmetric and polynomial, then $V^T V \in \ND \cap \Bp$.
A similar condition shows that $V V^T \in \ND \cap \Bp$ as well.
\end{itemize}
\end{ex}

As we will see below, the nondegeneracy condition described by \eqref{NDCond} facilitates the introduction of one of our key tools.
However, there are also practical reasons to avoid working with matrices that aren't nondegenerate.
For example, consider a matrix-valued Schr\"odinger operator of the form $- \LP + V$, where $V$ is as given in Example \ref{degenMatrix}.
The fundamental matrix of this operator is diagonal with only the first entry exhibiting decay, while all other diagonal entries contain the fundamental solution for $\LP$.
In particular, since the norm of this fundamental matrix doesn't exhibit exponential decay, we believe that the assumption of nondegeneracy is very natural for our setting.

\subsection{Noncommutativity condition}
\label{NCCondition}

As we'll see below, the single assumption that $V \in \Bp$ will not suffice for our needs, and we'll impose additional conditions on $V$.
To define the noncommutativity condition that we use, we need to introduce the lower auxiliary function associated to $V \in \Bp \cap \ND$.

If $V \in \ND$, then by \eqref{eqB.2} and \eqref{NDCond}, for each $x \in \R^n$ and $r > 0$, we have $\Psi\pr{x, r; V} > 0$.
If $V \in \Bp$ and $p > \frac{n}{2}$, then the power $2 - \frac n p > 0$ and it follows from Lemma \ref{lB.1} that
\begin{equation}
\label{eqB.3}
\begin{aligned}
& \lim_{r \to 0^+} \innp{\Psi\pr{x, r; V}\V{e}, \V{e}}  = 0 \; \text{ for any } \V{e} \in \Rd,  \\
& \lim_{R \to \iny} \min_{\V{e} \in \Sd} \innp{\Psi\pr{x, R; V} \V{e}, \V{e}}  = \iny.
\end{aligned}
\end{equation}

These observations allows us to make the following definition of $\um$, the lower auxiliary function.

\begin{defin}[Lower auxiliary function]
Let $V \in \Bp \cap \ND$ for some $p > \frac n 2$.
We define the {\bf lower auxiliary function} $ \um\pr{\cdot, V} : \Rn \rightarrow  (0, \infty)$ as follows:
\begin{align*}
\frac{1}{\um\pr{x, V}} = \sup_{r > 0} \set{ r : \min_{\V{e} \in \Sd} \innp{\psi\pr{x,r; V} \V{e}, \V{e}} \le 1}.
\end{align*}
\end{defin}

We will investigate this function and others in much more detail within Section \ref{MaxFun}.
For now, we use $\um$ to define our next matrix class.

\begin{defn}[Noncommutativity class]
If $V \in \Bp \cap \ND$, then we say that $V$ belongs to the {\bf noncommutativity class}, $V \in \NC$, if there exists $N_V > 0$ so that for every $x \in \R^n$ and every $\V{e} \in \R^d$,
\begin{equation}
N_V \abs{\V{e}}^2 \le \int_Q \innp{V^\frac12 (y) V(Q)^{-1} V^\frac12 (y) \V{e}, \V{e}} \, dy,
\label{NCCond}
\end{equation}
where $Q = Q\pr{x, \frac 1 {\um(x, V)}}$.
\end{defn}

For most of our applications, we consider $\NC$ as a subset of $\Bp \cap \ND$ in order to make sense of $\um(\cdot, V)$.
However, if $V \notin \Bp \cap \ND$ or $\um\pr{\cdot, V}$ is not well-defined, we say that $V \in \NC$ if \eqref{NCCond} holds for every cube $Q \su \R^n$.

To show that this class of matrices is meaningful, we provide a non-example.

\begin{ex}[$\NC$ is a proper subset of $\Bp \cap \ND$]
\label{notNCEx}
Define $V : \R^n \to \R^{2 \times 2}$ by
$$V(x) = \begin{bmatrix}1 & \abs{x}^2 \\ \abs{x}^2 & \abs{x}^4 \end{bmatrix} = \begin{bmatrix}1 & x_1^2 + \ldots x_n^2 \\ x_1^2 + \ldots x_n^2 & \pr{x_1^2 + \ldots x_n^2}^2 \end{bmatrix}.$$
By Example \ref{polyEx}, $V \in \Bp \cap \ND$.
However, as shown in Appendix \ref{Examples}, $V \notin \NC$.
\end{ex}

In Section \ref{UpBds}, we prove one of our main results: an upper exponential decay estimate for  the fundamental matrix of the elliptic operator $\LV$, where $V  \in \Bp \cap\ND \cap\NC$.
A further discussion of these matrix classes and their relationships is available in Appendix \ref{Examples}.

\subsection{Stronger conditions}

To finish our discussion of matrix weights, we introduce some closely related and more well-known classes of matrices.
Note that these assumptions are stronger and more readily checkable.

\begin{defn}[$\Ai$ matrices]
We say that $V$ belongs to the {\bf A-infinity class of matrices}, $V \in \Ai$, if for any $\epsilon > 0$, there exists $\delta > 0$ so that for every cube $Q$,
\begin{equation}
\label{Ainf}
\abs{\set{x \in Q: V\pr{x   } \geq \delta \fint_Q V\pr{y} dy}} \geq (1-\epsilon) |Q|.
\end{equation}
\end{defn}

This class of matrix weights was first introduced in \cite{Dall15}, where the author proved a Shubin-Maz'ya type sufficiency condition for the discreteness of the spectrum of a Schrodinger operator $- \Delta + V$, where $V \in \Ai$.
Interestingly, and somewhat surprisingly, we show in the Appendix \ref{AiApp} that the condition $V \in \Ai$ is equivalent to $V \in \MC{A}_{2, \infty}$.
The class $\MC{A}_{2, \infty}$ is the readily checkable class of matrix weights introduced in \cite{NT96,Vol97}, which we now define.

\begin{defn}[$\mathcal{A}_{2, \infty}$ matrices]
We say $V \in \mathcal{A}_{2, \infty}$, if there exists $A_V > 0$ so that for every cube $Q$, we have
\begin{equation}
\det \pr{\fint_Q V(x) \, dx } \le A_V \exp \pr{\fint_Q \ln \det V(x) \, dx }.
\label{AtwoInf}
\end{equation}
\end{defn}

We briefly discuss the relationship between $\Bp$ and $\Ai$.
First, we have the following application of Gehring's lemma.

\begin{lem}[$\Ai \su \mathcal{B}_q$]
If $V \in \Ai$, then $V \in \mathcal{B}_q$ for some $q > 1$.
\end{lem}

\begin{proof}
Since $\innp{V \V{e}, \V{e}} \in \sAi$ uniformly in $\V{e} \in \Sd$, then by \cite[Lemma 7.2.2]{Gra14},
there exists $q > 1$ so that $\innp{V \V{e}, \V{e}} \in B_q$ uniformly in $\V{e} \in \Sd$.
In particular, $V \in \mathcal{B}_q$, as required.
\end{proof}

On the other hand, if $V \in \Bp$, there is no reason to expect that $V \in \Ai$.
Namely, if $V \in \Bp$, then for each $\V{e} \in \R^d$, $\innp{V(x) \V{e}, \V{e}}$ is a scalar $\sBp$ function.
This means that $\innp{V(x) \V{e}, \V{e}}$ is a scalar $\sAi$ function.
Therefore, if $V \in \Bp$, then for each \textit{fixed} $\V{e}$ and any $\epsilon > 0$, there exists $\delta = \de\pr{\V{e}} > 0$ so that for every cube $Q$,
\begin{equation}
\label{WkAinf}
\abs{\set{x \in Q: \innp{V\pr{x} \V{e}, \V{e}} \geq \delta \fint_Q \innp{V\pr{y}  \V{e}, \V{e}} \, dy}} \geq (1-\epsilon) |Q|.
\end{equation}
In particular, since there is no guarantee that $\disp \inf\set{\de\pr{\V{e}} : \V{e} \in \Sd} > 0$, the assumption that $V \in \Ai$ is not inherited from the assumption that $V \in \Bp$.
Example \ref{notNCEx} gives a matrix function that belongs to $\Bp$ for any $p$, but doesn't belong to $\NC$, and therefore by Lemma \ref{AIinNC}, also doesn't belong to $\Ai$.

Recall that $\la_d = \abs{V}$ denotes the largest eigenvalue of $V$.
As we saw in Lemma \ref{normVBp}, if $V \in \Bp$, then $\la_d$ belongs to $\sBp$.
Let $\la_1$ denote the smallest eigenvalue of $V$.
That is, $\la_1 = \abs{V^{-1}}^{-1}$.
Under a stronger set of assumptions, we can also make the interesting conclusion that $\la_1$ is in $\sBp$.

\begin{prop}[Smallest eigenvalue is scalar $\sBp$]
\label{la1Prop}
If $V \in \Bp \cap \Ai$, then $\la_1 \in \sBp$.
\end{prop}

The proof of this result appears in Appendix \ref{TechProofs}.
Although the assumption that $V \in \Bp \cap \Ai$ implies that the smallest and largest eigenvalues of $V$ belong to $\sBp$, it is unclear what conditions would imply that the other eigenvalues belong to this reverse H\"older class.

The next result and its proof show that the $\NC$ condition can be thought of as a noncommutative, non-$\Ai$ condition that is very naturally implied by the noncommutativity that is built into the $\Ai$ definition.

\begin{lem}[$\Ai \su \NC$]
\label{AIinNC}
If $V \in \Ai$, then $V \in \NC$.
\end{lem}

In the following proof, we establish that \eqref{NCCond} holds for all cubes $Q \su \R^n$, not just those at the special scale which are defined by $Q = Q\pr{x, \frac 1 {\um(x, V)}}$.

\begin{proof}
Since $V \in \Ai$, we may choose $\de > 0$ so that \eqref{Ainf} holds with $\eps = \frac 1 2$.
That is, for any $Q \su \R^n$, if we define $\disp S = \set{x \in Q: V\pr{x} \geq \delta \fint_Q V\pr{y} dy}$, then $\abs{S} \ge \frac 1 2 \abs{Q}$.
Observe that since
$$S = \set{x \in Q: V(x)^{\frac 1 2} V\pr{Q}^{-1} V(x)^{\frac 1 2} \geq \frac{\delta}{\abs{Q}} I},$$
then
\begin{align*}
\int_Q V(x)^{\frac 1 2} V\pr{Q}^{-1} V(x)^{\frac 1 2} dx
&\ge \int_{S} V(x)^{\frac 1 2} V\pr{Q}^{-1} V(x)^{\frac 1 2} dx
\ge \int_{S} \frac{\delta}{\abs{Q}} I dx
\ge \frac \de 2 I,
\end{align*}
showing that $V \in \NC$.
\end{proof}

Next we describe a collection of examples of matrix functions in $\Bp \cap \Atwi$.
Let $A = \pr{a_{ij}}_{i, j = 1}^d$ be a $d \times d$ Hermitian, positive definite matrix and let $\Ga = \pr{\ga_{ij}}_{i, j = 1}^d$ be some constant matrix.
We use $A$ and $\Ga$ to define the $d \times d$ matrix function $V : \R^n \to \R^{d \times d}$ by
\begin{equation}
\label{mpower}
V(x) = \begin{pmatrix}
a_{11} |x|^{\gamma_{11}} & \dots & a_{1d} |x|^{\gamma_{1d}} \\
\vdots &\ddots & \vdots\\
a_{d1} |x|^{\gamma_{d1}}  & \dots & a_{dd} |x|^{\gamma_{dd}}
\end{pmatrix}.
\end{equation}
By \cite[Theorem 3.1]{BLM17}, a matrix of the form \eqref{mpower} is positive definite a.e. iff $\gamma_{ij} = \ga_{ji} = \frac12 \pr{\gamma_{ii} + \gamma_{jj}}$ for $i, j = 1, \ldots, d$.
Moreover, in this setting, \cite[Lemma 3.4]{BLM17} shows that $V^{-1} : \R^n \to \R^{d \times d}$ is well-defined and given by
\begin{equation}
\label{mpowerIn}
V(x)^{-1} = \pr{a^{ij} \abs{x}^{- \ga_{ji}}}_{i, j =1}^d,
\end{equation}
where $A^{-1} = \pr{a^{ij}}_{i, j = 1}^d$.
Under the assumption of positive definiteness, these matrices provide a full class of examples of matrix weights in $\Bp \cap \Atwi$.

\begin{prop}
\label{BickelToProveProp}
Let $V$ be defined by \eqref{mpower} where $A = \pr{a_{ij}}_{i, j = 1}^d$ is a $d \times d$ Hermitian, positive definite matrix and $\ga_{ij} = \frac 1 2 \pr{\ga_i + \ga_j}$ for some $\V{\ga} \in \R^d$.
If $p \geq 1$ and $\gamma_{i} > - \frac{n}{p}$ for each $1 \leq i \leq d$, then $V \in \Bp \cap \Atwi$.
\end{prop}

The proof of this result appears in Appendix \ref{TechProofs}.

The classical Brunn-Minkowski inequality implies that the map $A \mapsto (\det A)^{\frac{1}{d}}$, defined on the set of $d \times d$ symmetric positive semidefinite matrices $A$, is concave.
An application of Jensen's inequality shows that
$$\pr{\det \fint_Q V(x) dx}^\frac{1}{d}  \geq  \fint_Q \brac{ \det V(x)} ^\frac{1}{d} dx;$$
see \eqref{DetConvexIneq} in the proof of Lemma \ref{MatrixJensen}.
Accordingly, we make the following definition of an associated reverse class.

\begin{defn}[$\RBM$]
\label{RBMDef}
We say that a matrix weight $V$ belongs to the {\bf reverse Brunn-Minkowski class}, $V \in \RBM$, if there exists a constant $B_V > 0$ so that for any cube $Q \su \Rn$, it holds that
\begin{equation}
\label{RBrunnMin}
\pr{\det \fint_Q V(x) dx}^\frac{1}{d} \leq B_V \fint_Q \brac{\det V(x)}^\frac{1}{d} dx.
\end{equation}
\end{defn}

In Appendix \ref{TechProofs}, we also provide the proof of the following``non-$\Ai$" condition for $V \in \NC$.

\begin{prop}
\label{RBrunnMinProp}
If $V \in \ND$ and there exists a constant $B_V > 0$ so that \eqref{RBrunnMin} holds for every cube $Q \su \R^n$, then $V \in \NC$.
\end{prop}

Even for a $d \times d$ diagonal matrix weight $V$ with a.e. positive entries $\lambda_1(x) , \ldots, \lambda_d(x)$, it is not clear when \eqref{RBrunnMin} holds.
If each $\lambda_j (x) \in \sAi$ for $1 \leq j \leq d$, then \eqref{RBrunnMin} holds.
In particular, since every diagonal matrix weight $V$ with positive a.e. entries belongs to $\NC$, then \eqref{RBrunnMin} doesn't provide a necessary condition for $\NC$.
It would be interesting to find an easily checkable sufficient condition for $V \in \NC$ that is at least trivially true in the case of diagonal matrix weights.

For a much deeper discussion of the classes $\Bp, \Ai, \Atwi$, and their relationships to each other, we refer the reader to Appendix \ref{AiApp}.
In fact, we hope that Appendix \ref{AiApp} will serve as a self-contained reference for the reader who is unfamiliar with theory of matrix weights.
We don't discuss matrix $\mathcal{A}_p$ weights in Appendix \ref{AiApp} since they play no role in this paper.
However, \cite{Aa09} serves as an excellent reference for the theory of matrix $\mathcal{A}_p$ weights and the boundedness of singular integrals on these spaces.

\section{Auxiliary Functions and Agmon Distances}
\label{MaxFun}

Now that we have introduced the class of matrices that we work with, we develop the theory of their associated auxiliary functions.
In the scalar setting, these ideas appear in \cite{She94}, \cite{She95}, \cite{She99}, and \cite{MP19}, for example.
As we are working with matrices instead of scalar functions, there are many different ways to generalize these ideas.

We assume from now on that $V \in \Bp \cap \ND$ for some $p \in \brac{\frac{n}{2}, \iny}$.
By Lemma \ref{GehringLemma}, there is no loss in assuming that $p > \frac n 2$.
Since $V \in \ND$, then by \eqref{eqB.2} and \eqref{NDCond}, for each $x \in \R^n$ and $r > 0$, we have $\Psi\pr{x, r; V} > 0$.
Since $p > \frac{n}{2}$, then the power $2 - \frac n p > 0$ and it follows from Lemma \ref{lB.1} that for any $\V{e} \in \Rd$,
\begin{equation}
\label{eqB.3}
\begin{aligned}
& \lim_{r \to 0^+} \innp{\Psi\pr{x, r; V} \V{e}, \V{e}} = 0 \\
& \lim_{R \to \iny} \innp{\Psi\pr{x, R; V} \V{e}, \V{e}} = \iny.
\end{aligned}
\end{equation}
This allows us to make the following definition.

If $V \in \Bp \cap \ND$ for some $p > \frac n 2$, then for $x \in \R^n$ and $\V{e} \in \Sd$, the \textit{auxiliary function} $m\pr{x, \V{e}, V} \in (0, \iny)$ is defined by
\begin{align}
\frac{1}{m\pr{x, \V{e}, V}} = \sup_{r > 0} \set{ r : \innp{\Psi\pr{x,r; V} \V{e}, \V{e}} \le 1}.
\label{eqB.5}
\end{align}

\begin{rem}
If $v$ is a scalar $\sBp$ function, then we may eliminate the $\V{e}$-dependence and define
\begin{align}
\frac{1}{m\pr{x, v}} = \sup_{r > 0} \set{ r : \Psi\pr{x,r; v} \le 1}.
\label{scalarmDef}
\end{align}
See \cite{She94}, \cite{She95}, \cite{She99}, for example.
\end{rem}

We recall the following lemma from \cite{She95}, for example, that applies to scalar functions.

\begin{lem}[cf. Lemma 1.4, \cite{She95}]
\label{lB.3}
Assume that $v \in \sBp$ for some $p > \frac n 2$.
There exist constants $C, c, k_0 > 0$, depending on $n$, $p$, and $C_v$, so that for any $x, y \in \R^n$,
\begin{enumerate}
\item[(a)] If $\disp \abs{x - y} \lesssim \frac{1}{m\pr{x, v}}$, then $\disp m\pr{x, v} \simeq_{(n, p, C_V)} m\pr{y, v}$,
\item[(b)] $\disp m\pr{y, v} \le C \brac{1 + \abs{x - y}m\pr{x, v}}^{k_0} m\pr{x, v}$,
\item[(c)] $\disp m\pr{y, v} \ge \frac{c \, m\pr{x, v}}{\brac{1 + \abs{x -y}m\pr{x, v}}^{k_0/\pr{k_0+1}}}$.
\end{enumerate}
\end{lem}

As the properties described in this lemma will be very useful below, we seek auxiliary functions that also satisfy this set of results in the matrix setting.
We define two auxiliary functions as follows.

\begin{defn}[Lower and upper auxiliary functions]
Let $V \in \Bp \cap \ND$ for some $p > \frac n 2$.
We define the \textbf{lower auxiliary function} as follows:
\begin{align}
\frac{1}{\um\pr{x, V}} = \sup_{r > 0} \set{ r : \min_{\V{e} \in \Sd} \innp{\Psi\pr{x,r; V} \V{e}, \V{e}} \le 1}.
\label{umDef}
\end{align}
The \textbf{upper auxiliary function} is given by
\begin{align}
\frac{1}{\ovm\pr{x, V}}
= \sup_{r > 0} \set{ r : \max_{\V{e} \in \Sd} \innp{\Psi\pr{x,r; V} \V{e}, \V{e}} \le 1}
= \sup_{r > 0} \set{ r : \abs{\Psi\pr{x,r; V}} \le 1}.
\label{omDef}
\end{align}
\end{defn}

\begin{rem}
\label{noND}
Since $\abs{\Psi\pr{x, r; V}}$ satisfies Lemma \ref{lB.1} whenever $V \in \Bp$, then for the upper auxiliary function, $\ovm\pr{x, V}$, we do not need to assume that $V \in \ND$.
\end{rem}

For $V$ fixed, we define
\begin{equation*}
\begin{aligned}
\underline{\Psi}\pr{x} &= \Psi\pr{x, \frac 1 {\um\pr{x, V}}; V} \\
\overline{\Psi}\pr{x} &= \Psi\pr{x, \frac 1 {\ovm\pr{x, V}}; V},
\end{aligned}
\end{equation*}
then observe that $\overline{\Psi}\pr{x} \le I \le \underline{\Psi}\pr{x}$.
In particular, for every $\V{e} \in \Sd$,
\begin{equation}
\label{psi12Prop}
\innp{\overline{\Psi}\pr{x} \V{e}, \V{e}} \le 1 \le \innp{\underline{\Psi}\pr{x} \V{e}, \V{e}}.
\end{equation}

With this pair of functions in hand, we now seek to prove Lemma \ref{lB.3} for both $\um$ and $\ovm$.
The following pair of observations for each auxiliary function will allow us to prove the desired results.

\begin{lem}[Lower observation]
\label{compLem}
Let $V \in \Bp \cap \ND$ for some $p > \frac n 2$.
If $\disp c \ge \innp{\Psi\pr{x,r; V} \V{e}, \V{e}}$ for some $\V{e} \in \Sd$, then $\disp r \leq \max\set{ 1, (C_Vc)^\frac{p}{2p - n}} \frac 1{\um(x, V)} $.
\end{lem}

\begin{proof}
If $r \le \frac{1}{\um\pr{x, V}}$, then we are done, so assume that $\frac 1 {\um\pr{x, V}} < r$.
Then it follows from \eqref{psi12Prop} and Lemma~\ref{lB.1} that for any $\V{e} \in \Sd$,
\begin{align*}
1 &\le \innp{\underline{\Psi}\pr{x} \V{e}, \V{e}}
= \innp{\Psi\pr{x, \frac 1 {\um\pr{x, V}}; V} \V{e}, \V{e}}
\le C_V \pr{\frac 1 {\um\pr{x, V}r}}^{2 - \frac n p} \innp{\Psi\pr{x, r; V} \V{e}, \V{e}} \\
&\le C_V c \pr{\frac 1 {\um\pr{x, V}r}}^{2 - \frac n p} .
\end{align*}
The conclusion follows after algebraic simplifications.
\end{proof}

As we observed in Lemma \ref{normVBp}, if $V \in \Bp$, then $\abs{V} \in \sBp$.
Thus, it is meaningful to discuss the quantity $m\pr{x, \abs{V}}$.
For $\ovm\pr{x, V}$, we rely on the following relationship regarding norms.
Note that by Remark \ref{noND}, we do not need to assume that $V \in \ND$ for this result.

\begin{lem}[Upper auxiliary function relates to norm]
\label{omCompLem}
If $V \in \Bp$ for some $p > \frac n 2$, then
$$\ovm(x, V) \le m(x, \abs{V})  \le \pr{d^2 C_V}^{\frac{2}{2p-n}} \ovm(x, V).$$
\end{lem}

\begin{proof}
For any $r > 0$, choose $\V{e} \in \Sd$ so that
\begin{align*}
\abs{\Psi(x, r ;V)}
= \innp{\Psi(x, r ;V) \V{e}, \V{e}}
= \innp{\pr{\frac{1}{r^{n-2}} \int_{Q\pr{x,r}} V\pr{y}dy } \V{e}, \V{e}}
= \frac{1}{r^{n-2}} \int_{Q\pr{x,r}} \innp{V\pr{y} \V{e}, \V{e}} dy.
\end{align*}
Since $\innp{V\pr{y}  \V{e}, \V{e}} \le \abs{V}$, then $\abs{\Psi(x, r ;V)} \leq \Psi(x, r ; \abs{V})$.
It follows that $\frac{1}{\ovm\pr{x, V}} \geq \frac{1}{m\pr{x, \abs{V}}}$ so that
$${m(x, \abs{V})} \geq \ovm(x, V).$$
Let $\set{\V{e}_i}_{i=1}^d$ denote the standard basis of $\R^d$.
For any $r > 0$, it follows from \eqref{normProp} that
\begin{equation}
\label{normRelationship}
\begin{aligned}
\Psi(x, r, \abs{V})
&= r^{2-n} \int_{Q(x, r)} \abs{V(y)} \, dy
\le r^{2-n} \int_{Q(x, r)} d \sum_{j = 1}^d \innp{V(y)\V{e}_j, \V{e}_j} \, dy \\
&= d \sum_{j = 1}^d \innp{\Psi(x, r, V) \V{e}_j, \V{e}_j}
\le d^2 \abs{\Psi(x, r;V)}.
\end{aligned}
\end{equation}
Combining the fact that $\Psi\pr{x, \frac{1}{m\pr{x, \abs{V}}}; \abs{V}} = 1$ with the previous observation, Lemma \ref{BasicShenLem}, and the definition of $\ovm$ shows that
\begin{align*}
1 & = \Psi\pr{x, \frac{1}{m\pr{x, \abs{V}}}; \abs{V}}
\le d^2 \abs{\Psi\pr{x, \frac{1}{m\pr{x, \abs{V}}}; V}}  \\
&\le d^2 C_V \brac{\frac{\ovm(x, V)}{{m}(x, \abs{V})}}^{2-\frac{n}{p}}  \abs{\Psi\pr{x, \frac{1}{\ovm\pr{x, {V}}}; V}}
= d^2 C_V \brac{\frac{\ovm(x, V)}{{m}(x, \abs{V})}}^{2-\frac{n}{p}},
\end{align*}
and the second part of the inequality follows.
\end{proof}

Since $\abs{V} = \la_d$, the largest eigenvalue of $V$, then this result shows that $\ovm\pr{x, V} \simeq m\pr{x, \la_d}$, indicating why we call $\ovm$ the upper auxiliary function.

Now we use these lemmas to establish a number of important tools related to the functions $\um\pr{x, V}$ and $\ovm\pr{x, V}$.
From now on we will assume that $|\cdot|$ on $\R^n$ refers to the $\ell_\infty$ norm.

\begin{lem}[Auxiliary function properties]
\label{muoBounds}
If $V \in \Bp \cap \ND$ for some $p > \frac n 2$, then both $\um(\cdot, V)$ and $\ovm(\cdot, V)$ satisfy the conclusions of Lemma \ref{lB.3} where all constants depend on $n$, $p$, and $C_V$.
For $\ovm(\cdot, V)$, the constants depend additionally on $d$ and we may eliminate the assumption that $V \in \ND$.
\end{lem}

\begin{proof}
First consider $\ovm(\cdot, V)$.
Lemma \ref{omCompLem} combined with Lemma \ref{normVBp} implies that all of these properties follow immediately from Lemma \ref{lB.3}.

Now consider $\um(\cdot, V)$.
Suppose $\abs{x - y} \le \frac{2^j -1}{\um\pr{x, V}}$ for some $j \in \N$.
Then $Q\pr{y, \frac 1 {\um\pr{x, V}}} \su Q\pr{x, \frac {2^j} {\um\pr{x, V}}}$.
Choose  $\V{e} \in \Sd$ so that $\innp{\Psi\pr{x, \frac 1 {\um\pr{x, V}}; V} \V{e}, \V{e}} = 1$.
Then
\begin{align*}
\innp{\Psi\pr{y, \frac 1 {\um\pr{x, V}}; V} \V{e}, \V{e}} \
&= \um\pr{x, V}^{n-2} \int_{Q\pr{y, \frac 1 {\um\pr{x, V}}}} \innp{V\pr{z} \V{e}, \V{e}} dz
\le \um\pr{x, V}^{n-2}\int_{Q\pr{x, \frac {2^j} {\um\pr{x, V}}}} \innp{V\pr{z} \V{e}, \V{e}} dz \\
&\le \um\pr{x, V}^{n-2} \ga^j \int_{Q\pr{x, \frac {1} {\um\pr{x, V}}}} \innp{V\pr{z} \V{e}, \V{e}} dz
= \ga^j \innp{\underline{\Psi}\pr{x} \V{e}, \V{e}}
= \ga^j,
\end{align*}
where we have used Lemma \ref{Vdbl} and $\gamma$ denotes the doubling constant.
It then follows from Lemma \ref{compLem} that $\disp \frac 1 {\um\pr{x, V}} \le \frac{\max\set{1, \pr{C_V \ga^j}^{p/(2p-n)}}} {\um\pr{y, V}}$  or
\begin{align}
\um\pr{y, V} \le \pr{C_V \ga^j}^{p/(2p-n)} \um\pr{x, V}.
\label{yxBd}
\end{align}

Since $\abs{x - y} \le \frac{2^j -1}{\um\pr{x, V}}$ and $\frac 1 {\um\pr{x, V}} \le \frac {\pr{C_V \ga^j}^{p/(2p-n)}} {\um\pr{y, V}}$, then $\abs{x - y} \le \frac{\pr{2^j - 1} \pr{C_V \ga^j}^{p/(2p-n)}}{\um\pr{y, V}}$.
Thus, $Q\pr{x, \frac 1 {\um\pr{y, V}}} \su Q\pr{y, \frac {\pr{2^j - 1}\pr{C_V \ga^j}^{p/(2p-n)}+1} {\um\pr{y, V}}}$.
Setting $\disp \tilde j =  \ceil{\ln\brac{\pr{2^j - 1}\pr{C_V \ga^j}^{p/(2p-n)}+1} / \ln 2}$, it can be shown, as above, that
\begin{align*}
\innp{\Psi\pr{x, \frac 1 {\um\pr{y, V}}; V} \V{e}, \V{e}}
\le \ga^{\tilde j},
\end{align*}
where now $\V{e} \in \Sd$ is such that $\innp{\underline{\Psi}\pr{y} \V{e}, \V{e}} = 1$.
Arguing as above, we see that $\frac{1}{\um\pr{y, V}} \le \frac{ \max\set{1, \pr{C_V \ga^{\tilde j}}^{p/(2p-n)}} }{\um\pr{x, V}}$, or
\begin{align}
\um\pr{x, V} \le \pr{C_V \ga^{\tilde j}}^{p/(2p-n)} \um\pr{y, V}.
\label{xyBd}
\end{align}

When $\abs{x - y} \lesssim \frac{1}{\um\pr{x, V}}$, we have that $j \simeq 1$ and $\tilde j \simeq 1$.
Then statement (a) is a consequence of \eqref{yxBd} and \eqref{xyBd}.

If $\abs{x - y} \le \frac 1 {\um\pr{x,V}}$, then part (a) implies that $\um\pr{y, V} \lesssim \um\pr{x, V}$ and the conclusion of (b) follows.
Otherwise, choose $j \in \N$ so that $\frac{2^{j-1}-1}{\um\pr{x, V}} \le \abs{x - y} < \frac{2^j-1}{\um\pr{x, V}}$.
From \eqref{yxBd}, we see that
$$\um\pr{y, V}
\le \pr{C_V \ga}^{\frac p {2p-n}} \pr{2^{j-1}}^{\frac{p \ln\ga}{\pr{2p-n} \ln 2 } } \um\pr{x, V}
\le \pr{C_V \ga}^{\frac p {2p-n}} \brac{1 + \abs{x - y}\um\pr{x, V}}^{\frac{p \ln\ga}{\pr{2p-n} \ln 2 } } \um\pr{x, V}.$$
Setting $C = \pr{C_V \ga}^{\frac p {2p-n}}$ and $k_0 = \frac{p \ln\ga}{\pr{2p-n} \ln 2 }$ gives the conclusion of (b).

If $\abs{x - y} \le \frac{1}{\um\pr{x, V}}$ or $\abs{x - y} \le \frac{1}{\um\pr{y, V}}$, then part (a) implies that $\um\pr{x, V} \lesssim \um\pr{y, V}$, and the conclusion of (c) follows.
Thus, we consider when $\abs{x - y} > \frac{1}{\um\pr{x, V}}$ and $\abs{x - y} > \frac{1}{\um\pr{y, V}}$.
Repeating the arguments from the previous paragraph with $x$ and $y$ interchanged, we see that
$$\um\pr{x, V} \le C \brac{1 + \abs{x - y}\um\pr{y, V}}^{k_0} \um\pr{y, V} \le 2^{k_0}C \abs{x - y}^{k_0}\um\pr{y, V}^{k_0+1}.$$
Rearranging gives that
\begin{align*}
 \um\pr{y, V}
 \ge \frac{2^{-k_0/(k_0+1)}C^{-1/(k_0+1)} \um\pr{x, V}}{ \pr{\um\pr{x, V}\abs{x - y}}^{k_0/(k_0+1)}}
  \ge \frac{2^{-k_0/(k_0+1)}C^{-1/(k_0+1)} \um\pr{x, V}}{ \pr{1 + \um\pr{x, V}\abs{x - y}}^{k_0/(k_0+1)}}.
\end{align*}
Taking $c = 2^{-k_0/(k_0+1)}C^{-1/(k_0+1)}$ leads to the conclusion of (c).
\end{proof}

Using these auxiliary functions, we now define the associated Agmon distance functions.

\begin{defn}[Agmon distances]
Let $\um(\cdot, V)$ and $\ovm(\cdot, V)$ be as in \eqref{umDef} and \eqref{omDef}, respectively.
We define the \textbf{lower Agmon distance function} as
\begin{equation*}
\ud(x, y, V) = \inf_{\ga} \int_0^1 \um(\ga(t), V) |\ga'(t)|\, dt ,
\end{equation*}
and the \textbf{upper Agmon distance function} as
\begin{equation*}
\ovd(x, y, V) = \inf_{\ga} \int_0^1 \ovm(\ga(t), V) |\ga'(t)|\, dt ,
\end{equation*}
where in both cases, the infimum ranges over all absolutely continuous $\ga:[0,1] \to \R^n$ with $\ga(0) = x$ and $\ga(1) = y$.
\end{defn}

We make the following observation.

\begin{lem}[Property of Agmon distances]
\label{closeRemark}
If $|x-y| \leq \frac{C}{\um(x, V)}$, then $\ud\pr{x, y, V} \lesssim_{(n, p, C_V)} C$.
If $|x-y| \leq \frac{C}{\ovm(x, V)}$, then $\ovd\pr{x, y, V} \lesssim_{(d, n, p , C_V)} C$.
\end{lem}

\begin{proof}
We only prove the first statement since the second one is analogous.
Let $x, y \in \R^n$ be as given.
Define $\ga : \brac{0,1} \to \R^n$ by $\ga\pr{t} = x + t\pr{y - x}$.
By Lemma \ref{muoBounds}(a), $\um\pr{\ga\pr{t}, V} \lesssim_{(n, p, C_V)} \um\pr{x, V}$ for all $t \in \brac{0, 1}$.
It follows that
\begin{align*}
\ud\pr{x, y, V}
&\le \int_0^1 \um(\ga(t), V) \abs{\ga'(t)} dt
\lesssim_{(n, p, C_V)} \int_0^1 \um\pr{x, V} \abs{x - y} dt
\lesssim_{(n, p, C_V)} C,
\end{align*}
as required.
\end{proof}

In future sections, the lower Agmon distance function will be an important tool for us once it has been suitably regularized.
We regularize this function $\ud(\cdot, \cdot, V)$ by following the procedure from \cite{She99}.
Observe that by Theorem \ref{muoBounds}(c), $\um\pr{\cdot, V}$ is a slowly varying function; see \cite[Definition 1.4.7]{Hor03}, for example.
As such, we have the following.

\begin{lem}[cf. the proof of Lemma 3.3. in \cite{She99}]
\label{partofU}
There exist sequences $\set{x_j}_{j=1}^\iny \su \R^n$ and $\set{\phi_j}_{j=1}^\iny \su C^\iny_0\pr{\R^n}$ such that
\begin{itemize}
\item[(a)] $\disp \R^n = \bigcup_{j=1}^\iny Q_j$, where $\disp Q_j = Q\pr{x_j, \frac 1 {\um\pr{x_j, V}}}$,
\item[(b)] $\phi_j \in C^\iny_0\pr{Q_j}$, $0 \le \phi_j \le 1$, and $\disp \sum_{j=1}^\iny \phi_j = 1$,
\item[(c)] $\abs{\gr \phi_j\pr{x}} \lesssim_{(n, p, C_V)} \um\pr{x, V}$,
\item[(d)] $\disp \sum_{j=1}^\iny \chi_{Q_j} \lesssim_{(n, p, C_V)} 1.$
\end{itemize}
\end{lem}

\begin{rem}
\label{partofURem}
Since $\ovm\pr{\cdot, V}$ is also a slowly varying function, the same result applies to $\ovm(\cdot, V)$ with constants that depend additionally on $d$.
\end{rem}

Using this lemma and \cite[Theorem 1.4.10]{Hor03}, we can follow the process from \cite[pp. 542]{She99} to establish the following pair of results.

\begin{lem}[Lemma 3.3 in \cite{She99}]
\label{RegLem0}
For each $y \in \R^n$, there exists nonnegative function $\varphi_V(\cdot, y) \in C^\infty(\R^n)$ such that for every $x \in \R^n$,
$$\abs{\varphi_V(x, y) - \ud(x, y, V)} \lesssim_{(n, p, C_V)} 1$$
and
$$|\nabla_x  \varphi_V(x, y)| \lesssim_{(n, p, C_V)} \um(x, V).$$
\end{lem}

\begin{lem}[Lemma 3.7 in \cite{She99}]
\label{RegLem1}
For each $y \in \R^n$, there exists a sequence of nonnegative, bounded functions $\set{\varphi_{V, j}\pr{\cdot, y}} \su C^\infty(\R^n)$ such that for every $x \in \R^n$,
$$\varphi_{V, j} (x, y) \leq \varphi_V(x, y)$$
and
$$\varphi_{V, j} (x, y) \to \varphi_V(x, y) \text{ as } j \to \infty.$$
Moreover,
$$\abs{\gr_x  \varphi_{V,j}(x, y)} \lesssim_{(n, p, C_V)} \um(x, V).$$
\end{lem}

To conclude the section, we observe that under the stronger assumption that $V \in \Bp \cap \Ai$, we can prove a result analogous to Lemma \ref{omCompLem} for the smallest eigenvalue.
By Proposition \ref{la1Prop}, $\la_1 \in \sBp$, so it is meaningful to discuss $m\pr{x, \la_1}$.
In subsequent sections, we will not assume that $V \in \Ai$, so this result should be treated as an interesting observation.
Its proof is provided at the end of Appendix \ref{TechProofs}.

\begin{prop}[Lower auxiliary function relates to $\la_1$]
\label{umCompLem}
If $V \in \Bp \cap \ND \cap \Ai$ for some $p > \frac n 2$, then
$$m(x, \la_1)  \le \um(x, V) \lesssim m(x, \la_1),$$
where the implicit constant depends on $n$, $p$, $C_V$ and the $\Ai$ constants.
\end{prop}

This result leads to the following observation.

\begin{cor}
If $V \in \Bp \cap \ND \cap \Ai$ for some $p > \frac n 2$, then $m\pr{x, \la_1}$ satisfies the conclusions of Lemma \ref{lB.3} where the constants have additional dependence on the $\Ai$ constants.
\end{cor}

\begin{rem}
In fact, if we assume that $V \in \Bp \cap \ND \cap \Ai$, then we can show that Lemma \ref{muoBounds} holds for $\um\pr{\cdot, V}$ in the same way that we show it holds for $\ovm\pr{\cdot, V}$.
That is, we apply Lemma \ref{lB.3} to $\la_1$, then use Proposition \ref{umCompLem}.
\end{rem}

\section{Fefferman-Phong Inequalities}
\label{FPI}

In this section, we present and prove our matrix versions of the Fefferman-Phong inequalities.
The first result is a lower Fefferman-Phong inequality which holds with the lower auxiliary function from \eqref{umDef}.
This result will be applied in Section \ref{UpBds} where we establish upper bound estimates for the fundamental matrices.
A corollary to this lower Fefferman-Phong inequality, which is used in Section \ref{LowBds} to prove lower bound estimates for the fundamental matrices, is then provided.
In keeping with \cite{She99}, we also present the upper bound with the upper auxiliary function from \eqref{omDef}.

Before stating and proving the lower Fefferman-Phong inequality, we present the Poincar\'e inequality that will be used in its proof.
Additional and more complex matrix-valued Poincar\'e inequalities and related Chanillo-Wheeden type conditions appear in the forthcoming manuscript \cite{DI22}.

\begin{prop}[Poincar\'e inequality]
\label{PoincareIneqThm}
Let $V \in \mathcal{B}_{\frac{n}{2}}$.
For any  open cube $Q \su \R^n$ and any $\V{u} \in C^1(Q)$, we have
\begin{equation*}
{\int_Q \int_Q \abs{(V(Q))^{-\frac12} V^\frac12 (y) \pr{\V{u} (x) - \V{u} (y)}}^2 \, dx \, dy}
\lesssim_{\pr{d, n, C_V}} |Q| ^\frac{2}{n} {\int_Q \abs{ D \V{u} (x)}^2 \, dx}.
\end{equation*}
\end{prop}

We prove this result by following the arguments from the scalar version of the Poincar\'{e} inequality in Shen's article, \cite[Lemma 0.14]{She99}.

\begin{proof}
Fix a cube $Q$ and define the scalar weight $v_Q (y) = \abs{V(Q)^{-\frac12} V(y) V(Q)^{-\frac12}}$.

First we show that $v_Q \in B_{\frac n 2}$ with a comparable constant.
For an arbitrary cube $P$, observe that by \eqref{normProp}
\begin{align*}
\pr{\fint_{P} \abs{v_Q (y)}^{\frac n 2} \, dy}^{\frac 2 n}
&\le \brac{ \fint_{P} \pr{d \sum_{j = 1}^d \innp{V(Q)^{-\frac12} V(y) V(Q)^{-\frac12} \V{e}_j, \V{e}_j}}^{\frac n 2} \, dy }^{\frac 2 n} \\
&\le d^{2 -\frac 2 n} \sum_{j = 1}^d  \pr{ \fint_{P}  \innp{V(Q)^{-\frac12} V(y) V(Q)^{-\frac12} \V{e}_j, \V{e}_j}^{\frac n 2} \, dy }^{\frac 2 n} \\
&\le d^{2 -\frac 2 n} C_V \fint_{P}   \sum_{j = 1}^d \innp{V(Q)^{-\frac12} V(y) V(Q)^{-\frac12} \V{e}_j, \V{e}_j}\, dy
\le d^{3 -\frac 2 n} C_V \fint_{P}  \abs{v_{Q} (y)} \, dy,
\end{align*}
where we have used that $V \in \mathcal{B}_{\frac n 2}$ to reach the third line.
This shows that $v_Q \in B_{\frac n 2}$.

Since $v_Q \in B_{\frac{n}{2}}$, then it follows from \cite[Lemma 0.14]{She99} that
\begin{align*}
\int_Q \int_Q |(V(Q))^{-\frac12} & V^\frac12 (y) (\V{u} (x) - \V{u} (y))|^2 \, dx \, dy
\le  {\int_Q \int_Q  \abs{(V(Q))^{-\frac12} V^\frac12 (y)}^2 \abs{\V{u} (x) - \V{u} (y)}^2 \, dx \, dy}
\\ & = {\int_Q \int_Q  \abs{\V{u} (x) - \V{u} (y)}^2 \, v_Q (y) dy  \, dx}
= {\int_Q \int_Q \sum_{j = 1}^d \abs{u_j (x) - u_j (y)}^2 \, v_Q (y) dy  \, dx}
\\ & \lesssim_{\pr{d, n, C_V}} |Q|^{\frac{2}{n}} v_Q (Q) \int_Q \sum_{j = 1}^d \abs{\nabla u_j (x)}^2 \, dx
= |Q|^{\frac{2}{n}} v_Q (Q) \int_Q \abs{ D \V{u} (x)}^2 \, dx.
\end{align*}
Since
\begin{align*}
v_Q (Q)
& = \int_Q \abs{V(Q)^{-\frac12} V(y) V(Q)^{-\frac12}} \, dy
\le d \sum_{j = 1}^d  \int_Q \innp{V(Q)^{-\frac12} V(y) V(Q)^{-\frac12} \V{e}_j, \V{e}_j}  \, dy
\\ & = d  \sum_{j = 1}^d \innp{V(Q)^{-\frac12} V(Q) V(Q)^{-\frac12} \V{e}_j, \V{e}_j}
= d^2,
\end{align*}
the conclusion follows.
\end{proof}

Now we present the lower Fefferman-Phong inequality.
This result will be applied in Section \ref{UpBds} to prove the exponential upper bound on the fundamental matrix.
Note that we assume here that $V$ belongs to the first three matrix classes that were introduced in Section \ref{MWeights}.

\begin{lem}[Lower Auxiliary Function Fefferman-Phong Inequality]
\label{FPml}
Assume that $V \in \Bp \cap \ND \cap \NC$ for some $p > \frac{n}{2}$.
Then for any $ \V{u} \in C^1 _0(\R^n)$, it holds that
$$\int_{\R^n} \abs{\um\pr{x, V} \V{u}(x)}^2  \, dx
\lesssim_{(d, n, p, C_V, N_V)} \inrn \abs{D\V{u} (x)}^2 \, dx
+ \inrn \abs{V^\frac12 (x) \V{u} (x)}^2 \, dx.$$
\end{lem}

\begin{proof}
For some $x_0 \in \R^n$, let $r_0 = \frac{1}{\um(x_0, V)}$ and set $Q = Q(x_0, r_0)$.
Property $\NC$ in \eqref{NCCond} shows that
\begin{align}
N_V \int_{Q} \abs{\V{u}(x)}^2  \, dx
&\le \int_{Q} \int_{Q} \innp{V^\frac12 (y) V(Q) ^{-1} V^\frac12 (y)\V{u}(x), \V{u}(x)} \, dy dx \nonumber \\
&= \int_{Q} \int_{Q} \abs{(V(Q)) ^{-\frac12} V^\frac12(y) \V{u}(x)}^2 \, dy dx \nonumber \\
& \lesssim \int_{Q} \int_{Q} \abs{(V(Q)) ^{-\frac12} V^\frac12(y)( \V{u}(x) - \V{u}(y))}^2  \, dy dx
+ \int_{Q} \int_{Q} \abs{(V(Q))^{-\frac12} V^\frac12(y) \V{u}(y)}^2 \, dy dx \nonumber \\
& \lesssim_{(d, n, C_V)} r_0^{2} \int_{Q} \abs{D\V{u}(x)}^2 \, dx
+ r_0^n \int_{Q} \abs{(V(Q))^{-\frac12} V^\frac12(y) \V{u}(y)}^2 \, dy,
\label{NCPoincIneq}
\end{align}
where the last line follows from an application of Proposition \ref{PoincareIneqThm}.
Now we multiply this inequality through by $r_0^{-2} = \um\pr{x_0, V}^{2}$, then apply Lemma \ref{muoBounds} to conclude that $\um\pr{x_0, V} \simeq_{(n, p, C_V)} \um\pr{x, V}$ on $Q$.
It follows that
\begin{align*}
\int_{Q} \abs{\um\pr{x, V} \V{u}(x)}^2  \, dx
&\lesssim_{(d, n, p, C_V, N_V)} \int_{Q} \abs{D\V{u}(x)}^2 \, dx + r_0^{n-2} \abs{(V(Q))^{-1}}  \int_{Q} \abs{V^\frac12(y) \V{u}(y)}^2 \, dy. 
\end{align*}
Since $r_0^{2-n}V(Q) = \Psi(x_0, r_0, V)  \geq I$ implies that $r_0^{n-2} \abs{(V(Q))^{-1}}  = \abs{\Psi(x_0, r_0, V)^{-1} } \leq 1$, then for any $Q = Q\pr{x_0, \frac 1 {\um\pr{x_0, V}}}$, we have shown that
\begin{align}
\label{loweronCubes}
\int_{Q} \abs{\um\pr{x, V} \V{u}(x)}^2  \, dx
&\lesssim_{(d, n, p, C_V, N_V)} \int_{Q} \abs{D\V{u}(x)}^2 \, dx + \int_{Q} \abs{V^\frac12(x) \V{u}(x)}^2 \, dx.
\end{align}

According to Lemma \ref{partofU}, there exists a sequence $\set{x_j}_{j=1}^\iny \su \R^n$ such that if we define $\disp Q_j = Q\pr{x_j, \frac 1 {\um\pr{x_j, V}}}$, then $\disp \R^n = \bigcup_{j=1}^\iny Q_j$ and $\disp \sum_{j=1}^\iny \chi_{Q_j} \lesssim_{(n, p, C_V)} 1.$
Therefore, it follows from \eqref{loweronCubes} that
\begin{align*}
\int_{\R^n}\abs{\um\pr{x, V} \V{u}(x)}^2  \, dx
&\le \sum_{j=1}^\iny \int_{Q_j} \abs{\um\pr{x, V} \V{u}(x)}^2  \, dx \\
&\lesssim_{(d, n, p, C_V, N_V)} \sum_{j=1}^\iny \pr{\int_{Q_j} \abs{D\V{u}(x)}^2 \, dx + \int_{Q_j} \abs{V^\frac12(x) \V{u}(x)}^2 \, dx} \\
&\lesssim_{(n, p, C_V)} \int_{\R^n} \abs{ D\V{u}(x)}^2 \, dx +  \int_{\R^n} \abs{V^\frac12(x) \V{u}(x)}^2 \, dx,
\end{align*}
as required.
\end{proof}

\begin{rem}
If we assume that $\V{u} \equiv 1$ on $Q$, then the condition $\NC$ is necessary for \eqref{NCPoincIneq} to be true on all such cubes.
As such, the condition $\NC$ is very natural assumption to impose.
In fact, as we show in Appendix \ref{Examples}, there are matrix weights $V \in \pr{\Bp \cap \ND} \setminus \NC$ for which this Fefferman-Phong estimate fails to hold.
\end{rem}

Finally, if we replace $V$ by $\abs{V} I$, we are essentially reduced to the scalar setting and we only need that $\abs{V} \in \sBp$.
In particular, we don't need to assume that $V \in \NC$.
As shown in Sections \ref{UpBds} and \ref{LowBds}, this result will be applied to prove the exponential lower bound on the fundamental matrix.

\begin{cor}[Norm Fefferman-Phong Inequality]
\label{FPmlCor}
Assume that $\abs{V} \in \sBp$ for some $p > \frac{n}{2}$.
Then for any $ \V{u} \in C^1 _0(\R^n)$, it holds that
$$\int_{\R^n} \abs{m\pr{x, \abs{V}} \V{u}(x)}^2  \, dx
\lesssim_{(d, n, p, C_{\abs{V}})} \inrn \abs{D\V{u} (x)}^2 \, dx
+ \inrn \abs{V} \abs{\V{u} (x)}^2 \, dx.$$
\end{cor}

To conclude the section, although we will not use it, we present the straightforward upper bound which is similar \cite[Theorem 1.13(b)]{She99}.
Notice that for this result, we only assume that $V \in \Bp \cap \ND$.

\begin{prop}[Upper Auxiliary Function Fefferman-Phong Inequality]
\label{FPmu}
Assume that $V \in \Bp$ for some $p > \frac n 2$.
For any $ \V{u} \in C_0^1(\R^n)$, it holds that
$$\int_{\R^n} \abs{V^\frac12(x) \V{u}(x)}^2  \, dx \lesssim_{(d, n, p, C_V)} \inrn \abs{D\V{u} (x)}^2 \, dx + \inrn \abs{\ovm(x, V) \V{u} (x)}^2 \, dx.$$
\end{prop}

\begin{proof}
Note that $\disp \abs{V^\frac12(x) \V{u}(x)}^2 = \innp{V(x) \V{u}(x), \V{u}(x)} \le \abs{V\pr{x}} \abs{\V{u}\pr{x} }^2$.
For some $x_0 \in \R^n$, let $r_0 = \frac{1}{\ovm(x_0, V)}$ and set $Q = Q(x_0, r_0)$.
Then using the classical Poincar\'{e} inequality, we have that
\begin{align*}
\int_{Q} \abs{V\pr{x}} \abs{\V{u}\pr{x} }^2  \, dx
&= \frac{1}{\abs{Q}} \int_{Q} \int_{Q} \abs{V\pr{x}} \abs{\V{u}\pr{x} }^2 \, dx \, dy   \\
&\lesssim \frac{1}{\abs{Q}} \int_{Q} \int_{Q} \abs{V\pr{x}} \abs{\V{u}\pr{x} - \V{u}\pr{y}}^2\, dx \, dy
+ \frac{1}{\abs{Q}} \int_{Q} \int_{Q} \abs{V\pr{x}} \abs{\V{u}\pr{y} }^2 \, dx \, dy \\
&\lesssim_{(n)} r_0^{2-n} \int_{Q} \abs{V\pr{x}} dx \int_{Q}  \abs{D\V{u}\pr{y}}^2\, dy
+ r_0^{-n} \int_{Q} \abs{V\pr{x}} dx \int_{Q} \abs{\V{u}\pr{y} }^2 \, dy \\
&= \Psi\pr{x_0, r_0; \abs{V}}\pr{ \int_{Q} \abs{D\V{u}\pr{y}}^2\, dy
+ r_0^{-2} \int_{Q} \abs{\V{u}\pr{y} }^2 \, dy} \\
&\lesssim_{(d)}  \int_{Q}  \abs{D\V{u}\pr{y}}^2\, dy
+ r_0^{-2} \int_{Q} \abs{\V{u}\pr{y} }^2 \, dy,
\end{align*}
since by \eqref{normRelationship}, $d^{-2} \Psi\pr{x_0, r_0; \abs{V}} \le \abs{\Psi\pr{x_0, r_0; V}} = 1$.
We apply Lemma \ref{muoBounds} to conclude that for $x \in Q$, $r_0^{-1} = \ovm\pr{x_0, V} \simeq_{(d, n, p, C_V)} \ovm\pr{x, V}$.
In particular, for any $Q = Q\pr{x_0, \frac{1}{\ovm(x_0, V)}}$,
\begin{equation}
\label{upperonCubes}
\int_{Q} \abs{V^\frac12(x) \V{u}(x)}^2 \, dx
\lesssim_{(d, n, p, C_V)}  \int_{Q} \abs{ D\V{u}(x)}^2\, dx
+  \int_{Q}  \abs{\V{u}(x)}^2 \ovm(x, V)^{2}\, dx.
\end{equation}
To pass from cubes to $\R^n$, we follow the arguments from the proof of Lemma \ref{FPml} and use Remark \ref{partofURem}.
\end{proof}

\section{The Elliptic Operator}
\label{EllOp}

In this section, we introduce the generalized Schr\"odinger operators.
For this section and the subsequent two, we do not need to assume that our the matrix weight $V$ belongs to $\Bp$ and therefore work in a more general setting.
In particular, to define the operator, the fundamental matrices, and discuss a class of systems operators that satisfy a set of elliptic theory results, we only require nondegeneracy and local $p$-integrability of the zeroth order potential terms.
The stronger assumption that $V \in \Bp$ is not required until we establish more refined bounds for the fundamental matrices; namely the exponential upper and lower bounds.
As such, the next three sections are presented for $V$ in this more general setting.

For the leading operator, let $A^{\al \be} = A^{\al \be}\pr{x}$, for each $\al, \be \in \set{ 1, \dots, n}$, be an $d \times d$ matrix with bounded measurable coefficients defined on $\R^n$.
We assume that there exist constants $0 < \la, \La < \iny$ so that $A^{\al \be}$ satisfies an ellipticity condition of the form
\begin{align}
\sum_{i, j = 1}^d \sum_{\al, \be = 1}^n A^{\al \be}_{ij}\pr{x} \xi_\be^{j} \xi_\al^{i}
&\ge \la \sum_{i = 1}^d \sum_{\al = 1}^n \abs{\xi_\al^i}^2 = \la \abs{\xi}^2
\quad \text{ for all } \, x \in \R^d, \xi \in \R^{d \times n}
\label{ellip}
\end{align}
and a boundedness assumption of the form
\begin{align}
& \sum_{i, j = 1}^d \sum_{\al, \be = 1}^n A_{ij}^{\al \be}\pr{x} \xi_\be^{j} \zeta_\al^{i}
\le \La \sum_{i, j = 1}^d \sum_{\al, \be = 1}^n \xi_\be^{j} \zeta_\al^{i}
\le \La \abs{\xi} \abs{\zeta}
\quad \text{ for all } x \in \R^d, \xi, \zeta \in \R^{d \times n}.
\label{Abd}
\end{align}
For the zeroth order term, we assume that
\begin{align}
V \in L^{\frac n 2}_{\loc}\pr{\R^n} \cap \ND.
\label{VAssump}
\end{align}

In particular, since $V$ is a matrix weight, then $V$ is a $d \times d$, a.e. positive semidefinite, symmetric, $\R$-valued matrix function.

\begin{rem}
\label{VAssumpRem}
Note that if $V \in \Bp \cap \ND$ for some $p \in \brac{\frac n 2, \iny}$, then since $V \in \Bp$ implies that $V \in L^p_{loc}\pr{\R^n}$ for some $p \ge \frac n 2$, such a choice of $V$ satisfies \eqref{VAssump}.
This more specific assumption on the potential functions will appear in Sections \ref{UpBds} and \ref{LowBds}.
\end{rem}

The equations that we study are formally given by
\begin{align}
\label{elEqDef}
\LV \V{u} &
= -D_\al\pr{A^{\al \be} D_\be \V{u} } + V \V{u}.
\end{align}

To make sense of what it means for some function $\V{u}$ to satisfy \eqref{elEqDef}, we introduce new Hilbert spaces.
But first we recall some familiar and related Hilbert spaces.
For any open $\Om \su \R^n$, $W^{1,2}(\Om)$ denotes the family of all weakly differentiable functions $u \in L^{2}(\Om)$ whose weak derivatives are functions in $L^2(\Om)$, equipped with the norm that is given by
$$\norm{u}_{W^{1,2}(\Om)}^2 = \norm{u}_{L^{2}(\Om)}^2 + \norm{D u}_{L^2(\Om)}^2.$$
The space $W^{1,2}_0(\Om)$ is defined to be the closure of $C^\iny_c(\Om)$ with respect to $\norm{\cdot}_{W^{1,2}(\Om)}$.
Recall that $C^\iny_c(\Om)$ denotes the set of all infinitely differentiable functions with compact support in $\Om$.

Another related class of functions will be used also.
For any open $\Om \su \R^n$, the space $Y^{1,2}(\Om)$ is the family of all weakly differentiable functions $u \in L^{2^*}(\Om)$ whose weak derivatives are functions in $L^2(\Om)$, where $2^*=\frac{2n}{n-2}$.
We equip $Y^{1,2}(\Om)$ with the norm
\begin{align*}
\norm{u}^2_{Y^{1,2}(\Om)} :=  \norm{u}^2_{L^{2^*}(\Om)} + \norm{D u}^2_{L^2(\Om)}.
\end{align*}
Define $Y^{1,2}_0(\Om)$ as the closure of $C^\iny_c(\Om)$ in $Y^{1,2}(\Om)$.
When $\Om = \R^n$, $Y^{1,2}\pr{\R^n} = Y^{1,2}_0\pr{\R^n}$ (see, e.g.,  Appendix A in \cite{DHM18}).
By the Sobolev inequality,
\begin{equation*}
\norm{u}_{L^{2^*}(\Om)}
\le c_n \norm{D u}_{L^2(\Om)} \quad \text{for all $u \in Y^{1,2}_0(\Om)$.}
\end{equation*}
It follows that $W^{1,2}_0(\Om) \su Y^{1,2}_0(\Om)$ with set equality when $\Om$ has finite measure.
The bilinear form on $Y_0^{1,2}(\Om)$ that is given by
\begin{align*}
\innp{u, v}_{Y_0^{1,2}(\Om)} := \int_{\Om} D_\al u D_\al v
\end{align*}
defines an inner product on $Y_0^{1,2}(\Om)$.
With this inner product, $Y_0^{1,2}(\Om)$ is a Hilbert space with norm
\begin{align*}
\norm{u}_{Y_0^{1,2}(\Om)} := \innp{u, u}_{Y_0^{1,2}(\Om)}^{1/2} = \norm{Du}_{L^2(\Om)}.
\end{align*}
We refer the reader to \cite[Appendix A]{DHM18} for further properties of $Y^{1,2}(\Om)$, and some relationships between $Y^{1,2}(\Om)$ and $W^{1,2}(\Om)$.
These spaces can be generalized to vector-valued functions in the usual way.

Towards the development of our new function spaces, we start with the associated inner products.
For any $V$ as in \eqref{VAssump} and any $\Om \su \R^n$ open and connected, let $\innp{\cdot, \cdot}^2_{W_V^{1,2}(\Om)} : C_c^\iny(\Om) \times C_c^\iny(\Om) \to \R$ be given by
\begin{align*}
\innp{\V{u}, \V{v}}_{W_V^{1,2}(\Om)}
= \int_{\Om} \innp{V \V{u}, \V{v}} + \innp{D \V{u}, D \V{v}}.
\end{align*}
This inner product induces a norm, $\norm{\cdot}^2_{W_V^{1,2}(\Om)} : C_c^\iny(\Om) \to \R$ that is defined by
\begin{align*}
\norm{\V{u}}^2_{W_V^{1,2}(\Om)}
:= \norm{V^{1/2} \V{u}}^2_{L^{2}(\Om)} + \norm{D \V{u}}_{L^2(\Om)}^2
= \int_{\Om} \innp{V \V{u}, \V{u}} + \innp{D \V{u}, D \V{u}}.
\end{align*}
The nondegeneracy condition described by \eqref{NDCond} ensures that this is indeed a norm and not just a semi-norm.
In particular, if $\norm{D \V{u}}_{L^2(\Om)} = 0$, then $\V{u} = \V{c}$ a.e., where $\V{c}$ is a constant vector.
But then by \eqref{NDCond}, $\norm{V^{1/2} \V{u}}^2_{L^{2}(\Om)} = \norm{V^{1/2} \V{c}}^2_{L^{2}(\Om)} = 0$ iff $\V{c} = \V{0}$.

For any $\Om \su \R^n$ open and connected, we use the notation $L_V^2(\Om)$ to denote the space of $V$-weighted square integrable functions.
That is,
$$L_V^2(\Om) = \set{\V{u} : \Om \to \R^d : \norm{V^{1/2} \V{u}}_{L^2\pr{\Om}} < \iny}.$$

For any $V$ as in \eqref{VAssump} and any $\Om \su \R^n$ open and connected, define each space $W_{V,0}^{1,2}(\Om)$ as the completion of $C_c^\iny(\Om)$ with respect to the norm $\norm{\cdot}_{W_V^{1,2}(\Om)}$.
That is,
\begin{equation}
\label{WV012Def}
W_{V,0}^{1,2}(\Om) = \overline{C_c^\iny(\Om)}^{\norm{\cdot}_{W_V^{1,2}(\Om)}}.
\end{equation}

The following proposition clarifies the meaning of our trace zero Sobolev space.

\begin{prop}
\label{W12V0Properties}
Let $V$ be as in \eqref{VAssump} and let $\Om \su \R^n$ be open and connected.
For every sequence $\{\V{u}_k\}_{k=1}^\iny \subset C_c ^\infty (\Om)$ that is Cauchy with respect to the $W_{V}^{1,2}(\Om)$-norm, there exists a $\V{u} \in L_V ^2(\Om) \cap Y^{1,2}_0(\Om)$  for which
$$\lim_{k \to \iny} \norm{\V{u}_k - \V{u}}_{W^{1,2}_V(\Om)}^2 = \lim_{k \to \iny} \pr{\inom \abs{V^\frac12 \pr{\V{u}_k - \V{u}}}^2 + \inom \abs{D\V{u}_k - D\V{u}}^2}  = 0.$$
\end{prop}

\begin{proof}
Since $\{\V{u}_k\} \subset C_c ^\infty (\Om)$ is Cauchy in the $W_{V}^{1,2}(\Om)$ norm, $\set{V^{1/2} \V{u}_k}$ is Cauchy in $L^2(\Om)$, and thus there exists $\V{h} \in L^2(\Om)$ so that
\begin{equation}
\label{L2Lim}
V^{1/2} \V{u}_k \to \V{h} \quad \text{ in } \quad L^2(\Om).
\end{equation}
Similarly, since $\set{D\V{u}_k}$ is Cauchy in $L^2(\Om)$, there exists $U \in L^2(\Om)$ so that
\begin{equation}
\label{L2LimD}
D\V{u}_k \to U \quad \text{ in } \quad L^2(\Om).
\end{equation}
By the Sobolev inequality applied to $\V{u}_k-\V{u}_j$, we have
$$\norm{\V{u}_k-\V{u}_j}_{L^{2^*}(\Om)} \leq c_n \norm{D\pr{\V{u}_k - \V{u}_j}}_{L^{2}(\Om)} \leq c_n \norm{\V{u}_k-\V{u}_j}_{W_V^{1,2}(\Om)}.$$
In particular, $\set{\V{u}_k}$ is also Cauchy in $L^{2^*}(\Om)$ and then there exists $\V{u} \in L^{2^*}(\Om)$ so that
\begin{equation}
\label{L2*Lim}
\V{u}_k \to \V{u} \quad \text{ in } \quad L^{2^*}(\Om).
\end{equation}
For any $\Om' \Subset \R^n$, observe that another application of H\"older's inequality shows that
\begin{align*}
\norm{V^{1/2} \V{u}_k - V^{1/2} \V{u} }_{L^{2}\pr{\Om \cap \Om'}}
&\le \pr{\int_{\Om \cap \Om'} \abs{V} \abs{ \V{u}_k - \V{u}}^2 }^{\frac 1 2}
\le \norm{V}_{L^{\frac n 2}\pr{\Om'}}^{2} \norm{\V{u}_k - \V{u} }_{L^{2^*}\pr{\Om}}.
\end{align*}
Since $V \in L^{\frac n 2}_{\loc}(\R^n)$, then $\norm{V}_{L^{\frac n 2}\pr{\Om'}} < \iny$ and we conclude that $V^{1/2} \V{u}_k \to V^{1/2} \V{u}$ in $L^2\pr{\Om \cap \Om'}$ for any $\Om' \Subset \R^n$.
By comparing this statement with \eqref{L2Lim}, we deduce that $V^{1/2} \V{u} = \V{h}$ in $L^2(\Om)$ and therefore a.e., so that \eqref{L2Lim} holds with $\V{h} = V^{1/2} \V{u}$.
Moreover, $\V{u} \in L^2_V\pr{\Om}$.

Next we show that $D\V{u} = U$ weakly in $\Om$.
Let $\V{\xi} \in C^\iny_c(\Om)$.
Then for $j = 1, \ldots, n$, we get from \eqref{L2*Lim} and \eqref{L2LimD} that
\begin{align*}
\int_{\Om} \innp{\V{u} (x), D_j \V{\xi} (x)} \, dx
&= \lim_{k \to \infty}  \int_{\Om} \innp{\V{u}_k (x), D_j \V{\xi} (x)} \, dx
= -\lim_{k \to \infty} \int_{\Om} \innp{D_j \V{u}_k (x), \V{\xi} (x)} \, dx \\
&= - \int_{\Om} \innp{U_j(x), \V{\xi} (x)} \, dx,
\end{align*}
where $U_j$ denotes the $j^{\textrm{th}}$ column of $U$.
That is, $D\V{u} = U \in L^2(\Om)$.
In particular, \eqref{L2LimD} holds with $U = D\V{u}$.
Finally, in combination with \eqref{L2LimD} and \eqref{L2*Lim}, this shows that $\V{u} \in Y^{1,2}_0(\Om)$.
\end{proof}

By Proposition \ref{W12V0Properties}, associated to each equivalence class of Cauchy sequences $[\{\V{u}_k\}] \in W_{V, 0} ^{1, 2} (\Om)$ is a function $\V{u}  \in L_V ^2(\Om) \cap Y^{1,2}_0(\Om)$ with
$$\lim_{k \to \iny} \norm{\V{u}_k - \V{u}}_{W^{1,2}_V(\Om)} = 0$$
so that
$$ \norm{[\{\V{u}_k\}]}_{W_{V}^{1, 2} (\Om)} :=  \lim_{k \to \iny} \norm{\V{u}_k }_{W^{1,2}_V(\Om)} = \norm{ \V{u}}_{W^{1,2}_V(\Om)}.$$
In fact, this defines a norm on weakly-differentiable vector-valued functions $\V{u}$ for which  $\norm{ \V{u}}_{W^{1,2}_V(\Om)} < \iny$.
It follows that the function $\V{u}$ is unique and this shows that $W_{ V, 0} ^{1, 2} (\Om)$ isometrically imbeds into the space $L_V^2(\Om) \cap Y^{1,2}_0(\Om)$ equipped with the norm $\norm{\cdot}_{W^{1,2}_V(\Om)}$.

Going forward, we will slightly abuse notation and denote each element in $W_{V, 0}^{1, 2} (\Om)$ by its unique associated function $\V{u} \in L_V ^2(\Om) \cap Y^{1,2}_0(\Om)$.
To define the nonzero trace spaces that we need below to prove the existence of fundamental matrices, we use restriction.
That is, define the space
\begin{equation}
\label{WV12Def}
W_{V}^{1,2}(\Om)
= \set{\V{u}\rvert_\Om : \V{u} \in W_{V,0}^{1,2}(\R^n)}
\end{equation}
and equip it with the $W_{V}^{1,2}(\Om)$-norm.

Note that $W_{V}^{1,2}(\R^n) = W_{V, 0}^{1,2}(\R^n)$.
Moreover, when $\Om \ne \R^n$, $W_{V}^{1,2}(\Om)$ may not be complete so we simply treat it as an inner product space.
We stress that in general, $W_{V}^{1,2}(\Om)$ should \textit{not} be thought of as a kind of ``Sobolev space," but should instead be viewed as a convenient collection of functions used in the construction from \cite{DHM18}.
Specifically, the construction of fundamental matrices from \cite{DHM18} uses the restrictions of elements from appropriate ``trace zero Hilbert-Sobolev spaces" defined on $\Rn$.
For us, $W_{V,0}^{1,2}(\Rn)$ plays the role of the trace zero Hilbert-Sobolev space.
Also, as an immediate consequence of Proposition \ref{W12V0Properties} we have the following.

\begin{cor}
\label{W12VProperties}
Let $V$ be as in \eqref{VAssump} and let $\Om \su \R^n$ be open and connected.
If $\V{u} \in W_V ^{1, 2} (\Om)$, then $\V{u} \in L_V ^2(\Om) \cap Y^{1,2}(\Om)$ and there exists $\{\V{u}_k \}_{k=1}^\iny \subset C_c ^\infty(\Rn)$ for which
$$\lim_{k \to \iny} \norm{\V{u}_k - \V{u}}_{W^{1,2}_V(\Om)}^2 = 0.$$
\end{cor}

We now formally fix the notation and then we will discuss the proper meaning of the operators at hand.
For every $\V{u} = \pr{u^1,\ldots, u^d }^T$ in $W^{1,2}_{V}(\Om)$ (and hence in  $Y^{1,2}\pr{\Om}$), we define $\Lz \V{u} = - D_\al \pr{A^{\al \be} D_\be \V{u}}$.
Component-wise, we have $\pr{\Lz \V{u}}^i = - D_\al \pr{A_{ij}^{\al \be} D_\be u^j}$ for each $i = 1, \ldots, d$.
The second-order operator is written as
$$\LV = \Lz + V,$$
see \eqref{elEqDef}.
Component-wise, $\pr{ \LV \V{u}}^i = -D_\al\pr{A_{ij}^{\al \be} D_\be u^j} + V_{ij} u^j$ for each $i = 1, \ldots, d$.

The transpose operator of $\Lz$, denoted by $\Lz^*$, is defined by $\Lz^* \V{u} = - D_\al \brac{\pr{A^{\al\be}}^* D_\be \V{u}}$, where $\pr{A^{\al \be}}^* = \pr{A^{\be\al}}^T$, or rather $\pr{A_{ij}^{\al\be}}^* = A_{ji}^{\be\al}$.
Note that the adjoint coefficients, $\pr{A_{ij}^{\al\be}}^*$ satisfy the same ellipticity assumptions as $A_{ij}^{\al\be}$ given by \eqref{ellip} and \eqref{Abd}.
Take $V^* = V^T (= V$, since $V$ is assumed to be symmetric).
The adjoint operator to $\LV$ is given by
\begin{align}
\label{el*OpDef}
\LV^* \V{u}
&:=\,  \Lz^* \V{u} + V^* \V{u}
= -D_\al\brac{\pr{A^{\be\al}}^T D_\be \V{u}} + V^T \V{u}.
\end{align}

All operators, $\Lz, \Lz^*, \LV, \LV^*$ are understood in the sense of distributions on $\Omega$.
Specifically, for every $\V{u} \in W^{1,2}_{V}(\Om)$ and $\V{v}\in C_c^\infty(\Om)$, we use the naturally associated bilinear form and write the action of the functional $\LV\V{u}$ on $\V{v}$ as
\begin{align*}
({\LV}\V{u}, \V{v})
=\BV\brac{\V{u}, \V{v}}
&= \int_\Om \innp{A^{\al \be} D_\be \V{u}, D_\al \V{v}} + \innp{V \, \V{u}, \V{v}}
= \int_\Om A_{ij}^{\al \be} D_\be u^j D_\al v^i + V_{ij} u^j v^i.
\end{align*}
It is straightforward to check that for such $\V{v}, \V{u}$ and for the coefficients satisfying \eqref{Abd}, the bilinear form above is well-defined and finite since $V \in L^\frac{n}{2}_{\loc}$.
We explore these details in the next section.
Similarly, $\BV^*\brac{\cdot, \cdot}$ denotes the bilinear operator associated to $\LV^*$, given by
\begin{align*}
(\LV^*\V{u}, \V{v})
=\BV^*\brac{\V{u}, \V{v}}
&= \int \innp{ \pr{A^{\be \al}}^T D_\be \V{u}, D_\al \V{v}} + \innp{V^T \, \V{u}, \V{v}}
= \int A_{ji}^{\be \al} D_\be u^j D_\al v^i + V_{ji} u^j v^i .
\end{align*}
Clearly, $\BV\brac{\V{v},\V{u}}=\, \BV^*\brac{\V{u},\V{v}}$.

For any vector distribution $\V{f}$ on $\Omega$ and $\V{u}$ as above, we always understand ${\mathcal L}\V{u}= \V{f} $ on $\Omega$ in the sense of distributions; that is, as $\mathcal{B}\brac{\V{u},\V{v}}= \V{f}(\V{v})$ for all $\V{v}\in C_c^\infty(\Om)$.
Typically $\V{f}$ will be an element of some $L^\ell(\Omega)$ space and so the action of $\V{f}$ on $\V{v}$ is then simply $\disp \int \V{f}\cdot \V{v}.$
The identity ${\mathcal L}^*\V{u}= \V{f}$ is interpreted similarly.

We define the associated local spaces as
$$\WT{W}^{1,2}_{V, \loc}(\Om) = \{\V{u} \text{ weakly differentiable on } \Om : \|\V{u}\|_{W^{1, 2} _V (\Om')} < \infty \text{ for every } \Om' \Subset \Om\},$$
where the tilde notation here is meant to emphasize that this notion of local differs from the standard one.
Note that ${W}^{1,2}_{V}(\Om) \su \WT{W}^{1,2}_{V, \loc}(\Om)$.
Moreover, the operators and bilinear forms described above may all be defined in the sense of distributions for any $\V{u} \in \WT{W}^{1,2}_{V, \loc}(\Om)$.

\begin{rem}
Given open connected sets $U \subset \Om \subset \Rn$, we can define $W_{V}^{1,2}(U)$ via restriction from $W_{V, 0}^{1,2}(\Om)$.
That is, the space is given by
\begin{equation}
\label{WV12Def}
W_{V}^{1,2}(U)
= \set{\V{u}\rvert_U : \V{u} \in W_{V,0}^{1,2}(\Om)}
\end{equation}
and is equipped with the $W_{V}^{1,2}(U)$-norm.
This viewpoint may be useful when $V$ is only locally integrable or positive definite on proper subsets of $\R^n$.
In particular, this approach could be used in the study of Green's functions (which we do not focus on in this paper).
\end{rem}

\section{Fundamental Matrix Constructions}
\label{FundMat}

We maintain the assumptions from the previous section.
That is, $A^{\al \be}$ is a coefficient matrix that satisfies boundedness \eqref{Abd} and ellipticity \eqref{ellip}, and $V$ is a locally integrable matrix weight that satisfies \eqref{VAssump}.
The elliptic operator $\LV$ is defined formally by \eqref{elEqDef}.
For any open, connected $\Om \su \R^n$, $V$ is used to define the Hilbert spaces $W^{1,2}_{V,0}\pr{\Om}$ and the inner product spaces $W^{1,2}_{V}\pr{\Om} := W^{1,2}_{V,0}\pr{\R^n}\rvert_\Om$.

To justify the existence of fundamental matrices associated to our generalized Schr\"odinger operators, we use the constructions and results presented in \cite{DHM18}.
By the fundamental matrix, we mean the following.

\begin{defn}[Fundamental Matrix]
\label{d3.3}
We say that a matrix function $\Ga^V\pr{x,y}= \pr{\Ga^V_{ij}\pr{x,y}}_{i,j=1}^d$ defined on $\set{\pr{x,y} \in \R^n \times \R^n : x \ne y}$ is the \textbf{fundamental matrix} of $\LV$ if it satisfies the following properties:
\begin{enumerate}
\item[a)] $\Ga^V\pr{\cdot, y}$ is locally integrable and $\LV \Ga^V\pr{\cdot, y} = \de_y I$ for all $y \in \R^n$ in the sense that for every $\V{\phi} = \pr{\phi^1, \ldots, \phi^d}^T \in C^\iny_c\pr{\R^n}^{d}$,
\begin{align*}
&\int_{\R^n} A_{ij}^{\al \be} D_\be \Ga^V_{jk}\pr{\cdot, y} D_\al \phi^i + V_{ij} \Ga^V_{jk}\pr{\cdot, y} \phi^i
= \phi^k\pr{y}.
\end{align*}
\item[b)] For all $y \in \R^n$ and $r > 0$, $\Ga^V(\cdot, y) \in Y^{1,2}\pr{\R^n \setminus B\pr{y, r}}$. \\ 
\item[c)] For any $\V{f} = \pr{f^1, \ldots, f^d}^T \in L^\iny_c\pr{\R^n}$, the
function $\V{u} = \pr{u^1, \ldots, u^d}^T$ given by
$$u^k\pr{y} = \int_{\R^n} \Ga^V_{jk}\pr{x,y} f^j\pr{x} \,dx$$
belongs to $W^{1,2}_{V,0}(\R^n)$ and satisfies $\LV^* \V{u} = \V{f}$ in the sense that for every $\phi = \pr{\phi^1, \ldots, \phi^d}^T \in C^\iny_c\pr{\R^n}^{d}$,
\begin{align*}
&\int_{\R^n} A_{ij}^{\al \be} D_\al u^i D_\be \phi^j + V_{ij}  u^i\phi^j
= \int_{\R^n} f^j \phi^j.
\end{align*}
\end{enumerate}
We say that the matrix function $\Ga^V\pr{x,y}$ is the \textbf{continuous fundamental matrix} if it satisfies the conditions above and is also continuous.
\end{defn}

We restate the following theorem from \cite{DHM18}.
The stated assumptions and properties will be described below.

\begin{thm}[Theorem 3.6 in \cite{DHM18}]
\label{t3.6}
Assume that \rm{\ref{A1} - \ref{A7}} as well as properties {\rm{(IB)}} and {\rm{(H)}} hold.
Then there exists a unique continuous fundamental matrix, $\Ga^{V}(x,y)=\pr{\Gamma^{V}_{ij}(x,y)}_{i,j=1}^d, \,\set{x\ne y}$, that satisfies Definition \ref{d3.3}.
We have $\Ga^V(x,y)= \Ga^{V*}(y,x)^T$, where $\Ga^{V*}$ is the unique continuous fundamental matrix associated to $\LV^*$ as defined in \eqref{el*OpDef}.
Furthermore, $\Ga^V(x,y)$ satisfies the following estimates:
\begin{align}
&\norm{\Ga^V(\cdot, y)}_{Y^{1,2}\pr{\R^n\setminus B(y,r)}}
+ \norm{\Ga^V(x, \cdot)}_{Y^{1,2}\pr{\R^n\setminus B(x,r)}}
\le C r^{1-\frac{n}{2}}, \quad \forall r>0,
\label{eq3.55} \\
&\norm{\Ga^V(\cdot, y)}_{L^q\pr{B(y,r)}}
+ \norm{\Ga^V(x, \cdot)}_{L^q\pr{B(x,r)}}
\le C_q r^{2-n+\frac{n}{q}}, \quad \forall q\in \left[1, \tfrac{n}{n-2}\right), \quad \forall r>0,
\label{eq3.56} \\
& \norm{D \Ga^V\pr{\cdot, y}}_{L^{q}\pr{B(y,r)}}
+ \norm{D \Ga^V\pr{x, \cdot}}_{L^{q}\pr{B(x,r)}}
\le C_q r^{1-n +\frac{n}{q}}, \qquad \forall q \in \left[ 1, \tfrac{n}{n-1}\right), \quad \forall r>0,
\label{eq3.57} \\
& \abs{\set{x \in \R^n : \abs{\Ga^V\pr{x,y}} > \tau}}
+ \abs{\set{y \in \R^n : \abs{\Ga^V\pr{x,y}} > \tau}}
\le C \tau^{- \frac{n}{n-2}}, \quad \forall \tau > 0,
\label{eq3.58} \\
& \abs{\set{x \in \R^n : \abs{D_x \Ga^V\pr{x,y}} > \tau}}
+ \abs{\set{y \in \R^n : \abs{D_y \Ga^V\pr{x,y}} > \tau}}
\le C \tau^{- \frac{n}{n-1}},  \quad \forall \tau >0,
\label{eq3.59} \\
& \abs{\Ga^V\pr{x,y}} \le C \abs{x - y}^{2 - n}, \qquad \forall x \ne y ,
\label{eq3.60}
\end{align}
where each constant depends on $d, n, \La, \la$, and $C_{\rm{IB}}$, and each $C_q$ depends additionally on $q$.
Moreover, for any $0<R\le R_0<|x-y|$,
\begin{align}
&\abs{\Ga^V\pr{x,y} - \Ga^V\pr{z,y}}
\le C_{R_0} C \pr{\frac{|x-z|}{R}}^\eta R^{2-n}
\label{eq3.61}
\end{align}
whenever $|x-z|<\frac{R}{2}$ and
\begin{align}
&\abs{\Ga^V\pr{x,y} - \Ga^V\pr{x,z}}
\le C_{R_0} C \pr{\frac{|y-z|}{R}}^\eta R^{2-n}
\label{eq3.62}
\end{align}
whenever $|y-z|<\frac{R}{2}$, where $C_{R_0}$ and $\eta=\eta(R_0)$ are the same as in assumption {\rm{(H)}}.
\end{thm}

To justify the existence of $\Ga^V$ satisfying Definition \ref{d3.3} and the results in Theorem \ref{t3.6}, it suffices to show that for our Hilbert space $W^{1,2}_{V, 0}\pr{\R^n}$ (and the associated inner product spaces $W^{1,2}_V(\Om)$ where $\Om \su \R^n$), operators $\LV$, $\LV^*$, and bilinear forms $\BV$, $\BV^*$ that were introduced in the previous section, the assumptions \rm{\ref{A1} - \ref{A7}} from \cite{DHM18} hold.

In addition to properties \rm{\ref{A1} - \ref{A7}}, we must also {\em assume} that we are in a setting where de Giorgio-Nash-Moser theory holds.
Therefore, we assume the following interior boundedness (IB) and H\"older continuity (H) conditions:

\begin{itemize}
\item[(IB)]
\label{IB}
We say that \rm{(IB)} holds if whenever $\V{u} \in W^{1,2}_V(B\pr{0, 4R})$ is a weak solution to $\mathcal{L} \V{u} = \V{f}$ or $\mathcal{L}^* \V{u} = \V{f}$ in $B(0,2R)$, for some $R>0$, where $\V{f} \in L^\ell\pr{B(0,2R)}$
for some $\ell \in \pb{ \frac{n}{2}, \iny}$, then for any $q \ge 1$,
\begin{equation}
\norm{\V{u}}_{L^\iny\pr{B(0,R)}}
\le C_{\rm{IB}} \brac{ R^{- \frac n q}\norm{\V{u}}_{L^q\pr{B(0,2R)}} + R^{2 - \frac{n}{\ell}} \|\V{f}\|_{L^\ell\pr{B(0,2R)}}},
\label{eq3.47}
\end{equation}
where the constant $C_{\rm{IB}}>0$ is independent of $R>0$.

\item[(H)]
\label{H}
 We say that \rm{(H)} holds if whenever $\V{u} \in W^{1,2}_V(B(0, 2R_0))$ is a weak solution to $\mathcal{L} \V{u} = \V{0}$ or $\mathcal{L}^* \V{u} = \V{0}$ in $B(0,R_0)$ for some $R_0>0$, then there exists $\eta \in \pr{0, 1}$ and $C_{R_0}>0$, both depending on $R_0$, so that whenever $0 < R \le R_0$,
\begin{align}
\sup_{x, y \in B(0,R/2), x \ne y} \frac{\abs{\V{u}\pr{x} - \V{u}\pr{y}}}{\abs{x - y}^\eta}
&\le C_{R_0} R^{-\eta} \pr{\fint_{B(0,R)} \abs{\V{u}}^{2^*}}^{\frac 1 {2^*}}.
\label{eq3.48}
\end{align}
\end{itemize}

For systems of equations, the assumptions (IB) and (H) may actually fail.
However, for the class of weakly coupled Schr\"odinger systems that are introduced in the next section, we prove that these assumptions are valid.
To establish (IB) in that setting, it suffices to consider $V \in L^{\frac n 2}_{\loc}\pr{\R^n} \cap \ND$, while our validation of (H) requires the stronger assumption that $V \in L^{\frac n 2+}_{\loc}\pr{\R^n} \cap \ND$.
For the simpler, scalar setting, we refer the reader to \cite[Section 5]{DHM18} for such a discussion of validity.
For many of the scalar settings discussed in \cite[Section 5]{DHM18}, one must assume, as is standard, that $V \in L^{\frac n 2 +}_{\loc}\pr{\R^n}$.

Now we proceed to recall and check that \rm{\ref{A1} - \ref{A7}} from \cite{DHM18} hold for our setting.
Since we are working with fundamental matrices, we only need the following conditions to hold when $\Om = \R^n$.
However, we'll show that the assumptions actually hold in the full generality from \cite{DHM18}.

Recall that $V \in L^{\frac n 2}_{\loc}\pr{\R^n} \cap \ND$ and for any $\Om \su \R^n$ open and connected, $W^{1,2}_{V,0}(\Om) = \overline{C^\iny_c(\Om)}^{\norm{\cdot}_{W^{1,2}_V}}$ and
$W^{1,2}_{V}(\Om)$ is defined via restriction as $W^{1,2}_{V}(\Om) = W^{1,2}_{V}(\R^n) \rvert_\Om$.
Moreover, by Proposition \ref{W12V0Properties} and Corollary \ref{W12VProperties}, these spaces consist of weakly-differentiable, vector-valued $L^1_{\loc}$ functions.

\begin{enumerate}[label=A\arabic*)]
\item\label{A1}
\textit{Restriction property: For any $U\subset \Omega$, if $\V{u} \in W^{1,2}_V\pr{\Om}$, then $\V{u}|_U \in W^{1,2}_V(U)$ with $\left\|\V{u}|_U\right\|_{W^{1,2}_V(U)} \le \left\|\V{u}\right\|_{W^{1,2}_V\pr{\Om}}$.}

The restriction property holds by definition.
That is, for any $U \su \Om \su \R^n$, if $\V{u} \in W_V^{1,2}(\Om)$, then there exists $\V{v} \in W^{1,2}_V(\R^n)$ for which $\V{v} \rvert_\Om = \V{u}$.
Since $\V{v} \rvert_U = \V{u} \rvert_U$, then $\V{u} \rvert_U \in W^{1,2}_V(U)$ and $\norm{\V{u}\rvert_U}_{W^{1,2}_V(U)} \le \norm{\V{u}}_{W^{1,2}_V(\Om)}$.

\item\label{A2}
\textit{Containment of smooth compactly supported functions: $C_c^\infty\pr{\Om}$ functions belong to $W^{1,2}_{V}\pr{\Om}$.
The space $W^{1,2}_{V,0}\pr{\Om}$, defined as the closure of $C_c^\infty\pr{\Om}$ with respect to the $W^{1,2}_{V}\pr{\Om}$-norm, is a Hilbert space with respect to some $\|\cdot\|_{W^{1,2}_{V,0}(\Omega)}$ such that $\disp \|\V{u}\|_{W^{1,2}_{V,0}(\Omega)} \simeq \|\V{u}\|_{W^{1,2}_{V}(\Omega)}$ for all $\V{u} \in W^{1,2}_{V,0}(\Omega)$.}

To establish that $C^\iny_c(\Om) \su W^{1,2}_V(\Om)$, we'll show that $C^\iny_c(\Om) \su W^{1,2}_{V, 0}(\Om)$ and $W^{1,2}_{V,0}(\Om) \su W^{1,2}_V(\Om)$.
The first containment follows from the definition: $W_{V,0}^{1,2}(\Om)$ is defined as the closure of $C_c^\infty(\Om)$ with respect to the $W_V^{1,2}(\Om)$-norm and is a Hilbert space with respect to that same norm.
To establish the second containment, let $\V{u} \in W^{1,2}_{V, 0}\pr{\Om}$.
Then there exists $\set{\V{u}_k}_{k = 1}^\iny \su C^\iny_c\pr{\Om}$ such that $\V{u}_k \to \V{u}$ in $W^{1,2}_V(\Om)$.
It follows that $\V{u} \in W^{1,2}_V(\R^n)$ since $\set{\V{u}_k}_{k = 1}^\iny \su C^\iny_c\pr{\R^n}$ and $\V{u}_k \to \V{u}$ in $W^{1,2}_V(\R^n)$.
However, $\V{u} \rvert_\Om = \V{u}$, so we conclude that $\V{u} \in W^{1,2}_V(\Om)$, as required.

\item\label{A3}
\textit{Embedding in $Y^{1,2}_0\pr{\R^n}$: The space $W^{1,2}_{V,0}\pr{\Om}$ is continuously embedded into $Y^{1,2}_0\pr{\Om}$ and respectively, there exists $c_0>0$ such that for any $\V{u} \in W^{1,2}_{V,0}\pr{\Om}$, $\norm{\V{u}}_{Y^{1,2}_0{\pr{\Om}}} \le c_0 \norm{\V{u}}_{W^{1,2}_{V}\pr{\Om}}$}.

Proposition \ref{W12V0Properties} shows that $W^{1,2}_{V,0}(\Om)$ is contained in $Y^{1,2}_0(\Om)$.
In fact, for any $\V{u} \in W^{1,2}_{V,0}(\Om)$, since
\begin{align}
\norm{\V{u}}_{Y^{1,2}_0{(\Om)}} \lesssim \norm{\V{u}}_{W^{1,2}_{V}(\Om)},
\label{A3Check}
\end{align}
then $W^{1,2}_{V,0}(\Om)$ is continuously embedded into $Y^{1,2}_0(\Om)$.
Moreover, a Sobolev embedding implies that for any $\V{u} \in W^{1,2}_{V,0}(\Om)$,
\begin{align*}
\norm{\V{u}}_{L^{2^*}{(\Om)}} \le \norm{D\V{u}}_{L^{2}(\Om)}.
\end{align*}

\item\label{A4}
\textit{Cutoff properties:For any $U\subset \R^n$ open and connected
\begin{equation}
\label{eq2.7}
\begin{array}{c}
\mbox{ $\V{u}\in W^{1,2}_V(\Om)$ and $\xi\in C_c^\infty(U) \quad \Longrightarrow \quad \V{u} \xi\in W^{1,2}_V(\Om \cap U)$,} \\
\mbox{ $\V{u}\in W^{1,2}_V(\Om)$ and $\xi\in C_c^\infty(\Om \cap U) \quad \Longrightarrow \quad \V{u} \xi\in W^{1,2}_{V,0}(\Om \cap U)$,} \\
\mbox{ $\V{u}\in W^{1,2}_{V,0}(\Om)$ and $\xi\in C_c^\infty(\R^n) \quad \Longrightarrow \quad \V{u} \xi\in W^{1,2}_{V,0}(\Om)$.}
\end{array}
\end{equation}
with $\|\V{u} \xi\|_{W^{1,2}_V(\Om\cap U)}\leq C_\xi \, \|\V{u}\|_{W^{1,2}_V(\Om)}$ in the first two cases.}

To establish \eqref{eq2.7}, first let $\V{u} \in W_V^{1,2}(\Om)$.
Since $\V{u} \in W_V^{1,2}(\Om)$, then there exists $\V{v} \in W_{V,0}^{1,2}\pr{\R^n}$ such that $\V{v}|_\Om = \V{u}$.
Moreover, there exists $\set{\V{v}_k}_{k=1}^\iny \su C^\iny_c\pr{\R^n}$ such that $\norm{\V{v}_k - \V{v}}_{W^{1,2}_{V}\pr{\R^n}} \to 0$.
If $\xi \in C^\iny_c\pr{U}$, then $\set{\V{v}_k \xi} \su C^\iny_c\pr{\R^n}$.
We first show that
$$\lim_{k \to \iny }\norm{\V{v}_k \xi - \V{v} \xi}_{W^{1,2}_{V}\pr{\R^n}} = 0.$$
Observe that
\begin{align*}
\norm{\V{v}_k \xi - \V{v} \xi}_{W^{1,2}_{V}\pr{\R^n}}^2
&= \norm{V^{1/2}\pr{\V{v}_k \xi - \V{v} \xi}}_{L^{2}\pr{\R^n}}^2
+ \norm{D\pr{\V{v}_k \xi - \V{v} \xi}}_{L^{2}\pr{\R^n}}^2 \\
&\lesssim \norm{V^{1/2}\pr{\V{v}_k - \V{v}} \xi }_{L^{2}\pr{\R^n}}^2
+ \norm{D\pr{\V{v}_k - \V{v}} \xi }_{L^{2}\pr{\R^n}}^2
+ \norm{\pr{\V{v}_k - \V{v}} D \xi}_{L^{2}\pr{\R^n}}^2 \\
&\lesssim \norm{\V{v}_k - \V{v}}_{W^{1,2}_{V}\pr{\R^n}}^2
+ \norm{\V{v}_k - \V{v}}_{L^{2}\pr{U}}^2,
\end{align*}
where the constants depend on $\xi$.
There is no loss in assuming that $U$ is bounded, so an application of H\"older's inequality shows that
\begin{align*}
\norm{\V{v}_k - \V{v}}_{L^{2}\pr{U}}
&\le \norm{\V{v}_k - \V{v}}_{L^{2^*}\pr{\R^n}} \abs{U}^{\frac{2^*-2}{2^*2}}
\lesssim \norm{\V{v}_k - \V{v}}_{Y^{1,2}_0\pr{\R^n}}.
\end{align*}
Combining the previous two inequalities, then applying \eqref{A3Check}, we see that
\begin{align*}
\norm{\V{v}_k \xi - \V{v} \xi}_{W^{1,2}_{V}\pr{\R^n}}^2
&\lesssim \norm{\V{v}_k - \V{v}}_{W^{1,2}_{V}\pr{\R^n}}^2 \to 0.
\end{align*}
In particular, $\V{v} \xi \in W^{1,2}_{V,0}\pr{\R^n}$.
Since $\xi$ is compactly supported on $U$, then $\pr{\V{v} \xi}\rvert_{\Om \cap U} = \V{v}|_\Om \xi = \V{u} \xi$ and we conclude that $\V{u} \xi \in W^{1,2}_V\pr{\Om \cap U}$.

Now assume that $\xi \in C^\iny_c\pr{\Om \cap U}$.
For each $k \in \N$, define $\V{u}_k = \V{v}_k|_{\Om} \in C^\iny(\Om)$.
Then $\set{\V{u}_k \xi} \su C^\iny_c\pr{\Om \cap U}$.
Since $\norm{\V{v}_k \xi - \V{v} \xi}_{W^{1,2}_{V}\pr{\R^n}} \to 0$, as shown above, then by the restriction property {\rm\ref{A1}}, $\norm{\V{u}_k \xi - \V{u} \xi}_{W^{1,2}_{V}\pr{\Om \cap U}} \to 0$ as well.
It follows that $\V{u} \xi \in W^{1,2}_{V,0}\pr{\Om \cap U}$, as required.

The third line of \eqref{eq2.7} follows immediately from the arguments above.
\end{enumerate}

We require that $\mathcal{B}$ and $\mathcal{B}^*$ can be extended to bounded and accretive bilinear forms on $W^{1,2}_{V,0}(\Om) \times W^{1,2}_{V,0}(\Om)$ so that the Lax-Milgram theorem may be applied in $W^{1,2}_{V,0}(\Om)$.
The next two assumptions capture this requirement.

\begin{enumerate}[label=A\arabic*), resume]

\item\label{A5} {\it Boundedness hypotheses:
There exists a constant $\Ga > 0$ so that $\mathcal{B}\brac{\V{u}, \V{v}} \le \Ga \norm{\V{u}}_{W^{1,2}_{V}} \norm{\V{v}}_{W^{1,2}_{V}}$ for all $\V{u}, \V{v} \in W^{1,2}_{V,0}\pr{\Om}$.}

For any $\V{u}, \V{v} \in W^{1,2}_{V,0}(\Om)$, it follows from \eqref{Abd} that
\begin{align*}
\mathcal{B}\brac{\V{u}, \V{v}}
\le \La \int \innp{D\V{u}, D\V{v}}
+ \int \innp{V\V{u}, \V{v}}
\le \pr{\La + 1} \norm{\V{u}}_{W^{1,2}_{V}} \norm{\V{v}}_{W^{1,2}_{V}},
\end{align*}
so we may set $\Ga = \La + 1$.

\item\label{A6} {\it Coercivity hypotheses:
There exists a $\ga > 0$ so that $\ga \norm{\V{u}}_{W^{1,2}_V}^2 \le \mathcal{B}\brac{\V{u}, \V{u}}$ for any $\V{u} \in W^{1,2}_{V,0}\pr{\Om}$.
}

For any $\V{u} \in W^{1,2}_{V,0}(\Om)$, it follows from \eqref{ellip} that
\begin{align*}
 \la \norm{\V{u}}_{W^{1,2}_V}^2 \le \mathcal{B}\brac{\V{u}, \V{u}},
\end{align*}
so we take $\la = \ga$.

\end{enumerate}

Using \rm{\ref{A1} - \ref{A6}}, we prove the following.

\begin{lem}[Caccioppoli inequality]
\label{l4.1}
Let $\Om \su \R^n$ be open and connected.
Let $\V{u} \in W^{1,2}_V\pr{\Omega}$ and $\zeta \in C^\infty(\R^n)$ with $D\zeta \in C_c^\infty(\R^n)$
be such that $\V{u} \zeta  \in W^{1,2}_{V,0}(\Om)$, $ \partial^i\zeta \,\V{u}\in L^2(\Omega)$, $i=1,...,n$,
and $\disp \mathcal{B}\brac{\V{u}, \V{u} \zeta^2} \le \int \V{f} \, \cdot \V{u}\, \zeta^2$ for some $\V{f} \in L^{\ell}\pr{\Omega}$,  $\ell \in \pb{ \frac n 2, \iny}$.
Then
\begin{equation}
\label{eq4.1}
\int \abs{D \V{u}}^2 \zeta^2 \le C \int \abs{\V{u}}^2 \abs{D \zeta}^2 + c \abs{\int  \V{f}\cdot \V{u} \,\zeta^2 },
\end{equation}
where $C = C\pr{n, \la, \La}$, $c = c\pr{n, \la}$.
\end{lem}

\begin{proof}
Let $\V{u}$, $\zeta$ be as in the statement.
Since $D\zeta \in C_c^\infty(\R^n)$, $\zeta$ is a constant outside some large ball (call it $C_\zeta$) so that $C_\zeta-\zeta \in C_c^\infty(\R^n)$.
Then, by \rm{\ref{A4}}, $\V{u}\zeta^2=C_\zeta \V{u}\zeta-(C_\zeta-\zeta) \V{u}\zeta \in W^{1,2}_{V, 0}\pr{\Omega}$ as well.
A computation shows that
\begin{align*}
\mathcal{B}\brac{ \V{u} \zeta, \V{u} \zeta}
=& \mathcal{B}\brac{\V{u}, \V{u} \zeta^2}
+ \int - \innp{A^{\al \be} D_\be \V{u}, \V{u}} \zeta D_\al \zeta + \innp{A^{\al \be} \V{u}, D_\al \V{u}} \zeta  D_\be \zeta + \innp{A^{\al \be} \V{u}, \V{u}} D_\be \zeta \, D_\al \zeta,
\end{align*}
where \rm{\ref{A5}} and \rm{\ref{A6}} ensure that $\mathcal{B}[\cdot, \cdot]$ is well-defined at all of its inputs.

By hypothesis, we have $\disp \mathcal{B}\brac{\V{u}, \V{u} \zeta^2} \le \abs{ \int \V{f} \cdot \V{u}\,  \zeta^2}$.
By the boundedness assumption in \eqref{Abd},
\begin{align*}
& \int - \innp{A^{\al \be} D_\be \V{u}, \V{u}} \zeta D_\al \zeta + \innp{A^{\al \be} \V{u}, D_\al \V{u}} \zeta  D_\be \zeta + \innp{A^{\al \be} \V{u}, \V{u}} D_\be \zeta \, D_\al \zeta \\
\le& 2 \La \int \abs{D \V{u}} \abs{D \zeta} \abs{\V{u}} \zeta + \La \int \abs{D \zeta}^2 \abs{\V{u}}^2
\le \pr{ \frac{4 \La^2}{\la}  + \La} \int \abs{\V{u}}^2 \abs{D \zeta}^2
+ \frac{\la}{4} \int \abs{D \V{u}}^2 \zeta^2.
\end{align*}
It follows from the inequalities above and the coercivity assumption on $\mathcal{B}$ described by \rm{\ref{A6}} combined with \rm{\ref{A3}} that
\begin{align*}
&\frac{\la}{2} \int \abs{D \V{u}}^2 \zeta^2
- \la \int \abs{\V{u}}^2 \abs{D \zeta}^2
\le \la \int \abs{D\pr{\V{u} \zeta}}^2
\le \la \norm{\V{u} \zeta}_{W^{1,2}_V}^2
\le \mathcal{B}\brac{ \V{u} \zeta, \V{u} \zeta} \\
\le& \pr{ \frac{4 \La^2}{\la}  + \La } \int \abs{\V{u}}^2 \abs{D \zeta}^2
+ \frac{\la}{4} \int \abs{D \V{u}}^2 \zeta^2
+ \abs{ \int \V{f} \cdot \V{u}\,  \zeta^2}.
\end{align*}
The assumptions that $\V{u} \zeta \in W^{1,2}_{V,0}(\Om)$, $ \partial^i\zeta \,\V{u}\in L^2(\Omega)$, $i=1,...,n$, and $D\zeta \in C_c^\infty(\R^n)$ ensure that the first and the second integrals above are finite.
Therefore, we can rearrange to reach the conclusion.
\end{proof}

This gives the final assumption from \cite{DHM18}:

\begin{enumerate}[label=A\arabic*), resume]
\item\label{A7} {\it The Caccioppoli inequality:
If $\V{u} \in W^{1,2}_V\pr{\Om}$ is a weak solution to $\mathcal{L} \V{u} = \V{0}$ in $\Omega$ and $\zeta \in C^\infty(\R^n)$ is such that $D\zeta \in C_c^\infty (\Omega)$, $\zeta \V{u}  \in W^{1,2}_{V,0}\pr{ \Omega}$, and $\partial^i\zeta \,\V{u}\in L^2(\Omega)$, $i=1, ..., n$, then there exists $C = C\pr{n, \la, \La}$ so that
\begin{align*}
\int \abs{D \V{u}}^2 \zeta^2 \le C \int \abs{\V{u}}^2 \abs{D \zeta}^2.
\end{align*}
Note that $C$ is independent of the set on which $\zeta$ and $D\zeta$ are supported.
}
\end{enumerate}

In conclusion, the fundamental solution for the operator $\LV$ defined in \eqref{elEqDef}, denoted by $\Ga^V$, exists and satisfies Definition \ref{d3.3} as well as the properties listed in Theorem \ref{t3.6} whenever we assume that assumptions (IB) and (H) hold for $\LV$.

In fact, given that assumptions \rm{A1)} through \rm{A7)} hold for general $\Om \su \R^n$, not just $\Om = \R^n$, the framework here allows us to also discuss Green's matrices as defined in \cite[Definition 3.9]{DHM18}, for example.
That is, whenever we assume that assumptions (IB) and (H) hold for $\LV$, the results of \cite[Theorem 3.10]{DHM18} hold for the Green's matrix.
As we show in the next section, there are many examples of vector-valued Schr\"odinger operators that satisfy assumptions (IB) and (H).
However, for the boundary boundedness assumption (BB) introduced in \cite[Section 3.4]{DHM18}, it is not clear to us when any vector-valued Schr\"odinger operators satisfy this assumption.
As such, determining whether the global estimates for Green's matrices as described in \cite[Corollary 3.12]{DHM18} hold for operators $\LV$ is an interesting question, but is beyond the scope of this current investigation.

\section{Elliptic Theory for Weakly Coupled Systems}
\label{ellipExamples}

In this section, we introduce a class of elliptic systems called weakly coupled Schr\"odinger operators and show that they satisfy the elliptic theory assumptions from Section \ref{FundMat}.
In particular, these are elliptic systems for which the fundamental matrices that were described in the previous section may be directly proven to exist without having to \textit{assume} that (\rm{IB}) and (\rm{H}) hold.
That is, for the class of weakly coupled Schr\"odinger operators that we introduce in the next paragraph, we prove here that local boundedness and H\"older continuity actually hold.

We introduce the class of {\bf weakly coupled Schr\"odinger operators}.
As above, let the leading coefficients be given by $A^{\al \be} = A^{\al \be}(x)$, where for each $\al, \be \in \set{ 1, \dots, n}$, $A^{\al \be}$ is a $d \times d$ matrix with bounded measurable coefficients.
Here we impose the condition that $A^{\al \be}(x) = a^{\al \be}(x) I_d$, where each $a^{\al \be}$ is scalar-valued and $I_d$ is the $d \times d$ identity matrix.
That is, $A^{\al \be}_{ij}(x) = a^{\al \be}(x) \de_{ij}$.
As usual, we assume that there exist constants $0 < \la, \La < \iny$ so that $A^{\al \be}$ satisfies the ellipticity condition described by \eqref{ellip} and the boundedness condition \eqref{Abd}.
For the zeroth-order term, let $V$ satisfy \eqref{VAssump}.
That is, $V$ is a nondegenerate, symmetric, positive semidefinite $d \times d$ matrix function in $L^{\frac n 2}_{\loc}\pr{\R^n}$.
The equations that we study are formally given by \eqref{elEqDef}.
With our additional conditions on the leading coefficients, the operator takes the component-wise form
\begin{equation}
\label{LVWDefn}
\pr{ \LV \V{u}}^i = -D_\al\pr{A_{ij}^{\al \be} D_\be u^j} + V_{ij} u^j = -D_\al\pr{a^{\al \be} D_\be u^i} + V_{ij} u^j
\end{equation}
for each $i = 1, \ldots, d$.
As discussed in Sections \ref{EllOp} and \ref{FundMat}, weak solutions exists and belong to the space $W_V^{1,2}(\Om)$.
Given the specific structure of the leading coefficient matrix, we refer to these elliptic systems as ``weakly coupled".

We begin with a lemma that will be applied in the Moser boundedness arguments below.

\begin{lem}[Local boundedness lemma]
\label{winW12V}
If $\V{u} \in W^{1,2}_V(B_2)$, then for any $k > 0$, it holds that $\disp\V{w} = \V{w}\pr{k} := \frac{\V{u}}{\sqrt{\abs{\V{u}}^2 + k^2}} \in W^{1,2}_V(B_2)$ as well.
\end{lem}

\begin{proof}
Since $\V{u} \in W^{1,2}_V(B_2)$, then $\V{u}$ is the restriction of an element $\V{u} \in W^{1,2}_V(\R^n)$, so there exists $\set{\V{u}_j}_{j=1}^\iny \su C^\iny_c(\R^n)$ so that $\V{u}_j \to \V{u}$ in $W^{1,2}_V(\R^n)$.
For each $j \in \N$, define $\V{w}_j := \V{u}_j \pr{\abs{\V{u}_j}^2 + k^2}^{-\frac 1 2} \in C^\iny_c(\R^n)$.
We will show that $\V{w}_j \to \V{w}$ in $W^{1,2}_V(B_2)$.

To simplify notation, let $v_j = \pr{\abs{\V{u}_j}^2 + k^2}^{\frac 1 2}$ and $v = \pr{\abs{\V{u}}^2 + k^2}^{\frac 1 2}$.
Observe that
\begin{equation}
\label{wnwDiff}
\begin{aligned}
\V{w}_j - \V{w}
&= \frac{\V{u}_j}{ v_j} - \frac{\V{u}}{v}
= \frac{\V{u}_j - \V{u}}{ v_j}
+ \frac{\V{u}\pr{\abs{v }^2 -\abs{v_j}^2}}{ v_j v\pr{v_j + v}}
= v_j^{-1} \brac{\V{u}_j - \V{u}
+ \V{w} \innp{\V{u} - \V{u}_j, \frac{\V{u} + \V{u}_j}{v + v_j}}}.
\end{aligned}
\end{equation}
Since $\disp \abs{\V{w}}, \abs{\frac{\V{u} + \V{u}_j}{v + v_j}} \le 1$ and $v_j \ge k$ for all $j \in \N$, then $\disp \abs{\V{w}_j - \V{w}} \le 2 k^{-1} \abs{\V{u}_j - \V{u}}$ and it follows from a Sobolev embedding that
\begin{align*}
\pr{\int_{B_2} \abs{ \V{w}_j - \V{w}}^{2^*}}^{\frac{n-2}{n}}
&\lesssim k^{-2} \pr{ \int_{B_2} \abs{ \V{u}_j - \V{u}}^{2^*}}^{\frac{n-2}{n}}
\le k^{-2} \pr{ \int_{\R^n} \abs{ \V{u}_j - \V{u}}^{2^*}}^{\frac{n-2}{n}}
\le k^{-2} c_n \int_{\R^n} \abs{D\V{u}_j - D\V{u}}^2.
\end{align*}
Since $\V{u}_j \to \V{u}$ in $W^{1,2}_V(\R^n)$, then $D\V{u}_j \to D\V{u}$ in $L^2(\R^n)$, so we deduce that both $\V{w}_j \to \V{w}$ in $L^{2^*}(B_2)$ and $\V{u}_j \to \V{u}$ in $L^{2^*}(B_2)$.
Therefore, there exists subsequences $\set{\V{w}_{j_i}}_{i = 1}^\iny$ and $\set{\V{u}_{j_i}}_{i = 1}^\iny$ so that $\V{w}_{j_i} \to \V{w}$ a.e. and $\V{u}_{j_i} \to \V{u}$ a.e.
In particular, we relabel so that $\V{w}_{j} \to \V{w}$ a.e. and $\V{u}_{j} \to \V{u}$ a.e.

From \eqref{wnwDiff} and that $k \le v_j$, we have
\begin{align*}
k^2 \abs{V^{\frac 1 2}\pr{\V{w}_j - \V{w}}}^2
&\le \innp{V\pr{\V{u}_j - \V{u}
+ \V{w} \innp{\V{u} - \V{u}_j, \frac{\V{u} + \V{u}_j}{v + v_j}}},
\V{u}_j - \V{u}
+ \V{w} \innp{\V{u} - \V{u}_j, \frac{\V{u} + \V{u}_j}{v + v_j}}} \\
&= \abs{V^{\frac 1 2}\pr{\V{u}_j - \V{u}}}^2
+ \abs{V^{\frac 1 2} \V{w}}^2 \innp{\V{u} - \V{u}_j, \frac{\V{u} + \V{u}_j}{v + v_j}}^2
+ 2\innp{V\pr{\V{u}_j - \V{u}} , \V{w}} \innp{\V{u} - \V{u}_j, \frac{\V{u} + \V{u}_j}{v + v_j}} \\
&\le  2 \abs{V^{\frac 1 2}\pr{\V{u}_j - \V{u}}}^2
+ 2 \abs{V} \abs{\V{u} - \V{u}_j}^2 \abs{\V{w}}^2 \abs{\frac{\V{u} + \V{u}_j}{v + v_j}}^2 \\
&\le 2 \abs{V^{\frac 1 2}\pr{\V{u}_j - \V{u}}}^2
+ 2 \abs{V} \abs{\V{u}_j - \V{u}}^2 ,
\end{align*}
where we have applied Cauchy-Schwarz and that $\disp \abs{\V{w}}, \abs{\frac{\V{u} + \V{u}_j}{v + v_j}} \le 1$ for all $j \in \N$.
It follows from H\"older and Sobolev inequalities that
\begin{align*}
\int_{B_2} \innp{V\pr{\V{w}_j - \V{w}}, \V{w}_j - \V{w}}
&\le 2  k^{-2} \int_{B_2} \abs{V^{\frac 1 2}\pr{\V{u}_j - \V{u}}}^2
+ 2  k^{-2} \int_{B_2} \abs{V} \abs{\V{u}_j - \V{u}}^2 \\
&\le 2  k^{-2} \int_{B_2} \abs{V^{\frac 1 2}\pr{\V{u}_j - \V{u}}}^2
+ 2  k^{-2} \pr{\int_{B_2} \abs{V}^{\frac n 2}}^{\frac{2}{n}} \pr{\int_{\R^n} \abs{\V{u}_j - \V{u}}^{2^*}}^{\frac{n-2}{n}} \\
&\le 2  k^{-2} \int_{\R^n} \abs{V^{\frac 1 2}\pr{\V{u}_j - \V{u}}}^2
+ 2  k^{-2} c_n \norm{V}_{L^{\frac n 2}(B_2)} \int_{\R^n} \abs{D\V{u}_j - D\V{u}}^2.
\end{align*}
Since $\V{u}_j \to \V{u}$ in $W^{1,2}_V(\R^n)$, this shows that $V^{\frac 1 2}\V{w}_j \to V^{\frac 1 2}\V{w}$ in $L^2(B_2)$, or $\V{w}_j \to \V{w}$ in $L^2_V\pr{B_2}$.

Now we consider the gradient terms.
Since
\begin{align*}
D\V{w}_j
&= \frac{D\V{u}_j}{ v_j} - \frac{\V{u}_j \innp{D \V{u}_j, \V{u}_j}}{ v_j^3}
= v_j^{-1}\brac{ D\V{u}_j - \V{w}_j \innp{D \V{u}_j, \V{w}_j}}
\end{align*}
and analogously for $D\V{w}$, then
\begin{align*}
D\V{w}_j - D\V{w}
&= \frac{D\V{u}_j}{ v_j}  - \frac{D\V{u}}{v} + \frac{\V{w} \innp{D \V{u}, \V{w}}}{v} - \frac{\V{w}_j \innp{D \V{u}_j, \V{w}_j}}{ v_j}
= A_j + B_j,
\end{align*}
where
\begin{align*}
A_j &= v_j^{-1}\brac{ D\V{u}_j- D\V{u}
- \V{w}_j \innp{D \V{u}_j - D \V{u}, \V{w}_j}} \\
B_j &= \pr{\V{w}_j \innp{D \V{u}, \V{w}_j} - D\V{u}} \innp{\frac{\V{u}_j - \V{u}}{v_j v}, \frac{\V{u}_j + \V{u}}{v_j + v}}
+ v^{-1} \brac{ \pr{\V{w} - \V{w}_j} \innp{D \V{u}, \V{w}}
+ \V{w}_j \innp{D \V{u}, \V{w} - \V{w}_j} }.
\end{align*}
This shows that
\begin{align*}
\lim_{j \to \iny} \int_{B_2} \abs{D\V{w}_j - D\V{w} }^2
&\lesssim \lim_{j \to \iny} \int_{B_2} \abs{A_j}^2
+ \lim_{j \to \iny} \int_{B_2} \abs{B_j}^2 .
\end{align*}
Since $v_j, v \ge k$, $\abs{\V{w}_j}, \abs{\V{w}} \le 1$ for all $j \in \N$, and $D\V{u}_j \to D\V{u}$ in $L^2(B_2)$, then
\begin{align*}
\lim_{j \to \iny }\int_{B_2} \abs{A_j}^2
&\lesssim k^{-2} \lim_{j \to \iny}  \int_{B_2} \abs{D\V{u}_j- D\V{u}}^2
= 0.
\end{align*}
On the other hand, since $\abs{B_j} \le 8 k^{-1} \abs{D\V{u}}$ and $\abs{D\V{u}}^2 \in L^1(\R^n)$, then the Lebesgue Dominated Convergence Theorem shows that
\begin{align*}
\lim_{j \to \iny }\int_{B_2} \abs{B_{j}}^2
&= \int_{B_2}  \lim_{j \to \iny } \abs{B_{j}}^2 .
\end{align*}
Because $v_j, v \ge k$, $\abs{\V{w}_j}, \abs{\V{w}} \le 1$ for all $j \in \N$, $\abs{D\V{u}} < \iny$ a.e., $\V{w}_{j} \to \V{w}$ a.e., and $\V{u}_{j} \to \V{u}$ a.e., then $\disp \lim_{j \to \iny} B_{j} = 0$ a.e. and we deduce that
\begin{align*}
\lim_{j \to \iny }\int_{B_2} \abs{B_{j}}^2
&= 0.
\end{align*}
Thus, we conclude that $D\V{w}_j \to D\V{w}$ in $L^2(B_2)$.
In combination with the fact that $\V{w}_j \to \V{w}$ in $L^2_V(B_2)$, we have shown that $\V{w}_j \to \V{w}$ in $W^{1,2}_{V}(B_2)$, as required.
\end{proof}

With the above lemma, we prove local boundedness of solutions to weakly coupled systems.

\begin{prop}[Local boundedness of vector solutions]
\label{MoserBounded}
With $B_r = B(0, r)$, assume that $B_{4R} \su \Om$.
Let $\LV$ be as given in \eqref{LVWDefn}, where $A$ is bounded and uniformly elliptic as in \eqref{Abd} and \eqref{ellip}, and $V$ satisfies \eqref{VAssump}.
Assume that $\V{f} \in L^\ell(B_{2R})$ for some $\ell > \frac n 2$.
Let $\V{u} \in W_{V} ^{1, 2} (B_{4R})$ satisfy $\LV \V{u} = \V{f}$ in the weak sense on $B_{2R}$.
That is, for any $\V{\phi} \in W^{1,2}_{V,0}(B_{2R})$, it holds that
\begin{equation}
 \label{WeakSolDef}
\int_{B_{2R}} a^{\al \be} \innp{D_\be \V{u}, D_\al\V{\phi}} + \int_{B_{2R}} \innp{V \V{u}, \V{\phi}}
= \int_{B_{2R}} \innp{\V{f}, \V{\phi}}.
\end{equation}
 Then, for any $q \geq 1$, we have
 \begin{equation}
 \label{IB1}
 \norm{\V{u}}_{L^\infty(B_{R})} \leq C \brac{R^{-\frac{n}{q}} \norm{\V{u}}_{L^q(B_{2R})} + R^{2 - \frac n \ell} \|\V{f}\|_{L^\ell(B_{2R})}},
 \end{equation}
 where the constant $C$ depends on $n$, $d$, $\la$, $\La$, $q$ and $\ell$.
 \end{prop}

 \begin{rem}
\label{MoserBoundCorRem}
Note that the constant $C$ in Proposition \ref{MoserBounded} is independent of $V$ and $R$.
Therefore, this result establishes that estimate \eqref{eq3.47} in (\rm{IB}) holds for our weakly-coupled systems.
\end{rem}

\begin{rem}
\label{positiveVExplanation}
This statement assumes that $V \in L^{\frac n 2}_{\loc}\pr{\R^n} \cap \ND$, but the proof only uses that $V$ is positive semidefinite.
As described in previous sections, the additional conditions on $V$ ensure that each $W^{1,2}_V\pr{\Om}$ is a well-defined inner product space over which we can talk about weak solutions.
As such, we maintain that $V$ satisfies \eqref{VAssump}.
However, if a different class of solution functions were considered, this proof would carry through assuming only that $V \ge 0$ a.e.
\end{rem}

\begin{rem}
\label{radiiRemark}
There is nothing special about the choice of $R$, $2R$ and $4R$ here except for the ordering $R < 2R < 4R$, the scale of differences $4R - 2R, 2R - R \simeq R$, and that this statement matches that of \eqref{eq3.47} in (IB).
In applications of this result, we may modify the choices of radii while maintaining the ordering and difference properties, keeping in mind that the associated constants will change in turn.
 \end{rem}

\begin{proof}
We assume first that $R = 1$.

For some $k > 0$ to be specified below, define the scalar function
$$v = v\pr{k} := \pr{\abs{\V{u}}^2 + k^2}^{\frac 1 2}$$
and the associated vector function
$$\V{w} = \V{w}\pr{k} := \V{u} \, v^{-1}.$$
Observe that $v \ge k > 0$ and since $v > \abs{\V{u}}$, then $\abs{\V{w}} \le 1$.
In fact, since $v \le \abs{\V{u}} + k$ and $\V{u} \in W^{1,2}_V(B_2)$ implies that $\V{u} \in L^{2}(B_2)$, then $v \in L^2(B_2)$.
Moreover, since $\disp D_\be v = \innp{D_\be \V{u}, \V{w}}$, then $\abs{D_\be v} \le \abs{D_\be \V{u}}$ and we deduce that each $D_\be v \in L^2\pr{B_R}$.
In particular, $v \in W^{1,2}(B_2)$.
An application of Lemma \ref{winW12V} implies that $\V{w} \in W^{1,2}_V(B_2)$.
In particular, since $\V{w}$ and $v^{-1}$ are bounded, then $\disp D_\al \V{w} = \brac{ D_\al\V{u} - \V{w} \innp{D_\al \V{u}, \V{w}}} v^{-1} \in L^2(B_2)$.
Let $\vp \in C^\iny_c(B_2)$ satisfy $\vp \ge 0$ in $B_2$ and note that $D\pr{\V{w} \,\vp} \in L^{2}(B_2)$.
Then
\begin{equation}
\label{vkPDE1}
\begin{aligned}
\int_{B_2} a^{\al \be} D_\be v \, D_\al \vp
&= \int_{B_2} a^{\al \be} \innp{D_\be \V{u}, \V{w}} \, D_\al \vp \\
&= \int_{B_2} a^{\al \be} \innp{D_\be \V{u}, D_\al \pr{\V{w} \vp} }
- \int_{B_2} a^{\al \be} \innp{D_\be \V{u}, v \, D_\al \V{w}} \, v^{-1} \vp.
\end{aligned}
\end{equation}
To simplify the last term, observe that
\begin{align*}
\innp{D_\be \V{u}, v \, D_\al \V{w} }
&= \innp{D_\be \V{u}, D_\al\V{u} - \V{w} \innp{D_\al \V{u}, \V{w}}}
= \innp{D_\be \V{u}, D_\al\V{u}}
- \innp{D_\be \V{u}, \V{w}} \innp{D_\al \V{u}, \V{w}},
\end{align*}
while
\begin{align*}
\innp{v \, D_\be \V{w}, v \, D_\al \V{w}}
&= \innp{D_\be\V{u} - \V{w} \innp{D_\be \V{u}, \V{w}}, D_\al\V{u} - \V{w} \innp{D_\al \V{u}, \V{w}}} \\
&=  \innp{D_\be\V{u}, D_\al\V{u}}
- \pr{1+ k^2v^{-2}}\innp{D_\be\V{u},  \V{w}} \innp{D_\al \V{u},  \V{w}}.
\end{align*}
By combining the previous two expressions, we see that
\begin{align*}
a^{\al \be} \innp{D_\be \V{u}, v \, D_\al \V{w}} v^{-1} \vp
&= a^{\al \be} \innp{v \, D_\be \V{w}, v\, D_\al \V{w}} v^{-1} \vp
+ a^{\al \be} \innp{D_\be\V{u},  \V{w}} \innp{D_\al \V{u},  \V{w}} k^2v^{-3} \vp,
\end{align*}
where all of the terms in these expressions are integrable since $D \V{u}, v \, D\V{w} \in L^2(B_2)$ and $\V{w}, v^{-1}, \vp, a^{\al \be} \in L^\iny(B_2)$.
Then \eqref{vkPDE1} becomes
\begin{equation*}
\begin{aligned}
\int_{B_2} a^{\al \be} D_\be v \, D_\al \vp
&= \int_{B_2} a^{\al \be} \innp{D_\be \V{u}, D_\al \pr{\V{w} \vp} } \\
&- \int_{B_2} a^{\al \be} \innp{D_\be \V{w}, D_\al \V{w}} \, v \vp
- \int_{B_2} a^{\al \be} \innp{D_\be\V{u}, \V{w}} \innp{D_\al \V{u}, \V{w}} \, k^2 v^{-3} \vp   \\
&\le \int_{B_2} a^{\al \be} \innp{D_\be \V{u}, D_\al \pr{\V{w} \vp} },
\end{aligned}
\end{equation*}
where we have used that $a^{\al \be}$ is elliptic to eliminate the last two terms.

Since $\vp \in C^\iny_c(B_2)$ and $\V{w} \in W^{1,2}_V(B_2)$, then A4) implies that
$\disp \V{\phi} := \V{w} \vp= \frac{\V{u}\vp}{v} \in W^{1,2}_{V,0}(B_2)$, so we may use \eqref{WeakSolDef} with $\disp \V{\phi}$ to get
\begin{align*}
\int_{B_2} a^{\al \be} \innp{D_\be \V{u}, D_\al \pr{\V{w} \vp}}
&= \int_{B_2} \innp{\V{f}, \V{w}} \vp
- \int_{B_2} \innp{V \V{u}, \V{u}} \frac{\vp}{v}
\le \int_{B_2} \innp{\V{f}, \V{w}} \vp,
\end{align*}
since $V$ is positive semidefinite.
By setting $F = \innp{\V{f}, \V{w}} \in L^\ell(B_2)$ and combining the previous two inequalities, we see that for any $\vp \in C^\iny_c(B_2)$ with $\vp \ge 0$, it holds that
\begin{align}
\label{WeakSubHarmonIneq}
\int_{B_2} a^{\al \be} D_\be v \, D_\al \vp
&\le \int_{B_2} F \vp.
\end{align}
An application of Lemma \ref{densityLemma} implies that \eqref{WeakSubHarmonIneq} holds for any $\vp \in W^{1,2}_0\pr{B_1}$ with $\vp \ge 0$ a.e., and we then have that $- \di\pr{A \gr v} \le F$ in the standard weak sense on $B_2$.
Then \cite[Theorem 4.1]{HL11}, for example, shows that for any $q \ge 1$,
\begin{align*}
\norm{\V{u}}_{L^\iny(B_1)}
\le \norm{v}_{L^\iny(B_1)}
\le C \brac{ \norm{v}_{L^q(B_2)} + \norm{F}_{L^\ell(B_2)}}
\le C \brac{ \norm{\V{u}}_{L^q(B_2)} + \norm{k}_{L^q(B_2)}  + \norm{F}_{L^\ell(B_2)}},
\end{align*}
where $C$ depends on $n, d, \la, \La, q, \ell$.
With $q = 2$, the righthand side is finite and therefore $\V{u} \in L^\iny(B_1)$.
Setting $\disp k = \frac{\norm{\V{u}}_{L^\iny(B_1)}}{2C \abs{B_2}^{\frac 1 q}}$, noting that $\disp \norm{F}_{L^\ell(B_2)} \le \|\V{f}\|_{L^\ell(B_2)}$, we get
\begin{align*}
\norm{\V{u}}_{L^\iny(B_1)}
\le 2 C \brac{ \norm{\V{u}}_{L^q(B_2)} + \|\V{f}\|_{L^\ell(B_2)}},
\end{align*}
as required.

For the general case of $R \ne 1$, we rescale.
That is, we apply the result above with $\V{u}_R(x) = \V{u}(Rx)$, $A_R(x) = A(Rx)$, $V_R (x) = R^2 V(Rx),$  and $\V{f}_R (x) = R^2 \V{f} (Rx)$ to get \eqref{IB1} in general.
\end{proof}

\begin{lem}
\label{densityLemma}
Let $C^\iny_c(\Om)^+ = C^\iny_c(\Om) \cap \set{\vp : \vp \ge 0 \text{ in } \Om}$ and let $W^{1,2}_0(\Om)^+ = W^{1,2}_0(\Om) \cap \set{u : u \ge 0 \text{ a.e. in } \Om}$.
For any ball $B \su \R^n$, $C^\iny_c(B)^+$ is dense in $W^{1,2}_0(B)^+$.
\end{lem}

\begin{proof}
Assume that $B = B_1$ and let $u \in W^{1,2}_0(B_1)^+$.
Let $\phi \in C_c^\iny(B_1)$ be a standard mollifier and set $\phi_t = t^{-n} \phi\pr{\frac{\cdot}{t}} \in C_c^\iny(B_t)$.
For every $k \in \N$, define $v_k = \phi_{k^{-1}} \ast u \in C^\iny_c\pr{B_{1 + k^{-1}}}$, then set $u_k = v_k\pr{\pr{1 + k^{-1}}\cdot} \in C^\iny_c\pr{B_1}$.
Since $u \in W^{1,2}_0(B_1)^+$, then $u_k \ge 0$ in $B_1$ so that $\set{u_k}_{k = 1}^\iny \su C^\iny_c\pr{B_1}^+$.
The aim is to show that $u_k \to u$ in $W^{1,2}(B_1)$.

Let $\eps > 0$ be given.
Since $u$ may be extended by zero to a function defined on all of $\R^n$, then regarding $u \in L^2\pr{\R^n}$, there exists $g = g_\eps \in C_c\pr{\R^n}$ such that $\norm{u - g}_{L^2\pr{\R^n}} < \eps$.
As $g$ is uniformly continuous, then there exists $K \in \N$ so that $\norm{g\pr{\pr{1 + k^{-1}} \cdot} - g}_{L^2\pr{\R^n}} < \eps$ whenever $k \ge K$.
By extending all functions to $\R^n$, we see that
\begin{align*}
\norm{u_k - u}_{L^2(B_1)}
&\le \norm{v_k\pr{\pr{1 + k^{-1}} \cdot} - u\pr{\pr{1 + k^{-1}} \cdot}}_{L^2(\R^n)}
+ \norm{u\pr{\pr{1 + k^{-1}} \cdot} - g\pr{\pr{1 + k^{-1}} \cdot}}_{L^2(\R^n)} \\
&+ \norm{g\pr{\pr{1 + k^{-1}} \cdot} - g}_{L^2(\R^n)}
+ \norm{g - u}_{L^2(\R^n)} \\
&= \frac{\norm{v_k- u}_{L^2(\R^n)} + \norm{u- g}_{L^2(\R^n)}}{\pr{1 + k^{-1}}^{\frac n 2} } + \norm{g\pr{\pr{1 + k^{-1}} \cdot} - g}_{L^2(\R^n)}
+ \norm{g - u}_{L^2(\R^n)} .
\end{align*}
Since $v_k \to u$ in $L^2(B_2)$, then there exists $M \in \N$ so that $\norm{v_k - u}_{L^2(\R^n)} < \eps$ whenever $k \ge M$.
In particular, if $k \ge \max\set{K, M}$, then $\norm{u_k - u}_{L^2(B_1)}  < 4 \eps$, so we deduce that $u_k \to u$ in $L^2\pr{\R^n}$ and hence in $L^2\pr{B_1}$.

Since $\gr v_k = \phi_{k^{-1}} \ast \gr u$ in $B_{1 + k^{-1}}$, then an analogous argument shows that $\gr u_k \to \gr u$ in $L^2\pr{B_1}$, completing the proof.
\end{proof}

The next main result of this section is the following H\"older continuity result.

\begin{prop}[H\"older continuity]
\label{HolderContThm}
With $B_r = B(0, r)$, assume that $B_{2R_0} \su \Om$.
Let $\LV$ be as given in \eqref{LVWDefn}, where $A$ is bounded and uniformly elliptic as in \eqref{Abd} and \eqref{ellip}, and $V$ $\in L^{\frac n 2 +}_{\loc}\pr{\R^n} \cap \ND$.
Assume that $\V{u} \in W_{V} ^{1, 2}(B_{2R_0})$ is a weak solution to $\LV \V{u} =  0$ in $B_{3R_0/2}$.
Then there exist constants $c_1, c_2, c_3 > 0$, all depending on $n$, $p$, $\la$, and $\La$, such that if
$$\eta := -c_1  \brac{\log \pr{\min\set{  \frac {c_2} {R_0}\|V\|_{L^p(B_{R_0})} ^{-\frac{1}{2  - \frac{n}{p}}}, c_3, \frac 1 2 }}}^{-1} \in \pr{ 0, 1},$$
then for any $R \le R_0$,
\begin{equation}
\label{HolderCont}
\sup_{\substack{x, y \in B_{R/2} \\ x \neq y}} \frac{\abs{\V{u} (x) - \V{u}(y)}}{|x - y|^\eta}
\leq 4 R^{-\eta} \|\V{u}\|_{L^\infty(B_R)}.
\end{equation}
In fact, for any $q \ge 1$, there exists $c_4\pr{n, q}$ so that
\begin{equation}
\label{HolderContp}
\sup_{\substack{x, y \in B_{R/2} \\ x \neq y}} \frac{\abs{\V{u} (x) - \V{u}(y)}}{|x - y|^\eta}
\leq c_4 R^{-\eta - \frac n q} \|\V{u}\|_{L^q(B_{3R/2})}.
\end{equation}
\end{prop}

\begin{rem}
\label{HolderRem}
We point out that the assumption on $V$ in this proposition is stronger than in previous statements.
First, we now need $V \in L^{\frac n 2 +}_{\loc}$, as opposed to $V \in L^{\frac n 2}_{\loc}$, in order to apply the Harnack inequality -- a crucial step in the proof.
Second, the assumption that $V$ is positive semidefinite is used in the application of Lemma \ref{MoserBounded}.
Finally, the full power of $V \in \ND$ is used to ensure that the spaces $W_{V} ^{1, 2}$ are well-defined.
However, if we were to use the spaces $W^{1,2}$ or $Y^{1,2}$ in place of $W_{V} ^{1, 2}$ to define our weak solutions, it may be possible to drop the requirement that $V \in \ND$ and establish \eqref{HolderContp} by only assuming that $V$ is positive semidefinite.
In fact, if we knew a priori that the weak solution is bounded (and therefore didn't need to resort to Lemma \ref{MoserBounded}), we could prove \eqref{HolderCont} by only assuming that $V \in L^{\frac n 2 +}_{\loc}$ without imposing that $V$ is positive semidefinite anywhere.
\end{rem}

\begin{rem}
\label{HolderRemark}
Although the choice radii in this statement do not match those in the statement of the H\"older continuity assumption \eqref{eq3.48} from (H), this presentation suits the proof well.
As usual, the the radii may be modified to give precisely \eqref{eq3.48} from (H), but we will not do that here.
\end{rem}

The proof was inspired by the arguments in \cite{Caf82} (see also \cite{Pin06} for a more detailed account of this method), which proves H\"{o}lder continuity for very general nonlinear elliptic systems.
To prove this result, we will carefully iterate the following lemma.

\begin{lem}[Iteration lemma]
\label{itLemma}
Let $B_r = B(0, r)$.
Let $\rho \le 1$ and $\V{\nu}_* \in \R^d$ with $|\V{\nu}_*| \leq 2$.
Let $\LV$ be as given in \eqref{LVWDefn}, where $A$ is bounded and uniformly elliptic as in \eqref{Abd} and \eqref{ellip}, and $V \in L^{\frac n 2 +}_{\loc}\pr{\R^n} \cap \ND$.
Assume that $\V{u} \in W_{V} ^{1, 2}(B_{3\rho})$ is a weak solution to $\LV \V{u} =  -V \V{\nu}_*$ in $B_{2\rho}$.
If $\|\V{u}\|_{L^\infty(B_\rho)} \leq M \leq 1$, then there exists $\delta = \delta(n, p, \la, \La)  \in (0, 1)$ and a universal constant $c_0 > 0$ so that for any $0 < \theta \le 1$, it holds that
\begin{equation}
\label{CaffIneq1}
\sup_{x \in B_{\te\rho/2}}\abs{\V{u}(x) - \delta a_{B_{3\te\rho/4}} \V{u}} \leq M (1 - \delta) + c_0 M^{\frac 1 2} \pr{\te\rho}^{1-\frac{n}{2p}} \|V\|_{L^p(B_1)}^{\frac 1 2}.
\end{equation}
\end{lem}

Before we prove this lemma, let us briefly discuss its application.
To run the arguments in \cite{Caf82}, we look at functions of the form $\V{u}_k = \V{u} - \V{\nu}_k$ for constant vectors $\V{\nu}_k$ to be inductively selected, where here $\LV\V{u} = 0$.
However, we then have
\begin{align*}
\LV\V{u}_k
&= - \di \pr{A \gr \V{u}_k} + V \V{u}_k
= - \di \pr{A \gr \V{u}} + V \pr{\V{u} - \V{\nu}_k}
= \LV\V{u}  - V \V{\nu}_k
= - V \V{\nu}_k.
\end{align*}

\begin{proof}
For some constant vector $\hat{\nu} \in \mathbb{S}^{n-1} \cup \set{\V{0}}$ to be determined later, set
$$h(x) = \frac12 M^2 + M - \frac12 \abs{\V{u}(x)}^2 - \innp{\V{u}(x), \hat{\nu}} \ge 0.$$
Since $\V{u} \in W^{1,2}_V(B_{2\rho}) \cap L^\iny(B_\rho)$ by assumption, then $h \in W^{1,2}(B_{\rho}) \cap L^\iny(B_\rho)$.
For any $\vp \in C^\iny_c\pr{B_{\rho}}^+$, it holds that
\begin{align*}
\int_{B_{\rho}} a^{\al \be} D_\be h \, D_\al \vp
&= - \int_{B_{\rho}} a^{\al \be} \innp{D_\be \V{u}, \V{u}}  \, D_\al \vp
- \int_{B_{\rho}} a^{\al \be} \innp{D_\be\V{u}, \hat{\nu}} \, D_\al \vp \\
&= - \int_{B_{\rho}} a^{\al \be} \innp{D_\be \V{u}, D_\al \brac{\vp\pr{\V{u} + \hat{\nu}}} }
+ \int_{B_{\rho}} a^{\al \be} \innp{D_\be \V{u}, D_\al \V{u}}  \, \vp.
\end{align*}
Since $\vp\pr{\V{u} + \hat{\nu}} \in W^{1,2}_{V,0}(B_{\rho})$ by \eqref{eq2.7} in A4) and $\LV \V{u} =  -V \V{\nu}_*$ weakly in $B_\rho$, then
\begin{align*}
\int_{B_{\rho}} a^{\al \be} D_\be h \, D_\al \vp
&= \int_{B_{\rho}} \innp{V \V{u}, \vp\pr{\V{u} + \hat{\nu}}}
+ \int_{B_{\rho}} \innp{V \V{\nu}_*, \vp\pr{\V{u} + \hat{\nu}}}
+ \int_{B_{\rho}} a^{\al \be} \innp{D_\be \V{u}, D_\al \V{u}}  \, \vp \\
&= \int_{B_{\rho}} \brac{ \innp{V \V{u}, \V{u} } + \innp{V \V{u}, \hat{\nu}} + \innp{V \V{\nu}_*, \V{u} }+ \innp{V \V{\nu}_*, \hat{\nu}}} \vp
+ \int_{B_{\rho}} a^{\al \be} \innp{D_\be \V{u}, D_\al \V{u}}  \, \vp \\
&\ge - 5 \int_{B_{\rho}} \abs{V} \vp,
\end{align*}
since $\abs{\V{u}}, \abs{\hat{\nu}} \le 1$, $\abs{\V{\nu}_*} \le 2$, and $A$ is elliptic.
An application of Lemma \ref{densityLemma} implies that
$-\di\pr{A \gr h} \geq - 5 |V| $ weakly on $B_{\rho}$.
By the weak Harnack inequality \cite[Theorem $4.15$]{HL11}, since $\abs{V} \in L^p\pr{B_{\rho}}$ for some $p > \frac n 2$, then there exists $\de_1(n, p, \la, \La) > 0$ so that
$$\de_1 a_{B_{3\rho/4}} h
\le \inf_{B_{\rho/2}} h + 5 \rho^{2-\frac{n}{p}} \|V\|_{L^p(B_\rho)}
\le \inf_{B_{\rho/2}} h + 5 \rho^{2-\frac{n}{p}} \|V\|_{L^p(B_1)} .$$
Since $\frac12 (M^2 - \abs{\V{u}}^2) \geq 0$, then for any $x \in B_{\rho/2}$,
$$\de_1\brac{M - \innp{ a_{B_{3\rho/4}} \V{u}(x), \hat{\nu}}  }
\le \de_1 a_{B_{3\rho/4}} h
\le \frac12 M^2 + M - \frac12 \abs{\V{u}(x)}^2 - \innp{\V{u}(x), \hat{\nu}} + 5 \rho^{2-\frac{n}{p}} \|V\|_{L^p(B_1)}.$$
Now we fix $x \in B_{\rho/2}$ and set
$\hat{\nu} = \left\{ \begin{array}{ll} \V{u}(x)/\abs{\V{u}(x)} & \text{ if } \V{u}(x) \neq 0 \\ \V{0} & \text{ otherwise} \end{array}\right.$, $r(x)  = \abs{\V{u}(x)}/M$, and define $\te = \te(x)$ so that $\innp{a_{B_{3\rho/4}} \V{u}, \V{u}(x)}  = \abs{\V{u}(x)} \abs{ a_{B_{3\rho/4}} \V{u}} \cos \theta$.
Then the previous inequality may be written as
$$M \de_1 \pr{1 - \frac{\abs{ a_{B_{3\rho/4}} \V{u}} }{M}\cos \theta }
\le M \pr{1 - r(x)}\brac{\frac M2 \pr{1 + r(x)} + 1} + 5 \rho^{2-\frac{n}{p}} \|V\|_{L^p(B_1)}.$$
Since $M\brac{\frac12 M \pr{1 + r(x)}  + 1} \le M(M + 1)$ and $M \le 1$, then
$$\frac{\de_1}{2} \pr{1 - \frac{\abs{ a_{B_{3\rho/4}} \V{u}} }{M}\cos \theta }
\le \frac{\de_1}{M + 1} \pr{1 - \frac{\abs{ a_{B_{3\rho/4}} \V{u}} }{M}\cos \theta }
\le 1 - r(x) + \frac{5 \rho^{2-\frac{n}{p}}}{M\pr{M + 1}} \|V\|_{L^p(B_1)}.$$
Since $r(x) \in \brac{0, 1}$ and $\de_1 \le 2$, then $\frac{\de_1}{2}\pr{1 - r(x)} \le 1 - r(x)$ or, equivalently, $\frac{\de_1}{2} \le 1 - r(x) + r(x) \frac{\de_1}{2}$.
Therefore, it follows from the previous inequality that
\begin{align*}
\frac{\de_1}{2} \pr{1 - r(x) \frac{\abs{ a_{B_{3\rho/4}} \V{u}} }{M}\cos \theta }
&\le 1 - r(x) + r(x) \frac{\de_1}{2}\pr{1 - \frac{\abs{ a_{B_{3\rho/4}} \V{u}} }{M}\cos \theta } \\
&\le 1 - r(x) + r(x) \brac{1 - r(x) + \frac{5 \rho^{2-\frac{n}{p}}}{M\pr{M + 1}} \|V\|_{L^p(B_1)}}.
\end{align*}
Rearranging this equation and using that $r(x) \le 1$, we see that
\begin{align*}
r(x)^2 - \frac{\de_1}{2} r(x) \frac{\abs{ a_{B_{3\rho/4}} \V{u}} }{M} \cos \theta
&\le 1 - \frac{\de_1}{2}  + \frac{5 \rho^{2-\frac{n}{p}}}{M} \|V\|_{L^p(B_1)}.
\end{align*}
Since $ \frac{\abs{ a_{B_{3\rho/4}} \V{u}} }{M} \leq 1$, with $\delta := \frac{\de_1}{4}$, this implies that
\begin{align*}
\frac{1}{M^2} \abs{\V{u}(x) - \delta a_{B_{3\rho/4}} \V{u}} ^2
&= r(x)^2 - 2 \delta r(x) \frac{\abs{a_{B_{3\rho/4}} \V{u} }}{M} \cos \theta + \pr{\de \frac{\abs{a_{B_{3\rho/4}} \V{u}}}{M}} ^2 \\
&\leq 1 - 2 \delta + \de^2 + \frac{5 \rho^{2-\frac{n}{p}}}{M} \|V\|_{L^p(B_1)}
= \pr{1 - \de}^2 + \frac{5 \rho^{2-\frac{n}{p}}}{M} \|V\|_{L^p(B_1)}.
\end{align*}
As this inequality holds for any $x \in B_{\rho/2}$, it follows that
\begin{equation}
\label{CaffIneq2}
\sup_{x \in B_{\rho/2}}\abs{\V{u}(x) - \delta a_{B_{3\rho/4}} \V{u}}^2 \leq M^2 (1 - \delta)^2 + 5M \rho^{2-\frac{n}{p}} \|V\|_{L^p(B_1)}.
\end{equation}
Taking a square root gives \eqref{CaffIneq1} when $\te = 1$.
As all of the above arguments still hold with $\rho$ replaced by $\theta \rho$ for any $0 < \theta < 1$, we get \eqref{CaffIneq1} in general.
\end{proof}

This lemma is used to recursively define a sequence of functions and constant vectors and establish bounds for them.

\begin{lem}[Sequence lemma]
\label{seqLemma}
Let $B_r = B(0, r)$.
Let $\LV$ be as given in \eqref{LVWDefn}, where $A$ is bounded and uniformly elliptic as in \eqref{Abd} and \eqref{ellip}, and $V\in L^{\frac n 2 +}_{\loc}\pr{\R^n} \cap \ND$.
Assume that $\V{u} \in W_{V} ^{1, 2}(B_{3})$ is a weak solution to $\LV \V{u} =  0$ in $B_{2}$.
Assume further that $\|\V{u}\|_{L^\infty(B_{1})} \leq 1$.
Let $\de = \de\pr{n, p, \la, \La} \in \pr{0, 1}$ and $c_0 > 0$ be as in Lemma \ref{itLemma}.
Recursively define the sequences $\set{\V{\nu}_k}_{k=0}^\iny \su \R^d$ and $\set{\V{u}_k}_{k=0}^\iny \su W^{1,2}_V(B_1)$ as follows:
Let $\V{\nu}_0 = \V{0}$, $\V{u}_0(x) = \V{u}(x)$, and for all $k \in \Z_{\ge 0}$, set
\begin{align*}
&\V{\nu}_{k+1} = \V{\nu}_{k} + \delta a_{B_{3/2 \pr{\theta/ 2}^{k+1}}} \V{u}_k \\
&\V{u}_{k+1}(x) = \V{u}(x) - \V{\nu}_{k+1}.
 \end{align*}
If $\theta \le \min\set{\pr{\frac{ \delta}{2 c_0 \norm{V}_{L^p(B_1)}^{1/2}} }^\frac{1}{1-\frac{n}{2p}}, 2\pr{1-\frac{\delta}{2}}^\frac{1}{2-\frac{n}{p}}, 1}$, then for all $k \in \N$, it holds that
\begin{equation} \label{CaffInd1}
\begin{aligned}
& |\V{\nu}_k|  \leq \delta  \sum_{i = 0}^{k - 1} \pr{1 - \frac{\delta}{2}}^i  \\
& \sup_{x \in  B_{\pr{\theta/2}^k}}\abs{\V{u}_k (x)} \leq  \pr{1 - \frac{\delta}{2}}^k.
\end{aligned}
\end{equation}
\end{lem}

\begin{proof}
Since $\theta \le 1$, then Lemma \ref{itLemma} is applicable with $\rho = 1$, $M = 1$, and $\V{\nu}_* = \V{0}$, so from \eqref{CaffIneq1} we get that
$$\sup_{x \in B_{\te/2}}\abs{\V{u}(x) - \delta a_{B_{3\te/4}} \V{u}}
\leq (1 - \delta) + c_0 {\te}^{1-\frac{n}{2p}} \|V\|_{L^p(B_1)}^{\frac 1 2}.$$
Since $\disp \te \le \pr{\frac{ \delta}{2 c_0 \norm{V}_{L^p(B_1)}^{1/2}} }^\frac{1}{1-\frac{n}{2p}}$ implies that $c_0 \theta ^{1-\frac{n}{2p}}  \|V\|_{L^p (B_1)} ^\frac12 \leq \frac{\delta}{2}$, then
$$\sup_{x \in B_{\te/2}}\abs{\V{u}(x) - \delta a_{B_{3\te/4}} \V{u}}
\leq \pr{1 - \frac \delta 2}.$$

By defining $\V{\nu}_1 = \delta a_{B_{3\te/4}} \V{u}$ and $\V{u}_1 = \V{u}(x) - \V{\nu}_1$, we get that $\|\V{u}_1\|_{L^\infty (B_{\theta/2})} \leq  \pr{1 - \frac{\delta}{2}}$, $\abs{\V{\nu}_1} \leq \delta $, and $\LV\V{u}_1 = -V\V{\nu}_1$.
Thus, we may apply Lemma \ref{itLemma} again, this time with $\V{u} = \V{u}_1$, $\rho = \te/2$, $M = 1 - \frac \de 2$, and $\V{\nu}_* = \V{\nu}_1$.
An application of \eqref{CaffIneq1} gives
$$\sup_{x \in B_{\pr{\te/2}^2}}\abs{\V{u}_1(x) - \delta a_{B_{3\te^2/8}} \V{u}_1}
\leq \pr{1 - \frac \de 2} (1 - \delta) + c_0 \pr{1 - \frac \de 2}^{\frac 1 2} \pr{\frac{\te}2}^{1-\frac{n}{2p}} \te^{1-\frac{n}{2p}} \|V\|_{L^p(B_1)} ^\frac12.$$
Since $\theta \le 2\pr{1-\frac{\delta}{2}}^\frac{1}{2-\frac{n}{p}}$, we have
\begin{equation}
\label{Theta}
\pr{\frac{\theta}{2}} ^{1 - \frac{n}{2p}} \leq \pr{1 - \frac{\delta}{2} }^{\frac 1 2}
\; \text{  and  } \;
c_0 \theta ^{1-\frac{n}{2p}}  \|V\|_{L^p (B_1)} ^\frac12 \leq \frac{\delta}{2},
\end{equation}
where the second bound is as above.
It follows that
\begin{equation*}
\sup_{x \in B_{\pr{\te/2}^2}}\abs{\V{u}_1(x) - \delta a_{B_{3\te^2/8}} \V{u}_1}
\leq  (1 - \delta)\pr{1 - \frac{\delta}{2}} + \frac{ \delta}{2} \pr{1 - \frac{\delta}{2}}
= \pr{1 - \frac{\delta}{2}}^2.
\end{equation*}
We now define
\begin{align*}
&\V{\nu}_2
= \V{\nu}_{1} + \delta a_{B_{3\te^2/8}} \V{u}_1
\\
&\V{u}_2(x)
= \V{u}(x) - \V{\nu}_2
= \V{u}_1(x) - \delta a_{B_{3\te^2/8}} \V{u}_1
\end{align*}
and we have  $\|\V{u}_2\|_{L^\infty \pr{B_{\pr{\theta/4}^2}}} \leq  \pr{1 - \frac{\delta}{2}}^2$ and $\abs{\V{\nu}_2} \le \abs{\V{\nu}_{1}} + \delta \abs{a_{B_{3\te^2/8}} \V{u}_1} \leq \delta + \delta \pr{1 - \frac{\delta}{2}}$ since $\frac{3\te^2}8 \le \frac \te 2$ and $\|\V{u}_1\|_{L^\infty (B_{\theta/2})} \leq  \pr{1 - \frac{\delta}{2}}$.

Notice that we have already proved \eqref{CaffInd1} for $k = 0, 1, 2$, so we now prove it for $k > 2$ via induction.
Assume that \eqref{CaffInd1} holds for some $k \ge 2$, and then
\begin{align*}
|\V{\nu}_{k+1}|
&\leq  |\V{\nu}_{k}| + \delta \abs{a_{B_{3/2 \pr{\theta/ 2}^{k+1}}} \V{u}_k}
\leq \delta  \sum_{i = 0}^{k - 1} \pr{1 - \frac{\delta}{2}}^i + \delta \norm{\V{u}_k}_{L^\iny\pr{B_{3/2 \pr{\theta/ 2}^{k+1}}}} \\
&\leq \delta  \sum_{i = 0}^{k - 1} \pr{1 - \frac{\delta}{2}}^i + \delta \norm{\V{u}_k}_{L^\iny\pr{B_{\pr{\theta/ 2}^{k}}}}
\leq \delta  \sum_{i = 0}^{k - 1} \pr{1 - \frac{\delta}{2}}^i + \delta \pr{1 - \frac \de 2}^k
= \de \sum_{i = 0}^{k } \pr{1 - \frac{\delta}{2}}^i.
 \end{align*}
Since
$$\abs{\V{\nu}_k} < \delta  \sum_{i = 0}^{\infty} \pr{1 - \frac{\delta}{2}}^i  = 2,$$
then an application of Lemma \ref{itLemma} with $\V{u} = \V{u}_k$, $\rho = \pr{\te/2}^k$, $M = \pr{1 - \frac \de 2}^k$, and $\V{\nu}_* = \V{\nu}_k$ gives us
\begin{align*}
\sup_{x \in B_{\pr{\te/2}^{k+1}}}\abs{\V{u}_k(x) - \delta a_{B_{3/2\pr{\te/2}^{k+1}}} \V{u}_k}
&\leq \pr{1 - \frac \de 2}^k (1 - \delta) + c_0 \pr{1 - \frac \de 2}^{\frac k 2} \te^{1-\frac{n}{2p}} \pr{\frac \te 2}^{k-\frac{k n}{2p}} \|V\|_{L^p(B_1)} ^\frac12.
\end{align*}
Combining this with \eqref{Theta} and that $\V{u}_{k+1} = \V{u}_k - \delta a_{B_{3/2 \pr{\theta/ 2}^{k+1}}} \V{u}_k$ shows that
\begin{equation*}
\sup_{x \in B_{\pr{\te/2}^{k+1}}}\abs{\V{u}_{k+1}(x)}
\leq \pr{1 - \frac{\delta}{2}}^k (1 - \delta) +  \pr{1 - \frac{\delta}{2}}^k \pr{\frac{\delta}{2}}  = \pr{1 - \frac{\delta}{2}}^{k+1},
\end{equation*}
which completes the proof of \eqref{CaffInd1}.
\end{proof}

Using Lemma \ref{seqLemma}, we give the proof of Proposition \ref{HolderContThm}.

\begin{proof}[Proof of Proposition \ref{HolderContThm}]
Assume first that $R_0 = 2$.
Then $\V{u} \in W_{V} ^{1, 2}(B_{4})$ is a weak solution to $\LV \V{u} =  0$ in $B_{3}$.
An application of Proposition \ref{MoserBounded} with modified radii (see Remark \ref{radiiRemark}) implies that $\V{u} \in L^\iny(B_2)$.

For any $R \le 2$ and $x_0 \in B\pr{0, \frac R 2}$, since $B\pr{x_0, \frac R 2} \su B(0, 2)$, then $\V{u} \in L^\iny\pr{B\pr{x_0, \frac R 2}}$.
Define
$$\V{u}_R(x) = \V{u}\pr{x_0 + \frac R 2 x} / \norm{\V{u}}_{L^\iny\pr{B\pr{x_0, \frac R 2}}}$$
and note that $\|\V{u}_R\|_{L^\infty(B_1)} = 1$.
Since $\V{u} \in W_{V} ^{1, 2}(B_4)$ is a weak solution to $\LV \V{u} =  0$ in $B_3$ by hypothesis, then it holds that $\V{u}_R \in W_{V} ^{1, 2}(B_3)$ is a weak solution to $\mathcal{L}_{V_R} \V{u}_R =  0$ in $B_2$, where $V_R(x) = \pr{\frac R 2}^2 V\pr{x_0 + \frac R 2 x}$.
Because $\norm{V_R}_{L^p(B_1)} = \pr{\frac R 2}^{2 - \frac n p} \norm{V}_{L^p\pr{B\pr{x_0, \frac R 2}}} \le \norm{V}_{L^p(B_2)}$, then with $\de > 0$ is as in Lemma \ref{itLemma}, we set
\begin{equation}
\label{thetaDefn}
\theta = \min\set{\pr{\frac{ \delta}{2 c_0 }}^\frac{1}{1-\frac{n}{2p}} \norm{V}_{L^p(B_2)} ^{-\frac{1}{2-\frac{n}{p}}}, 2\pr{1-\frac{\delta}{2}}^\frac{1}{2-\frac{n}{p}}, 1}.
\end{equation}
It follows that the hypotheses of Lemma \ref{seqLemma} are satisfied for any such $\V{u}_R$.

Define $\eta = \log \pr{1 - \frac{\delta}{2}} \brac{\log \pr{\frac \te 2}}^{-1} > 0$ so that $\pr{1 - \frac{\delta}{2}} = \pr{\frac{\theta}{2}} ^\eta$.
Since $\de \in \pr{0, 1}$, then $\te \le 1 < 2\pr{1-\frac{\delta}{2}}$, so that $\eta \le 1$.
Observe that since $\delta \in (0, 1)$, then $\pr{\frac 2 \theta}^\eta = \pr{1 - \frac{\delta}{2}} ^{-1}< 2$.

For $0 < r \leq 1$, choose $k \in \Z_{\ge 0}$ so that
$$\pr{\frac{\theta}{2}} ^{k+1} < r  \leq \pr{\frac{\theta}{2}}^k.$$
With any $x_0 \in B(0, 1)$, let $\disp \underset{r}{\osc} \, \V{u}_R = \sup_{x, y \in B_r } \abs{\V{u}_R (x) - \V{u}_R(y)}$ and observe that with the notation from Lemma \ref{seqLemma}, we have
\begin{align*}
\underset{r}{\osc} \, \V{u}_R
&\le \sup_{x, y \in B_{\pr{\theta/2}^k} } \abs{\V{u}_R(x) - \V{u}_R(y)}
= \sup_{x, y \in B_{\pr{\theta/2}^k} } \abs{\pr{\V{u}_R(x) - \nu_k} - \pr{\V{u}_R(y) - \nu_k} } \\
&= \sup_{x, y \in B_{\pr{\theta/2}^k} } \abs{\V{u}_{R,k} (x) - \V{u}_{R,k}(y)}.
\end{align*}
It then follows from an application of \eqref{CaffInd1} in Lemma \ref{seqLemma} that
\begin{equation*}
\underset{r}{\osc} \, \V{u}_R
\le 2 \pr{1 - \frac \de 2}^{k}
= 2  \pr{\frac2{\theta}}^\eta \pr{\frac{\theta}{2}}^{\eta(k+1)}
\le 4  r ^\eta.
\end{equation*}

Take $x, y \in B\pr{0, \frac R 2}$ and set $\tilde{r} = \frac{|x-y|}{2R} < \frac 1 2$.
For any $c > 1$, it holds that $\pm \frac{x - y}{2R} \in B(0, c\tilde{r})$, so we choose $c \in (1, 2]$ for which $c \tilde{r} \le 1$.
Then we have
\begin{align*}
\abs{\V{u}_R\pr{\frac{x - y}{2R}} - \V{u}_R\pr{\frac{y - x}{2R}}}
\leq \underset{c \tilde r}{\osc} \, \V{u}_R
\leq 4  (c\tilde{r}) ^\eta
\le 4 \pr{\frac{|x-y|}{R}}^\eta.
\end{align*}
With $x_0 = \frac 1 2 (x + y) \in B\pr{0, \frac R 2}$, we have $\disp \V{u}_R\pr{\frac{x - y}{2R}} = {\V{u}\pr{x_0 + R\frac{x - y}{2R}}}/{\norm{\V{u}}_{L^\iny(B(x_0, \frac R 2))}} = \V{u}(x)/\norm{\V{u}}_{L^\iny(B(x_0, \frac R 2))}$ and $\disp \V{u}_R\pr{\frac{y - x}{2R}} = \V{u}(y)/\norm{\V{u}}_{L^\iny(B(x_0, \frac R 2))}$.
Therefore,
\begin{align*}
\abs{\V{u}\pr{x} - \V{u}(y)}
&= \abs{\V{u}_R\pr{\frac{x - y}{2R}} - \V{u}_R\pr{\frac{y - x}{2R}}} \norm{\V{u}}_{L^\iny(B(x_0, \frac R 2))}
\le 4 \pr{\frac{|x-y|}{R}}^\eta \norm{\V{u}}_{L^\iny(B(0, R))}.
\end{align*}
Since $x, y \in B\pr{0, \frac R 2}$ were arbitrary, the proof of \eqref{HolderCont} is complete for any $R \le 2 = R_0$.
Estimate \eqref{HolderContp} follows from \eqref{HolderCont} and an application of Proposition \ref{MoserBounded} with a modified choice of radii (again).

As usual, the case of $R_0 \ne 2$ follows from a scaling argument.
With $V_{R_0}(x) = \pr{\frac {R_0}2}^2 V\pr{\frac{R_0}2 x}$, we have
$$\norm{V_{R_0}}_{L^p(B_2)}^{-\frac{1}{2  - \frac{n}{p}}}
= \brac{\pr{\frac{R_0}2}^{2 - \frac n p} \norm{V}_{L^p(B_{R_0})}}^{-\frac{1}{2  - \frac{n}{p}}}
= \frac 2 {R_0} \norm{V}_{L^p(B_{R_0})}^{-\frac{1}{2  - \frac{n}{p}}},$$
so the definition of $\te$ in \eqref{thetaDefn} changes accordingly.
\end{proof}

Propositions \ref{MoserBounded} and \ref{HolderContThm} (after modifying the choice of radii) show that assumptions (\rm{IB}) and (\rm{H}) hold for any operator in the class of weakly coupled Schr\"odinger systems.
Accordingly, the results of Section \ref{FundMat} hold for all such elliptic systems.
That is, the fundamental matrices associated to these systems exist and satisfy Definition \ref{d3.3} as well as the statements in Theorem \ref{t3.6}.

Finally, we point out that many of these results may be extended to weakly coupled elliptic systems with nontrivial first-order terms.
Since we do not consider such operators in this article, we don't include those details.

\section{Upper Bounds}
\label{UpBds}

We now prove an exponential decay upper bound for the fundamental matrix associated to our elliptic operator.
Going forward, the elliptic operator $\LV$ is given by \eqref{elEqDef}, where the matrix $A$ satisfies ellipticity and boundedness as described by \eqref{ellip} and \eqref{Abd}, respectively.
For the zeroth order term, we assume now that $V \in \Bp \cap \ND \cap \NC$ for some $p \ge \frac n 2$.
As pointed out in Remark \ref{VAssumpRem}, the assumption that $V \in \Bp \cap \ND$ for some $p \ge \frac n 2$ implies that \eqref{VAssump} holds.
Therefore, this current setting is more restrictive than that of the last three sections.
Since $V \in \Bp$ for some $p \ge \frac n 2$, then Lemma \ref{GehringLemma} implies that $V \in L^{\frac n 2+}_{\loc}$.
This is meaningful since the H\"older continuity results for weakly coupled systems given in Proposition \ref{HolderContThm} hold in this setting.
As such, there is no loss in assuming that $p > \frac n 2$ and we will do that throughout.
We impose the additional condition that $V \in \Bp \cap \NC$ so that we may apply the Fefferman-Phong inequality described by Lemma \ref{FPml}.
We also require that assumptions {\rm{(IB)}} and {\rm{(H)}} hold so that we can meaningfully discuss our fundamental matrices and draw conclusions about them.

We follow the general arguments of \cite{She99}.
Our first lemma is as follows.

\begin{lem}[Upper bound lemma]
\label{MainUpperBoundLem}
Let $\LV$ be given by \eqref{elEqDef}, where $A$ satisfies \eqref{ellip} and \eqref{Abd}, and $V \in \Bp \cap \ND \cap \NC$ for some $p > \frac n 2$.
Let $B \su \R^n$ be a ball.
Assume that $\V{u} \in W_{V}^{1, 2}(\R^n \backslash B)$ is a weak solution to $\LV \V{u} = 0$ in $\R^n \backslash B$.
Let $\phi \in C_c ^\infty (\R^n)$ satisfy $\phi = 0$ on $2B$ and let $g \in C^1(\R^n)$ be a nonnegative function satisfying $|\nabla g(x)| \lesssim_{(n, p, C_V)} \um(x, V)$ for every $x \in \R^n$.
Then there exists $\eps_0$, $C_0$, both depending on $d, n, p, C_V, N_V, \la, \La$, such that whenever $\eps \in \pr{0, \eps_0}$, it holds that
\begin{equation*}
\inrn \um(\cdot, V)^2 \abs{\phi \V{u}}^2  e^{2 \epsilon g} \, \le C_0 \inrn |\V{u}|^2 |\nabla \phi|^2 e^{2\epsilon g} .
\end{equation*}
\end{lem}

The proof is a modification of the proof of Proposition $6.5$ in \cite{MP19}.

\begin{proof}
Since $\V{u} \in W_{V}^{1, 2}(\R^n \backslash B)$, then by the definition of $W^{1,2}_V(\Om)$, there exists $\V{v} \in W_{V,0}^{1, 2}(\R^n)$ such that $\V{v}\rvert_{\R^n \setminus B} = \V{u}$.
Define the function $\V{\psi}  = \phi e^{\epsilon g} \V{v} = f \V{u}$.
By a modification to the arguments in {\ref{A4}}, since $\phi \in C_c^\infty(\R^n)$ and $g \in C^1\pr{\R^n}$, it holds that $\V{\psi} \in W_{V,0}^{1, 2}(\R^n)$.
A similar argument shows that $f^2 \V{u} \in W^{1,2}_{V,0}\pr{\R^n}$ as well.

We adopt notation used in \cite{Amb15}: For $\R^d$-valued functions $\V{u}$ and $\V{v}$, and for a scalar function $f$, we write
$$A \, D\V{u} \, D\V{v} = A_{ij} ^{\alpha \beta} D_\beta u^i D_\alpha v^j, \qquad  (\V{u} \otimes \nabla f )_{i\beta} = u^i D_\beta f.$$
By uniform ellipticity \eqref{ellip}, we have
\begin{equation*}
\inrn \la \abs{D \V{\psi}}^2  + \innp{V \V{\psi}, \V{\psi}}
\le \inrn A D\V{\psi} D\V{\psi} +  \innp{V \V{\psi}, \V{\psi}}.
\end{equation*}
Using that
$$D\V{\psi} = D(f \V{u}) =  \V{u} \otimes \nabla f + f D\V{u},$$
we get
\begin{align}
\inrn \la \abs{D \V{\psi}}^2  + \innp{V \V{\psi}, \V{\psi}}
 \lesssim& \inrn A (\V{u} \otimes \nabla f)(\V{u} \otimes \nabla f)
+ A (\V{u} \otimes \nabla f) (f D \V{u})
+ A (f D\V{u}) (\V{u} \otimes \nabla f) \nonumber  \\
&+ \int_{\R^n} A (f D\V{u}) (f D\V{u})
+  \innp{V\V{u}, f^2 \V{u}}.
\label{MainUpperBoundLemEst1}
\end{align}
Since $\V{v} \in W^{1,2}_{V,0}\pr{\R^n}$, $\V{v}\rvert_{\R^n \setminus B} = \V{u} \in W^{1,2}_V\pr{\R^n \setminus B}$ is a weak solution away from $B$, and $f^2 \V{u} \in W^{1,2}_{V,0}\pr{\R^n}$ is supported away from $B$, then
$$\mathcal{B}\brac{\V{u}, f^2 \V{u}} = 0.$$
That is,
\begin{align*}
0 & =  \inrn A \, D\V{u} \, D(f^2\V{u})
+ \innp{V\V{u}, f^2 \V{u}}
=\inrn2 A\, D\V{u} (f \V{u} \otimes \nabla f  )
+ A \, D\V{u} \, f^2 D \V{u}
+ \innp{V\V{u}, f^2 \V{u}}.
\end{align*}
Plugging this into \eqref{MainUpperBoundLemEst1} gives
\begin{align}
\inrn \abs{D \V{\psi}}^2  + \innp{V \V{\psi}, \V{\psi}}
&\lesssim_{(\la)} \inrn A (\V{u} \otimes \nabla f)(\V{u} \otimes \nabla f)
+ A (\V{u} \otimes \nabla f) (f D \V{u})
- A (f D\V{u}) (\V{u} \otimes \nabla f).
\label{MainUpperBoundLemEst2}
\end{align}

Now we obtain an upper bound for the right hand side of \eqref{MainUpperBoundLemEst2}.
Using the boundedness of $A$ from \eqref{Abd} and Cauchy's inequality, we get that for any $\de' > 0$,
$$\abs{A (\V{u} \otimes \nabla f)(f D\V{u})} \le \La |\V{u}| |f| |D\V{u}| |\nabla \V{f}|
\leq \de' |f|^2 |D\V{u}|^2 + \frac{C(\La)}{\de'} |\V{u}|^2 |\nabla f|^2 $$
and similarly
$$\abs{A (f D\V{u}) (\V{u} \otimes \nabla f)} \le \de' |f|^2 |D\V{u}|^2 + \frac{C(\La)}{\de'} |\V{u}|^2 |\nabla f|^2,$$
while
$$|A (\V{u} \otimes \nabla f)(\V{u} \otimes \nabla f)| \le \La |\V{u}|^2 |\nabla f|^2.$$
Then with $\de \simeq_{(\la)} \de'$, we see that
\begin{equation}
\label{MainUpperBoundLemEst3}
\inrn \abs{D \V{\psi}}^2  + \innp{V \V{\psi}, \V{\psi}}
\le \delta \inrn |f|^2 |D\V{u}|^2  + C\pr{\de, \la, \La} \inrn  |\V{u}|^2 |\nabla f|^2.
\end{equation}
Since $\V{\psi} = f \V{u}$, then
\begin{align*}
|D\V{\psi}| ^2
&= \innp{\V{u} \otimes \nabla f + f D\V{u}, \V{u} \otimes \nabla f + f D\V{u}}_{\text{tr}} \\
&= |f|^2 |D\V{u}|^2 + 2 f \innp{D\V{u}, \V{u} \otimes \nabla f}_{\text{tr}} + |\V{u}|^2  | \nabla f|^2.
\end{align*}
The Cauchy inequality implies that
\begin{align*}
\inrn |f|^2 |D\V{u}|^2
&\leq \inrn |D\V{\psi}|^2 + |\V{u}|^2  | \nabla f|^2 + 2 |f| |D\V{u}| | \V{u} \otimes \nabla f| \\
&\leq \inrn |D\V{\psi}|^2 + |\V{u}|^2  | \nabla f|^2 + \frac12  |f|^2 |D\V{u}|^2 +  2 |\V{u}|^2 |\nabla f|^2.
\end{align*}
Since $\disp \inrn |f|^2 |D\V{u}|^2 < \iny$, then we can absorb the third term into the left to get
\begin{align*}
\inrn|f|^2 |D\V{u}|^2
& \leq \inrn 2 |D\V{\psi}|^2 + 6 |\V{u}|^2  | \nabla f|^2.
\end{align*}
Plugging this expression into \eqref{MainUpperBoundLemEst3} shows that
\begin{equation*}
\inrn \abs{D \V{\psi}}^2  + \innp{V \V{\psi}, \V{\psi}}
\le 2 \delta  \inrn |D\V{\psi}|^2
+ C\pr{\de, \la, \La} \inrn |\V{u}|^2 | \nabla f|^2.
\end{equation*}
Setting $\de = \frac 1 3$, we see that
\begin{equation}
\label{psiNormBound}
\inrn \abs{D \V{\psi}}^2  + \innp{V \V{\psi}, \V{\psi}}
\le C(\la, \La) \inrn |\V{u}|^2 | \nabla f|^2.
\end{equation}
To apply Lemma \ref{FPml} to $\V{\psi}$, we require that $\V{\psi} \in C^1_0\pr{\R^n}^d$, so we use a limiting argument.
Since $\V{\psi} \in W^{1,2}_{V,0}\pr{\R^n}$, then {\ref{A2}} gives that there exists $\set{\V{\psi}_k}_{k = 1}^\iny \su C^\iny_c\pr{\R^n}^d$ for which $\V{\psi}_k \to \V{\psi}$ in $W^{1,2}_{V,0}\pr{\R^n}$.
Moreover, since $\V{\psi}_k \to \V{\psi}$ in $L^{2^*}\pr{\R^n}$ (as shown in the proof of Proposition \ref{W12V0Properties}), there exists a subsequence that converges a.e. to $\V{\psi}$.
After relabeling the indices, we may assume that $\V{\psi}_k \to \V{\psi}$ a.e. and in $W^{1,2}_{V,0}\pr{\R^n}$.
Fatou's Lemma followed by Lemma \ref{FPml} applied to $\V{\psi}_k \in C^\iny_c\pr{\R^n}$ gives
\begin{align*}
\inrn |\V{\psi}|^2  \um(\cdot, V)^2
&\le \liminf_{k \to \iny} \inrn |\V{\psi_k}|^2  \um(\cdot, V)^2 \\
&\le \liminf_{k \to \iny}  c_1 \pr{ \inrn \abs{D\V{\psi_k}}^2 + \innp{V \V{\psi}_k, \V{\psi}_k} }
= c_1 \pr{ \inrn \abs{D\V{\psi}}^2 + \innp{V \V{\psi}, \V{\psi}} },
\end{align*}
where the last line uses convergence in $W^{1,2}_{V,0}\pr{\R^n}$ and $c_1 = c_1\pr{d, n, p, C_V, N_V}$.
Combining this inequality with \eqref{psiNormBound} shows that
\begin{align*}
\inrn |\phi \V{u}|^2  \um(\cdot, V)^2 e^{2 \epsilon g}
&
\le c_1 \pr{ \inrn \abs{D\V{\psi}}^2 + \innp{V \V{\psi}, \V{\psi}} }
 \le c_2 \inrn |\V{u}|^2 | \nabla f|^2 \, dx \\
& \le 2 c_2 \epsilon^2 \inrn |\nabla g|^2 e^{2 \epsilon g} |\phi \V{u}|^2
+ 2 c_2 \inrn |\V{u}|^2 |\nabla \phi|^2 e^{2\epsilon g}  \\
& \le 2 c_3 \epsilon^2 \inrn \um(\cdot, V)^2  |\phi \V{u}|^2 e^{2 \epsilon g}
+ 2 c_2 \inrn |\V{u}|^2 |\nabla \phi|^2 e^{2\epsilon g},
\end{align*}
where $c_2 = c_2\pr{d, n, p, C_V, N_V, \la, \La}$ and $c_3 = c_3\pr{d, n, p, C_V, N_V, \la, \La}$.
If $\eps$ is sufficiently small, we may absorb the first term on the right into the left, completing the proof.
\end{proof}

\begin{rem}
\label{lemHypChanges}
This proof uses that $\V{u}, f^2 \V{u} \in W^{1,2}_{V,0}\pr{\R^n}$ in order to make sense of the expression $\mathcal{B}\brac{\V{u}, f^2 \V{u}}$.
It also uses that $f \in C^1_0\pr{\R^n}$ and $\V{u} \in W^{1,2}_{V,0}\pr{\R^n}$ to say that $f D\V{u} \in L^2\pr{\R^n}$.
Other assumptions on $\V{u}$ would still allow these arguments to carry through.
More specifically, we can apply Lemma \ref{MainUpperBoundLem} with $\V{u} \in Y^{1,2}_{\loc}\pr{\R^n}$ and $f \in C^1_0\pr{\R^n}$.
To see this, let $\supp f \su \Om$ where $\overline{\Om} \Subset \R^n$ and observe that since
\begin{align*}
\mathcal{B}\brac{\V{u}, f^2 \V{u}}
&= \int_{\R^n} \innp{A^{\al \be} D_\be \V{u}, D_\al \pr{f^2 \V{u}}}
+ \innp{V \, \V{u}, f^2 \V{u}} \\
&= \int_\Om f^2 \innp{A^{\al \be} D_\be \V{u}, D_\al \V{u}}
+ 2 f \innp{A^{\al \be} D_\be \V{u}, \V{u} \, D_\al f }
+ f^2 \innp{V \, \V{u}, \V{u}}
\end{align*}
then by applications of H\"older's inequality,
\begin{align*}
\abs{\mathcal{B}\brac{\V{u}, f^2 \V{u}}}
&\le \La \norm{f}_{L^\iny\pr{\Om}}^2 \norm{D \V{u}}_{L^2\pr{\Om}}^2
+ 2 \La \norm{f}_{L^\iny\pr{\Om}} \norm{D \V{u}}_{L^2\pr{\Om}} \norm{\V{u}}_{L^{2^*}\pr{\Om}} \norm{D f}_{L^n\pr{\Om}} \\
&+ \norm{f}_{L^\iny\pr{\Om}}^2 \norm{V}_{L^{\frac n 2}\pr{\Om}} \norm{\V{u}}_{L^{2^*}\pr{\Om}}^2 \\
&\le 2 \La \norm{f}_{L^\iny\pr{\Om}}^2 \norm{D \V{u}}_{L^2\pr{\Om}}^2
+ \pr{ \norm{D f}_{L^n\pr{\Om}}^2+ \norm{f}_{L^\iny\pr{\Om}}^2 \norm{V}_{L^{\frac n 2}\pr{\Om}}} \norm{\V{u}}_{L^{2^*}\pr{\Om}}^2 \\
&\lesssim_{\pr{\La, \norm{f}}}  \norm{\V{u}}_{Y^{1,2}\pr{\Om}}^2.
\end{align*}
In particular, this shows that $\mathcal{B}\brac{\V{u}, f^2 \V{u}}$ is well-defined and finite.
Moreover, since $D \V{u} \in L^2_{\loc}\pr{\R^n}$ and $\overline{\supp f }\Subset \R^n$, then $f D\V{u} \in L^2\pr{\R^n}$.
\end{rem}

\begin{rem}
\label{constantLVRem}
Going forward, we say that a constant $C$ depends on $\LV$ to mean that $C$ has the same dependence as the constants in Theorem \ref{t3.6}.
That is, $C = C\pr{\LV}$ means that $C$ depends on the package $d, n, \la, \La$, and $C_{\rm{IB}}$.
\end{rem}

We now prove our upper bound.

\begin{thm}[Exponential upper bound]
\label{UppBoundThm}
Let $\LV$ be given by \eqref{elEqDef}, where $A$ satisfies \eqref{ellip} and \eqref{Abd}, and $V \in \Bp \cap \ND \cap \NC$ for some $p > \frac n 2$.
Assume that {\rm{(IB)}} and {\rm{(H)}} hold.
Let $\Ga^V(x, y)$ denote the fundamental matrix of $\LV$ and let $\eps_0$ be as given in Lemma \ref{MainUpperBoundLem}.
For any $\eps < \eps_0$, there exists $C = C(\LV, p, C_V, N_V, \eps)$ so that for all $x, y \in \R^n$,
\begin{equation*}
\abs{\Ga^V(x, y)} \leq \frac{C e^{-\epsilon \ud(x, y, V)}}{|x-y|^{n-2}}.
\end{equation*}
\end{thm}

The following proof is similar to that of \cite[Theorem 6.7]{MP19}.

\begin{proof}
Fix $x, y \in \R^n$ with $x \neq y$.
If $\ud\pr{x, y, V} \lesssim_{(n, p, C_V)} 1$, then $e^{-C\eps} \le e^{-\eps \ud\pr{x, y, V}}$, so the result follows from \eqref{eq3.60} in Theorem \ref{t3.6}.
Therefore, we focus on $x, y \in \R^n$ for which $\ud\pr{x, y, V} \gtrsim_{(n, p, C_V)} 1$.
By Lemma \ref{closeRemark}, we can assume $|x - y| > \frac{C}{\um(x, V)}$ since otherwise $\ud(x, y, V) \lesssim_{(n, p, C_V)} 1$.
Likewise, we can assume $|x - y| > \frac{C}{\um(y, V)}$.
Finally, we can assume
\begin{equation}
\label{MPEq6.10}
B\pr{x, \frac{4}{\um(x, V)}} \cap B\pr{y, \frac{4}{\um(y, V)}} = \emptyset
\end{equation}
for if not, then the triangle inequality shows that
$$|x - y| \leq 8 \max \set{ \frac{1}{\um(x, V)}, \frac{1}{\um(y, V)}} $$
so that again $\ud(x, y, V) \lesssim_{(n, p, C_V)} 1$.

Let $r = \frac{1}{\um(y, V)}$ and pick $M > 0$ large enough so that $B(y, 4r) \subseteq B(0, M)$.
Let $\phi \in C_c ^\infty(\R^n)$ be such that $\phi \equiv 0$ on $B(y, 2r), \phi \equiv 1$ on $B(0, M) \backslash B(y, 4r), \phi \equiv 0$ on $\R^n \backslash B(0, 2M)$,
\begin{equation*}
|\nabla \phi| \leq \frac{1}{4r} \text{ on } B(y, 4r) \backslash B(y, 2r)
\text{ and }
|\nabla \phi| \leq \frac{2}{M} \text{ on } B(0, 2M) \backslash B(0, M).
\end{equation*}
The next step is to apply Lemma \ref{MainUpperBoundLem}.
We take $B = B(y, r)$, $\V{u}$ to be each of the individual columns of  $\Ga^V(\cdot, y)$, and $g = \varphi_{V, j}\pr{\cdot, y}$, where $\varphi_{V, j} \in C^\iny(\R^n)$ is as in Lemma \ref{RegLem1}.
Since $\Ga^V\pr{\cdot, y} \in Y^{1,2}\pr{\R^n \setminus B}$, then it can be shown that $\phi^2 \Ga^V\pr{\cdot, y} e^{2\epsilon \varphi_{V, j}\pr{\cdot, y}} \in Y^{1,2}_0\pr{\R^n \setminus B}$.
Since $C^\iny_c\pr{\R^n \setminus B}$ is dense in $Y^{1,2}_0\pr{\R^n \setminus B}$, then the expression $\mathcal{B}\brac{\Ga^V\pr{\cdot, y} , \phi^2 \Ga^V\pr{\cdot, y} e^{2\epsilon \varphi_{V, j}\pr{\cdot, y}}}$ is meaningful and equals zero, see Definition \ref{d3.3}(a).
Moreover, since $\Ga^V\pr{\cdot, y} \in Y^{1,2}\pr{\R^n \setminus B}$, then $\phi D \Ga^V\pr{\cdot, y} e^{\epsilon \varphi_{V, j}\pr{\cdot, y}} \in L^2\pr{\R^n}$.
In particular, according to Remark \ref{lemHypChanges}, we can apply Lemma \ref{MainUpperBoundLem}.
Doing so, we see that for any $\eps < \eps_0$,
\begin{align*}
\int_{B(0, M) \backslash B(y, 4r)} & \um(\cdot, V) ^2 | {\Ga^V}(\cdot, y)|^2 e^{2\epsilon \varphi_{V, j}\pr{\cdot, y}}
\leq \inrn \um(\cdot, V) ^2 |\phi  \Ga^V(\cdot, y)|^2 e^{2\epsilon \varphi_{V, j}\pr{\cdot, y}} \\
&\le C_0 \inrn | {\Ga^V}(\cdot, y)|^2  |\nabla \phi |^2  e^{2\epsilon \varphi_{V, j}\pr{\cdot, y}} \\
&\le \frac{C_0}{r^2} \int_{B(y, 4r) \backslash B(y, 2r)}  | {\Ga^V}(\cdot, y)|^2 e^{2\epsilon \varphi_{V, j}\pr{\cdot, y}}
+ \frac{C_0}{M^2} \int_{B(0, 2M) \backslash B(0, M)} | {\Ga^V}(\cdot, y)| ^2 e^{2\epsilon \varphi_{V, j}\pr{\cdot, y}}.
\end{align*}
For each fixed $j$, $\varphi_{V, j}\pr{\cdot, y}$ is bounded on $\R^n$.
Applying \eqref{eq3.60} then shows that
\begin{equation*}
\frac{1}{M^2} \int_{B(0, 2M) \backslash B(0, M)} | {\Ga^V}(\cdot, y)| ^2 e^{2\epsilon \varphi_{V, j}\pr{\cdot, y}}
\lesssim M^{n-2} (M - |y|)^{4 - 2n} \rightarrow 0 \text{ as } M \rightarrow \infty,
\end{equation*}
and so
\begin{equation}
\label{MPEq6.12}
\int_{\R^n \backslash B(y, 4r)} \um(\cdot, V) ^2 | {\Ga^V}(\cdot, y)|^2 e^{2\epsilon \varphi_{V, j}\pr{\cdot, y}}
\le \frac{C_0}{r^2} \int_{B(y, 4r) \backslash B(y, 2r)}  | {\Ga^V}(\cdot, y)|^2 e^{2\epsilon \varphi_{V, j}\pr{\cdot, y}}.
\end{equation}
By Lemma \ref{closeRemark} and our choice of $r$, if $z \in B(y, 4r) \backslash B(y, 2r)$, then $\ud(z, y, V) \lesssim_{(n, p, C_V)} 1$.
It follows from Lemmas \ref{RegLem1} and \ref{RegLem0} that $\varphi_{V, j}(z) \leq \varphi_{V}(z) \lesssim_{(n, p, C_V)} 1$.
Combining this observation with \eqref{eq3.60}, \eqref{MPEq6.12}, Fatou's Lemma, and Lemma \ref{RegLem1} shows that
\begin{equation*}
\int_{\R^n \backslash B(y, 4r)} \um(z, V)^2 \abs{\Ga^V(z, y)}^2 e^{2\epsilon \varphi_{V}(z, y)} dz
\lesssim_{(\LV, p, C_V, N_V)}  r^{2-n}.
\end{equation*}
If we set $R = \frac{1}{\um(x, V)}$ then \eqref{MPEq6.10}  shows that $B(x, R) \subseteq \R^n \backslash B(y, 4r)$.
Consequently,
\begin{equation*}
\int_{B(x, R)} \um(z, V) ^2 | {\Ga^V}(z, y)|^2 e^{2\epsilon \ud(z, y, V)} \, dz
\lesssim_{(\LV, p, C_V, N_V)} r^{2-n}.
\end{equation*}
An application of the triangle inequality and Lemma \ref{closeRemark} shows that for $z \in B(x, R)$,
\begin{align*}
\ud\pr{x, y, V}
&\le \ud\pr{x, z, V} + \ud\pr{z, y, V}
\le L + \ud\pr{z, y, V},
\end{align*}
where $L = L\pr{n, p, C_V}$, so that
$$ e^{2\epsilon \ud(z, y, V)} \geq C\pr{n, p, C_V, \eps} e^{2 \epsilon \ud(x, y, V)}.$$
Furthermore, Lemma \ref{muoBounds}(a) shows that $R^{-1} = \um(x, V) \simeq_{(n, p, C_V)} \um(z, V)$ so that
\begin{equation}
\label{MPEq6.14}
\pr{\fint_{B(x, R)} |{\Ga^V}(z, y)|^2 \, dz }^\frac12
\lesssim_{(\LV, p, C_V, N_V, \eps)} [\um(x, V) \um(y, V)]^{(n-2)/2}e^{-\epsilon \ud(x, y, V)}.
\end{equation}

Choose $\ga : \brac{0,1} \to \R^n$ so that $\ga\pr{0} = x$, $\ga\pr{1} = y$ and
\begin{align*}
2 \ud\pr{x, y, V} \ge \int_0^1 \um\pr{\ga\pr{t}, V} \abs{\ga'\pr{t}} dt.
\end{align*}
It follows from Lemma \ref{muoBounds}(c) that
\begin{align*}
\ud\pr{x, y, V}
\ge \frac c 2 \int_0^1 \frac{\um\pr{x, V}\abs{\ga'\pr{t}} dt}{\brac{1 + \abs{\ga\pr{t} - x} \um\pr{x, V}}^{k_0/(k_0+1)}}
= \frac c 2 \int_0^1 \frac{\abs{\widetilde \ga \,'\pr{t}} dt}{\brac{1 + \abs{\widetilde \ga\pr{t}}}^{k_0/(k_0+1)}},
\end{align*}
where $\widetilde \ga : \brac{0,1} \to \R^n$ is a shifted, rescaled version of $\ga$.
That is, $\widetilde \ga(0) = 0$ and $\widetilde \ga(1) = \um\pr{x, V} \pr{y - x}$.
This integral is bounded from below by the geodesic distance from $0$ to $\um\pr{x, V} \pr{y - x}$ in the metric
$$\frac{dz}{\pr{1 + \abs{z}}^{k_0/(k_0+1)}}.$$
A computation shows that the straight line path achieves this minimum.
Therefore,
\begin{align*}
\ud\pr{x, y, V}
&\ge \frac c 2 \int_0^1 \frac{\um\pr{x, V} \abs{y - x} dt}{\brac{1 + \um\pr{x, V} t \abs{y - x}}^{k_0/(k_0+1)}}
= \frac {c(k_0+1)} 2 \brac{ \pr{1+\um\pr{x, V} \abs{y - x}}^{1/(k_0+1)} - 1} \\
&\ge C \pr{\um\pr{x, V} \abs{y - x}}^{1/(k_0+1)},
\end{align*}
where we have used that $\abs{x - y} \ge \frac 4 {\um\pr{x, V}}$ to reach the final line.
In particular, for any $\eps' > 0$, it holds that
\begin{align}
\label{distlowBd}
\um\pr{x, V} \abs{y - x}
\leq \frac{1}{C^{k_0 + 1}} \ud(x, y, V)^{k_0 + 1}
\leq \frac{1}{C^{k_0 + 1}} C_{\epsilon'} e^{\epsilon' \ud(x, y, V)/2},
\end{align}
where $C_{\epsilon'} > 0$ depends on $\epsilon'$.
A similar argument shows that
\begin{align*}
\um\pr{y, V} \abs{y - x}
\leq \frac{1}{C^{k_0 + 1}} C_{\epsilon'} e^{\epsilon' \ud(x, y, V)/2}.
\end{align*}
Multiplying these two bounds gives
\begin{equation*}
\um\pr{x, V} \um\pr{y, V}  \leq C^{-2(k_0 + 1)} C_{\epsilon'}^2 e^{\epsilon' \ud(x, y, V)} \abs{y - x}^{-2}.
\end{equation*}
Define $\epsilon'  = \frac{\eps}{n-2}$.
We then substitute this upper bound into \eqref{MPEq6.14} and simplify to get
\begin{equation}
\label{averagedBound}
\pr{\fint_{B(x, R)} |{\Ga^V}(z, y)|^2 \, dz }^\frac12
\lesssim_{(\LV, p, C_V, N_V, \eps)} \frac{e^{-\epsilon \ud(x, y, V)/2}}{\abs{x - y}^{n-2}}.
\end{equation}
Finally, since we assume that $y \not \in B(x, R)$, then $\LV \Ga^V\pr{\cdot, y} = 0$ in $B(x, R)$.
In particular, \eqref{eq3.47} from assumption {\rm{(IB)}} is applicable, so that
\begin{align*}
\abs{\Ga^V\pr{x, y}}
&\le \norm{\Ga^V\pr{\cdot, y}}_{L^\iny\pr{B(x, R/2)}}
\le C_{\rm{IB}} \pr{\fint_{B(x, R)} |{\Ga^V}(z, y)|^2 \, dz }^\frac12,
\end{align*}
and the conclusion follows by combining the previous two inequalities.
\end{proof}

\begin{rem}
As in \cite{MP19}, if instead of assuming {\rm{(IB)}} and {\rm{(H)}}, we assume that $\Ga^V$ exists and satisfies the pointwise bound described by \eqref{eq3.60}, then \eqref{averagedBound} holds.
\end{rem}

Define the diagonal matrix $\La = \abs{V} I$, where $\abs{V} = \la_d \in \sBp$ is the largest eigenvalue of $V$ and $I$ denotes the $d \times d$ identity matrix.
Then set $\LGa = -D_\al\pr{A^{\al \be} D_\be  } + \La$ to be the associated Schr\"odinger operator.
We let $\Ga^\La$ denote the fundamental matrix for $\LGa$.
Since the assumptions imposed to make sense of $\Ga^V$ are inherited for $\Ga^\La$, then $\Ga^\La$ exists and satisfies the conclusions of Theorem \ref{t3.6} as well.
Because $\La$ is diagonal, then its upper and lower auxiliary functions coincide and are equal to $\ovm\pr{x, V}$.
That is, $\um\pr{x, \La} = \ovm\pr{x, \La} = \ovm\pr{x, V}$ so that $\ud\pr{x, \La} = \ovd\pr{x, \La} = \ovd\pr{x, V}$.
As such, we can obtain an upper bound for $\Ga^\La$ without having to assume that $V \in \NC$ or even that $V \in \ND$, see Remark \ref{noND}.
We accomplish this by applying the following lemma in place of Lemma \ref{MainUpperBoundLem}.

\begin{lem}[Upper bound lemma for $V = \La$]
\label{MainUpperBoundCor}
Let $\LGa$ be as defined above, where $A$ satisfies \eqref{ellip} and \eqref{Abd}, and $\abs{V} \in \sBp$ for some $p > \frac n 2$.
Let $B \su \R^n$ be a ball.
Assume that $\V{u} \in W_{\abs{V}I}^{1, 2}(\R^n \backslash B)$ is a weak solution to $\LGa \V{u} = 0$ in $\R^n \backslash B$.
Let $\phi \in C_c ^\infty (\R^n)$ satisfy $\phi = 0$ on $2B$ and let $g \in C^1(\R^n)$ be a nonnegative function satisfying $|\nabla g(x)| \lesssim_{(n, p, C_V)} m(x, \abs{V})$ for every $x \in \R^n$.
Then there exists $\eps_1$, $C_1$, both depending on $d, n, p, C_{\abs{V}}, \la, \La$, such that whenever $\eps \in \pr{0, \eps_1}$, it holds that
\begin{equation*}
\inrn m(\cdot, \abs{V})^2 \abs{\phi \V{u}}^2  e^{2 \epsilon g} \, \le C_1 \inrn |\V{u}|^2 |\nabla \phi|^2 e^{2\epsilon g} .
\end{equation*}
\end{lem}

The proof of this result exactly follows that of Lemma \ref{MainUpperBoundLem} except that the Fefferman-Phong inequality described by Corollary \ref{FPmlCor} is used in place of Lemma \ref{FPml}.
We arrive at the following corollary to Theorem \ref{UppBoundThm}.

\begin{cor}[Exponential upper bound for $V = \La$]
\label{UppBoundCor}
Let $\LGa = -D_\al\pr{A^{\al \be} D_\be} + \abs{V} I$, where $A$ satisfies \eqref{ellip} and \eqref{Abd}, and $\abs{V} \in \sBp$ for some $p > \frac n 2$.
Assume that {\rm{(IB)}} and {\rm{(H)}} hold.
Let $\Ga^\La(x, y)$ denote the fundamental matrix of $\MC{L}_\La$ and let $\eps_1$ be as given in Lemma \ref{MainUpperBoundCor}.
For any $\eps < \eps_1$, there exists $C = C(\LGa, p, C_{\abs{V}}, \eps)$ so that for all $x, y \in \R^n$,
\begin{equation*}
\abs{\Ga^\La(x, y)} \leq \frac{C e^{-\epsilon \ovd(x, y, V)}}{|x-y|^{n-2}}.
\end{equation*}
\end{cor}

This result will be used to obtain lower bounds in the next section.

\section{Lower Bounds}
\label{LowBds}

Here we prove an exponential decay lower bound for the fundamental matrix associated to our elliptic operator.
As before, the elliptic operator $\LV$ is given by \eqref{elEqDef}, where the matrix $A$ satisfies ellipticity and boundedness as described by \eqref{ellip} and \eqref{Abd}, respectively.
For the zeroth order term, we assume that $V \in \Bp \cap \ND$ for some $p > \frac n 2$.
In contrast to the upper bound section, we will not require that $V \in \NC$.
In fact, many of the results in this section hold when we assume that $\abs{V} \in \sBp$ (instead of $V \in \Bp$) and accordingly replace all occurrences of $\ovm\pr{\cdot, V}$ with $m\pr{\cdot, \abs{V}}$.
The assumption that $V \in \ND$ ensures that the spaces $W^{1,2}_{V, 0}\pr{\R^n}$ are Hilbert spaces and we require this for Lemma \ref{UniquenessLem}, for example.
Moreover, the Hilbert spaces are crucial to the fundamental matrix constructions in Section \ref{FundMat}.

We also assume that conditions {\rm{(IB)}} and {\rm{(H)}} hold so that we can meaningfully discuss our fundamental matrices and draw conclusions about them.
Further on, we will impose a pair of additional assumptions for fundamental matrices.
As with {\rm{(IB)}} and {\rm{(H)}}, these assumptions are known to hold in the scalar setting.

Let ${\Ga}^0(x, y)$ denote the fundamental matrix for the homogeneous operator $\Lz$ that we get when $V \equiv 0$.
That is, $\Lz := -D_\al\pr{A^{\al \be} D_\be}$.
Since the assumptions imposed to make sense of $\Ga^V$ are inherited for $\Ga^0$, the conclusions of Theorem \ref{t3.6} hold for $\Ga^0$.
Recall that $\mathcal{L}^\La = \Lz + \La$, where $\La = \abs{V} I$ and $\Ga^\La$ denotes the associated fundamental matrix.

In \cite{She99}, a clever presentation of $\Ga^0 - \Ga^V$ is used to prove bounds for that difference function.
Here, we take a slightly different approach and look at both $\Ga^0 - \Ga^\La$ and $\Ga^\La - \Ga^V$, then combine the bounds.
Using the fundamental matrix associated to the operator with a diagonal matrix as an intermediary allows us to prove the bounds that we require for the lower bound estimates without having to assume that $V \in \NC$ or impose other conditions.

We begin with the representation formula.
To establish this result, we follow the ideas from \cite{MP19}.

\begin{lem}[Representation formula]
\label{UniquenessLem}
Assume that the coefficient matrix $A$ satisfies boundedness \eqref{Abd} and ellipticity \eqref{ellip}, and that $V$ is a locally integrable matrix weight that satisfies \eqref{VAssump}.
Assume also that conditions {\rm{(IB)}} and {\rm{(H)}} hold.
Let $\Ga^0$, $\Ga^\La$, and $\Ga^V$ denote the fundamental matrices of $\Lz$ $\LGa$, and $\LV$, respectively.
Then
\begin{align*}
\Ga^0(x, y) - \Ga^V (x, y)
&= \inrn \Ga^0(x, z) \La(z) \Ga^\La(z, y) \, dz
+ \inrn \Ga^\La (x, z)\brac{V(z) - \La(z)} \Ga^V(z, y) \, dz.
\end{align*}
\end{lem}

\begin{proof}
Let $\WVd$ denote the dual space to $\WV$.
Given $\V{f} \in \WVd$, an application of the Lax-Milgram theorem shows that there exists unique $\V{u} \in \WV$ for which
$$\BV\brac{\V{u}, \V{v}} = \V{f}(\V{v}) \quad \text{for every } \V{v} \in \WV.$$
We denote $\V{u}$ by $\LV ^{-1} \V{f}$, so that $\LV^{-1}  : \WVd \rightarrow \WV$ and
\begin{equation}
\label{LaxEq}
\BV\brac{\LV^{-1} \V{f}, \V{v}} = \V{f}(\V{v}) \quad \text{for every } \V{v} \in \WV.
\end{equation}
Note that the inverse mapping $\LV : \WV \rightarrow \WVd$ satisfies $\pr{\LV\V{u}}(\V{v}) = \BV\brac{\V{u}, \V{v}}$ for every $\V{v} \in \WV$.
In particular,
$$(\LV \LV^{-1} \V{f}) (\V{v}) = \BV\brac{\LV^{-1} \V{f}, \V{v}} =  \V{f}(\V{v}) \quad \text{for every } \V{v} \in \WV$$
showing that $\LV \LV^{-1}$ acts as the identity on $\WVd$.
On the other hand, if $\V{f} = \LV \V{u} $, then $\V{f}(\V{v})  = \BV\brac{\V{u}, \V{v}}$ for any $\V{v} \in \WV$.
It follows that
$$\V{u} = \LV^{-1} \V{f} = \LV^{-1} \LV \V{u}$$
and we conclude that $\LV^{-1} \LV$ is the identity on $\WV$.
Since $\BV\brac{\V{v}, \V{u}} = \BV^*\brac{\V{u}, \V{v}}$ for every $\V{u}, \V{v} \in \WV$, then analogous statements may be made for $\LV^{*}$ and $\pr{\LV^*}^{-1}$.

Since $\norm{\V{u}}_{\WV} \leq \norm{\V{u}}_{\WL}$, then $\WL \subseteq \WV$ so that $\WVd \subseteq \WLd$.
It follows that for any $\V{f} \in \WVd$, $\LV \LGa ^{-1} \V{f} \in \WVd$.
Observe that if $\V{u}, \V{v} \in \WL$, then
$$\brac{\pr{\LGa - \LV}\V{u}}(\V{v}) = \BLa\brac{\V{u}, \V{v}} - \BV\brac{\V{u}, \V{v}} = \innp{(\Lambda - V) \V{u}, \V{v}}_{L^2(\Rn)}.$$
Thus, with $\V{f} \in \WVd$, we deduce that
\begin{align}
\label{MainRepEq}
\V{f} -  \LV \LGa ^{-1} \V{f}
= \pr{\LGa - \LV}\LGa^{-1} \V{f}
= (\Lambda - V) \LGa^{-1} \V{f}.
\end{align}
Since $\V{f},  \LV \LGa ^{-1} \V{f} \in \WVd$ as noted above, then $(\Lambda - V) \LGa^{-1} \V{f} \in \WVd$ as well.
It follows that
$$\LV ^{-1} (\Lambda - V) \LGa^{-1} \V{f} \in \WV.$$
By applying $\LV^{-1}$ to both sides of \eqref{MainRepEq}, we see that
\begin{equation}
\label{MainRepEq2}
\LV^{-1} \V{f} = \LGa ^{-1} \V{f}   + \LV ^{-1} (\Lambda - V) \LGa^{-1} \V{f}.
\end{equation}

For $\V{\phi} \in C_c^\infty (\Rn) \subseteq \WVd$ that acts via $\disp \V{\phi}(\V{u}) = \innp {\V{u}, \V{\phi}}_{L^2(\Rd)}$, we see from \eqref{LaxEq} that
\begin{align*}
\BV\brac{\LV ^{-1} (\Lambda - V) \LGa^{-1} \V{f} , \pr{\LV^*}^{-1} \V{\phi}}
& = \pr{(\Lambda - V) \LGa^{-1} \V{f}}\pr{\pr{\LV^*}^{-1} \V{\phi}}
= \innp{(\Lambda - V) \LGa^{-1} \V{f},  \pr{\LV^*}^{-1} \V{\phi}}_{L^2(\Rn)},
\end{align*}
and
\begin{align*}
\BV\brac{\LV ^{-1} (\Lambda - V) \LGa^{-1} \V{f} , \pr{\LV^*}^{-1} \V{\phi}}
&= \BV^*\brac{\pr{\LV^*}^{-1} \V{\phi}, \LV ^{-1} (\Lambda - V) \LGa^{-1} \V{f} }
 = \V{\phi} \pr{\LV ^{-1} (\Lambda - V) \LGa^{-1} \V{f}} \\
& = \innp{\LV ^{-1} (\Lambda - V) \LGa^{-1} \V{f}, \V{\phi}}_{L^2(\Rn)}.
\end{align*}
Combining these observations shows that
\begin{equation}
\label{DualIdent}
\innp{\LV ^{-1} (\Lambda - V) \LGa^{-1} \V{f}, \V{\phi}}_{L^2(\Rn)} =  \innp{(\Lambda - V) \LGa^{-1} \V{f},  \pr{\LV^*}^{-1} \V{\phi}}_{L^2(\Rn)} .
\end{equation}

Pairing \eqref{MainRepEq2} with $\V{\phi}$ in an inner product, integrating over $\R^n$, and using \eqref{DualIdent} then gives
\begin{equation}
\label{equalityToExpand}
\innp{\LV^{-1} \V{f}, \V{\phi}}_{L^2(\Rn)} = \innp{\LGa ^{-1} \V{f}, \V{\phi}}_{L^2(\Rn)}  + \innp{(\Lambda - V) \LGa^{-1} \V{f},  \pr{\LV^*}^{-1} \V{\phi}}_{L^2(\Rn)}.
\end{equation}
Recall from Definition \ref{d3.3} that $\disp \LV^{-1} \V{f}(x) = \int_{\R^n} \Ga^V\pr{x,y} \V{f}(y) dy$ for any $\V{f} \in L^\iny_c\pr{\R^n}^d$.
By taking $\V{f}, \V{\phi} \in C_c ^\infty(\Rn)$ with disjoint supports, it follows that
\begin{align*}
\innp{\LV^{-1} \V{f}, \V{\phi}}_{L^2(\Rn)}
&= \int_{\R^n} \innp{\int_{\R^n} \Ga^V\pr{x,y} \V{f}(y) dy, \V{\phi}(x)} dx
= \int_{\R^n}\int_{\R^n} \innp{ \Ga^V\pr{x,y} \V{f}(y), \V{\phi}(x)} dy dx,
\end{align*}
where the application of Fubini is justified by the fact that $\Ga^V$ is locally bounded away from the diagonal.
A similar equality holds for the second term in \eqref{equalityToExpand}.
For the last term in \eqref{equalityToExpand}, observe that
\begin{align*}
\innp{(\Lambda - V) \LGa^{-1} \V{f},  \pr{\LV^*}^{-1} \V{\phi}}_{L^2(\Rn)}
&=\int_{\R^n} \innp{(\Lambda(z) - V(z)) \int_{\R^n} \Ga^V\pr{z,y} \V{f}(y) dy,  \int_{\R^n} \Ga^{V*}\pr{z,x} \V{\phi}(x) dx} dz \\
&=\int_{\R^n}\int_{\R^n}\int_{\R^n} \innp{\Ga^{V*}\pr{z,x}^T (\Lambda(z) - V(z))  \Ga^V\pr{z,y} \V{f}(y) , \V{\phi}(x) } dz dy dx \\
&=\int_{\R^n}\int_{\R^n} \innp{\brac{\int_{\R^n} \Ga^{V}\pr{x,z} (\Lambda(z) - V(z))  \Ga^V\pr{z,y} dz} \V{f}(y) , \V{\phi}(x) } dy dx,
\end{align*}
where we have used the property that $\Gamma^V(x, z) = \Gamma^{V*}(z, x)^T$.
Putting it all together gives
\begin{equation}
\label{FLCoVRes}
0 =  \int_{\Rn} \int_{\Rn}\innp{\brac{\Gamma^V (x, y) - \Gamma^{\Lambda} (x, y) - \int_{\Rn}\Gamma^V  (x, z)(\Lambda(z) -V(z)) \Gamma^\Lambda (z, y)  \, dz} \V{f}(y),  \V{\phi}(x)} \, dy \, dx.
\end{equation}
By \eqref{eq3.59} in Theorem \ref{t3.6}, the functions $\disp \Gamma^V (x, y)$ and $\disp \Gamma^{\Lambda}(x, y)$ are locally bounded on $\R^n \times \R^n \setminus \Delta$.
As shown in Lemma \ref{localBoundedIntegrals} below, $\disp\int_{\Rn}\Gamma^V  (x, z)(\Lambda(z) -V(z)) \Gamma^\Lambda (z, y)  \, dz$ is also locally bounded on $\R^n \times \R^n \setminus \Delta$.
It follows that $\disp \Gamma^V (x, y) - \Gamma^{\Lambda} (x, y) - \int_{\Rn}\Gamma^V  (x, z)(\Lambda(z) -V(z)) \Gamma^\Lambda (z, y)  \, dz \in L^1_{\loc}\pr{\R^n \times \R^n \setminus \Delta}$.
As \eqref{FLCoVRes} holds for all $\V{f}, \V{\phi} \in C_c ^\infty(\Rn)$ with disjoint supports, then an application of Lemma \ref{fundCoV} shows that for a.e. $(x, y) \in \Rn \times \Rn$,
\begin{equation}
\label{VLambdaRep}
\Gamma^V (x, y)= \Gamma^{\Lambda} (x, y) - \int_{\Rn}\Gamma^V  (x, z)\brac{V(z) - \Lambda(z)} \Gamma^\Lambda (z, y)  \, dz.
\end{equation}

Since $\norm{\V{u}}_{\Y} \leq \norm{\V{u}}_{\WV}$ implies that $\WL \subseteq \Y$, then $\Yd \subseteq \WLd$.
In particular, all of the arguments from above hold with $V$ replaced by $0$, so we get that for a.e. $(x, y) \in \Rn \times \Rn$,
\begin{equation}
\label{0LambdaRep}
\Gamma^0 (x, y)= \Gamma^{\Lambda} (x, y) + \int_{\Rn}\Gamma^0  (x, z)\Lambda(z) \Gamma^\Lambda (z, y)  \, dz.
\end{equation}
Subtracting \eqref{VLambdaRep} from \eqref{0LambdaRep} leads to the conclusion of the lemma.
\end{proof}

For completeness, we prove the following version of the fundamental lemma of calculus of variations.

\begin{lem}[Fundamental lemma of calculus of variations for matrix-valued functions]
\label{fundCoV}
Let $G\pr{x,y}$ be a $d \times d$ matrix function defined on $\R^n \times \R^n \setminus \Delta$ that is locally integrable.
If
$$\int_{\Rn} \int_{\Rn}\innp{G(x,y) \V{f}(y),  \V{\phi}(x)} \, dy \, dx = 0$$
for every $\V{f}, \V{\phi} \in C_c ^\infty(\Rn)$ with disjoint supports, then $G(x, y) = 0$ a.e. on $\R^n \times \R^n$.
\end{lem}

\begin{proof}
Assume first that $G(x,y)$ is continuous.
For the sake of contradiction, assume that $G(x,y)$ is not identically zero on $\R^n \times \R^n \setminus \Delta$.
Then there exists some $\pr{x_0, y_0} \in \R^n \times \R^n$, $x_0 \ne y_0$, for which $G(x_0, y_0) \ne 0$.
This means that there exists some vector $\V{e} \in \Sd$ for which $\innp{G(x_0, y_0) \V{e}, \V{e}} \ne 0$.
Without loss of generality $\innp{G(x_0, y_0) \V{e}, \V{e}} = 2\eps > 0$.
Since $G(x, y)$ is continuous, then there exists $\de < \frac 1 5 \dist\pr{x_0, y_0}$ for which $\innp{G(x, y) \V{e}, \V{e}} \ge \eps$ whenever $\pr{x,y} \in B_\de\pr{x_0} \times B_\de\pr{y_0}$.
Let $\eta \in C^\iny_c\pr{B_2}$, where $B_2  \su \R^n$, be a non-negative cutoff function for which $\eta \equiv 1$ on $B_{1}$.
Now we let $\V{f}(x) = \eta\pr{\frac{x - x_0}\de} \V{e}$ and $\V{\phi}\pr{y} = \eta\pr{\frac{y - y_0}\de} \V{e}$ and observe that $\supp \V{f} \su B_{2\de}\pr{x_0}$, $\supp \V{\phi} \su B_{2\de}\pr{y_0}$ so that the supports of $\V{f}$ and $\V{\phi}$ are disjoint.
Moreover,
\begin{align*}
\int_{\Rn} \int_{\Rn}\innp{G(x,y) \V{f}(y),  \V{\phi}(x)} \, dy \, dx
&= \int_{\Rn} \int_{\Rn}\innp{G(x,y) \eta\pr{\frac{x - x_0}\de} \V{e},  \eta\pr{\frac{y - y_0}\de} \V{e}} \, dy \, dx \\
&\ge \int_{B_\de\pr{x_0}} \int_{B_\de\pr{y_0}}\innp{G(x,y) \V{e},  \V{e}}  \, dy \, dx
= \eps \abs{B_\de}^2
> 0,
\end{align*}
which gives a contradiction.
If $G\pr{x,y}$ is not continuous, we use that continuous functions are dense in $L^1$ to locally approximate $G(x,y)$, and then apply the previous result.
\end{proof}

Next, we establish that the integral functions in Lemma \ref{UniquenessLem} are locally integrable away from the diagonal.

\begin{lem}[Local integrability on $\R^n \times \R^n \setminus \Delta$]
\label{localBoundedIntegrals}
Assume that $A$ satisfies boundedness \eqref{Abd} and ellipticity \eqref{ellip}, and that $V \in \Bp \cap \ND$ for some $p > \frac n 2$.
Assume also that {\rm{(IB)}} and {\rm{(H)}} hold so that $\Ga^0$, $\Ga^\La$, and $\Ga^V$, the fundamental matrices of $\Lz$ $\LGa$, and $\LV$, respectively, exist and satisfy the conclusions of Theorem \ref{t3.6}.
Define $\disp G(x,y) = \int_{\Rn}\Gamma^V  (x, z)\brac{V(z) - \Lambda(z)} \Gamma^\Lambda (z, y)  \, dz$ and $\disp H(x,y) = \int_{\Rn}\Gamma^0 (x, z)\Lambda(z)\Gamma^\Lambda (z, y)  \, dz$.
Then $G, H \in L^1_{\loc}\pr{\R^n \times \R^n \setminus \Delta}$.
\end{lem}

\begin{proof}
We show that $G \in L^1_{\loc}\pr{\R^n \times \R^n \setminus \Delta}$ and note that the argument for $H$ is analogous.
Set $r = \abs{x - y}$ and let $\eps = \frac{\eps_1}2$, where $\eps_1 > 0$ is as in Lemma \ref{MainUpperBoundCor}.
An application of Lemma \ref{UniquenessLem} followed by Corollary \ref{UppBoundCor} along with the bound \eqref{eq3.60} from Theorem \ref{t3.6} applied to $\Ga^V$ shows that
\begin{equation}
\label{GBound}
\begin{aligned}
\abs{G\pr{x, y}}
&\le \inrn \abs{\Ga^V (x, z)} \abs{V(z) - \La(z)} \abs{\Ga^\La(z, y)} \, dz
\lesssim_{(\LV, p, C_V)} \inrn \frac{e^{-\epsilon \ovd(x, z, V)}\, |V(z) - \La(z)| }{|z-x|^{n-2} |z-y|^{n-2}}dz \\
\lesssim& \int_{B(x, \frac r 2)} \frac{ |V(z)|}{|z-x|^{n-2} |z-y|^{n-2}} dz
+ \int_{B(y, \frac r 2)} \frac{ |V(z)|}{|z-x|^{n-2} |z-y|^{n-2}} dz \\
+& \int_{\R^n \setminus \pr{B(x, \frac r 2) \cup B(y, \frac r 2)}} \frac{e^{-\epsilon \ovd(z, x, V)}\, |V(z)|}{|z-x|^{n-2} |z-y|^{n-2}} dz.
\end{aligned}
\end{equation}
For the first term, an application of H\"older's inequality shows that
\begin{equation}
\label{xBall}
\begin{aligned}
\int_{B(x, \frac r 2)} & \frac{ |V(z)| \, dz}{|z-x|^{n-2} |z-y|^{n-2}}
\lesssim_{(n)} r^{2-n}\int_{B(x, \frac r 2)} \frac{|V(z)| \, dz}{|z-x|^{n-2}}  \\
&\le r^{2-n} \norm{V}_{L^p\pr{B(x, \frac r 2)}} \pr{\int_{0}^{r/2} \rho^{n-1 + \frac {p\pr{2-n}} {p-1}} d\rho}^{\frac {p-1} p}
\lesssim_{\pr{n, p}} \norm{V}_{L^p\pr{B(x, \frac r 2)}} r^{4 - 2n + \frac {pn} {p-1}}.
\end{aligned}
\end{equation}
An analogous argument shows that for the second term in \eqref{GBound}, we get
\begin{align}
\label{yBall}
\int_{B(y, \frac r 2)} \frac{|V(z)| \, dz}{|z-x|^{n-2} |z-y|^{n-2}}
&\lesssim_{\pr{n, p}} \norm{V}_{L^p\pr{B(y, \frac r 2)}} r^{4 - 2n + \frac {pn} {p-1}}.
\end{align}

We now turn to the third integral in \eqref{GBound}.
Observe that with $R = \frac 1 {\ovm\pr{x, V}}$,
\begin{equation}
\label{thirdIntegral}
\begin{aligned}
\int_{\R^n \setminus \pr{B(x, \frac r 2) \cup B(y, \frac r 2)}} &\frac{e^{-\epsilon \ovd(z, x, V)}\, |V(z)|}{|z-x|^{n-2} |z-y|^{n-2}} dz
\lesssim \int_{\R^n \setminus B(x, \frac r 2)} \frac{e^{-\epsilon \ovd(z, x, V)}\, |V(z)|}{|z-x|^{2n-4}} dz \\
&\le \int_{B(x,R) \setminus B(x, \frac r 2)} \frac{|V(z)|}{|z-x|^{2n-4}} dz
+ \int_{\R^n \setminus B(x, R)} \frac{e^{-\epsilon \ovd(z, x, V)}\, |V(z)|}{|z-x|^{2n-4}} dz.
\end{aligned}
\end{equation}
Assuming that $R \ge \frac r 2$, choose $J \in \Z_{\ge 0}$ so that $2^{J-1} r \le R \le 2^J r$.
Let $q = p$ if $n \ge 4$ and $q \in \pr{\frac 3 2, \min\set{p, 3}}$ if $n = 3$.
Since $q \le p$, then by Lemma \ref{GehringLemma}, $V \in \MC{B}_{q}$ as well with the same uniform $\Bp$ constant.
Let $q'$ denote the H\"older conjugate of $q$.
An application of H\"older's inequality shows that
\begin{equation*}
\begin{aligned}
\int_{B(x, R) \setminus B(x, \frac r 2)} \frac{|V(z)|}{|z-x|^{2n-4}} dz
&\le \pr{\int_{B(x, R)} \abs{V(z)}^q dz}^{\frac 1 q} \pr{\int_{B(x, R) \setminus B(x, \frac r 2)} \frac{1}{|z-x|^{q'\pr{2n-4}}} dz}^{\frac 1 {q'}}.
\end{aligned}
\end{equation*}
Now
\begin{align*}
\pr{\int_{B(x, R)} \abs{V(z)}^q dz}^{\frac 1 q}
&= \abs{B(x, R)}^{\frac 1 q} \pr{\fint_{B(y, R)} \abs{V(z)}^q dz}^{\frac 1 q}
\lesssim_{\pr{d, n, q, C_V}} R^{\frac n q -2} \pr{\frac 1 {R^{n-2}}\int_{B(x, R)} \abs{V(z)} dz} \\
&= R^{\frac n q -2} \Psi\pr{x, \frac 1 {\ovm\pr{x, V}}; \abs{V}}
\lesssim_{(d, n, p, C_V)} R^{\frac n q -2},
\end{align*}
where we have used \eqref{normRelationship}.
On the other hand,
\begin{align*}
\pr{\int_{B(x, R) \setminus B(x, \frac r 2)} \frac{1}{|z-x|^{q'\pr{2n-4}}} dz}^{\frac 1 {q'}}
&= \pr{\int_{\frac r 2}^R \rho^{n- 2q'\pr{n-2}-1} d\rho}^{\frac 1 {q'}}
\lesssim_{\pr{n, q}} r^{4 - n - \frac n q},
\end{align*}
where we have used that $n- 2q'\pr{n-2} < 0$, which follows from the definition of $q$.
Combining the previous two inequalities shows that
\begin{equation}
\label{I31}
\begin{aligned}
\int_{B(x, R) \setminus B(x, \frac r 2)} \frac{|V(z)|}{|z-x|^{2n-4}} dz
&\lesssim_{(d, n, p, C_V)} r^{2 - n} \brac{r \ovm\pr{x, V}}^{2 - \frac n q}.
\end{aligned}
\end{equation}

For the exterior integral, we have
\begin{equation}
\label{I32}
\begin{aligned}
&\int_{\R^n \setminus B(x, R)} \frac{e^{-\epsilon \ovd(z, x, V)}\, |V(z)|}{|z-x|^{2n-4}} dz
= \sum_{j=1}^\iny \int_{B(x, 2^{j} R) \setminus B(x, 2^{j-1} R)} \frac{e^{-\epsilon \ovd(z, x, V)}\, |V(z)|}{|z-x|^{2n-4}} dz \\
\le& \sum_{j=1}^\iny  \pr{\int_{B(x, 2^{j} R) \setminus B(x, 2^{j-1} R)} \frac{1}{|z-x|^{q' \pr{2n-4}}} dz}^{\frac 1 {q'}} \pr{\int_{B(x, 2^{j} R) \setminus B(x, 2^{j-1} R)} e^{-q \epsilon \ovd(z, x, V)}\, |V(z)|^{q} dz}^{\frac 1 {q}} \\
\lesssim_{(n)}& \sum_{j=1}^\iny \pr{2^j R}^{4 - 2n + \frac n {q'} + \frac n {q}} \pr{\fint_{B(x, 2^{j} R) \setminus B(x, 2^{j-1} R)} e^{-q \epsilon \ovd(z, x, V)}\, |V(z)|^{q} dz}^{\frac 1 {q}}.
\end{aligned}
\end{equation}
We may repeat the arguments used to reach \eqref{distlowBd} and conclude that if $\abs{x -z} \ge 2^{j-1}R = \frac{2^{j-1}}{\ovm\pr{x,V}}$, then for any $\eps' > 0$,
\begin{equation}
\label{expDistBound}
e^{\eps' \ovd(z, x, V)} \gtrsim_{(d, n, p, C_V, \eps')} \ovm\pr{x, V} \abs{x - z} \gtrsim 2^{j}.
\end{equation}
For $\eps'$ to be specified below, it follows that with $c = \frac \eps {\eps'} \ln 2$,
\begin{align*}
\pr{\fint_{B(x, 2^{j} R) \setminus B(x, 2^{j-1} R)} e^{-q \epsilon \ovd(z, x, V)}\, |V(z)|^{q} dz}^{\frac 1 {q}}
&\lesssim_{(\LV, p, C_V, \eps', c)} e^{- cj} \pr{\fint_{B(x, 2^{j} R)}  |V(z)|^{q} dz}^{\frac 1 {q}} \\
&\le e^{- cj} C_V \fint_{B(x, 2^{j} R)}  |V(z)| dz
\lesssim_{(C_V)} e^{- cj} \ga^{j} \pr{2^j R}^{-n} \int_{B(x, R)} |V(z)| dz \\
&= R^{-2}\pr{\frac{\ga}{e^c 2^n}}^j \Psi\pr{x, R, \abs{V}}
\lesssim_{(d)} R^{-2}\pr{\frac{\ga}{e^c 2^n}}^j ,
\end{align*}
where we have used that $V \in \MC{B}_{q}$ and $\Psi\pr{x, R; \abs{V}} \le d^2$.
Substituting this expression into \eqref{I32} shows that
\begin{align*}
\int_{\R^n \setminus B(x, R)} \frac{e^{-\epsilon \ovd(z, x, V)}\, |V(z)|}{|z-x|^{2n-4}} dz
&\lesssim_{(\LV, p, C_V, \eps', c)} R^{2-n}\sum_{j=1}^\iny \pr{\frac{\ga}{e^c 4^{n-2}}}^j.
\end{align*}
By choosing $\eps' \simeq_{(\ga,n)} \eps$ sufficiently small, we can ensure that $c = c\pr{\ga, n}$ is large enough for the series to converge and then
\begin{align}
\label{outerBallBd}
\int_{\R^n \setminus B(x, R)} \frac{e^{-\epsilon \ovd(z, x, V)}\, |V(z)|}{|z-x|^{2n-4}} dz
&\lesssim_{(\LV, p, C_V)} \ovm\pr{x, V}^{n-2}.
\end{align}
Combining \eqref{GBound} with \eqref{xBall}, \eqref{yBall}, \eqref{thirdIntegral}, \eqref{I31} and \eqref{outerBallBd} then shows that
\begin{align*}
\abs{G\pr{x,y}}
&\lesssim_{(\LV, p, C_V)} \pr{\norm{V}_{L^p\pr{B(x, \frac r 2)}} + \norm{V}_{L^p\pr{B(y, \frac r 2)}}} r^{4 - 2n + \frac {pn} {p-1}}
+ r^{2 - n} \brac{r \ovm\pr{x, V}}^{2 - \frac n q}
+ \ovm\pr{x, V}^{n-2}.
\end{align*}

Let $K, L \su \R^n$ be compact sets with disjoint support.
Set $d = \diam \pr{K} + \diam\pr{L} + \dist\pr{K, L}$ and define $M$ to be the closed $d$-neighborhood of $K \cup L$, another compact set.
Note that for any $x \in K$ and any $y \in L$, $\disp \norm{V}_{L^p\pr{B(x, \frac r 2)}} + \norm{V}_{L^p\pr{B(y, \frac r 2)}} \le \norm{V}_{L^p\pr{M}}$.
As $r = \abs{x - y}$, then $r \le d$.
Finally, since $\ovm\pr{x, V}$ is bounded on compact sets, then we may conclude that $G\pr{x,y}$ is bounded on $K \times L$.
It follows that $G(x,y)$ is locally integrable away from the diagonal, as required.
\end{proof}

Using the representation formula from Lemma \ref{UniquenessLem} and many arguments from the proof of Lemma \ref{localBoundedIntegrals}, we can now bound the difference between $\Ga^V$ and $\Ga^0$.
We use $\Ga^\La$ as an intermediary because this allows us to use the upper bound described by Corollary \ref{UppBoundCor} instead of the one for $\Ga^V$ given in Theorem \ref{UppBoundThm}.
The advantage to this approach is that we don't need to assume that $V \in \NC$.

\begin{lem}[Lower bound lemma]
\label{SmallScaleLowerLem}
Let $\LV$ be given by \eqref{elEqDef}, where $A$ satisfies \eqref{ellip} and \eqref{Abd}, and $V \in \Bp \cap \ND$ for some $p > \frac n 2$.
Assume that {\rm{(IB)}} and {\rm{(H)}} hold.
Let $\Ga^V(x, y)$ denote the fundamental matrix of $\LV$.
Let $x, y \in \R^n$ be such that $|x-y| \le \frac{1}{\ovm(x, V)}$.
Set $\al = 2 - \frac n {q}$, where $q = p$ if $n \ge 4$ and $q \in \pr{\frac 3 2, \min\set{p, 3}}$ if $n = 3$.
Then there exists a constant $C_2 = C_2\pr{\LV, p, C_V}$ for which
\begin{equation*}
|\Ga^V (x, y) - \Ga^0(x, y)| \le C_2 \frac{\brac{|x-y| \ovm(x, V)}^{\al} }{|x-y|^{n-2}}.
\end{equation*}
\end{lem}

\begin{proof}
Set $r = \abs{x - y}$ and let $\eps = \frac{\eps_1}2$, where $\eps_1 > 0$ is as in Lemma \ref{MainUpperBoundCor}.
An application of Lemma \ref{UniquenessLem} followed by Corollary \ref{UppBoundCor} along with the bound \eqref{eq3.60} from Theorem \ref{t3.6} applied to $\Ga^0$ and $\Ga^V$ shows that
\begin{align}
|\Ga^V (x, y) - \Ga^0(x, y)|
\le& \inrn \abs{\Ga^0(x, z)} \abs{\La(z)} \abs{\Ga^\La(z, y)} \, dz
+ \inrn \abs{\Ga^\La (x, z)} \abs{V(z) - \La(z)} \abs{\Ga^V(z, y)} \, dz
\nonumber \\
\lesssim&_{(\LV, p, C_V)} \inrn \frac{e^{-\epsilon \ovd(z, y, V)}\, |\La(z)| }{|z-x|^{n-2} |z-y|^{n-2}}dz
+ \inrn \frac{e^{-\epsilon \ovd(x, z, V)}\, |V(z) - \La(z)| }{|z-x|^{n-2} |z-y|^{n-2}}dz
\nonumber \\
\lesssim& \int_{B(x, \frac r 2)} \frac{ |V(z)|}{|z-x|^{n-2} |z-y|^{n-2}} dz
+ \int_{B(y, \frac r 2)} \frac{ |V(z)|}{|z-x|^{n-2} |z-y|^{n-2}} dz
\nonumber \\
+& \int_{\R^n \setminus \pr{B(x, \frac r 2) \cup B(y, \frac r 2)}} \frac{e^{-\epsilon \ovd(z, y, V)}\, |V(z)|}{|z-x|^{n-2} |z-y|^{n-2}} dz
+ \int_{\R^n \setminus \pr{B(x, \frac r 2) \cup B(y, \frac r 2)}} \frac{e^{-\epsilon \ovd(x, z, V)}\, |V(z)|}{|z-x|^{n-2} |z-y|^{n-2}} dz.
\label{FSdiffBound}
\end{align}
For the first term in \eqref{FSdiffBound}, we take a different approach from the previous proof and we get
\begin{align*}
\int_{B(x, \frac r 2)} & \frac{|V(z)|}{|z-x|^{n-2} |z-y|^{n-2}} dz
\lesssim_{(n)} r^{2-n}\int_{B(x, \frac r 2)} \frac{|V(z)|}{|z-x|^{n-2}} dz \\
&= r^{2-n} \sum_{j=1}^\iny \int_{B(x, \frac r {2^{j}}) \setminus B(x, \frac r {2^{j+1}})} \frac{|V(z)|}{|z-x|^{n-2}} dz
\lesssim r^{2-n} \sum_{j=1}^\iny \pr{\frac r {2^{j}}}^{2-n}  \int_{B(x, \frac r {2^{j}})} |V(z)| dz \\
&= r^{2-n} \sum_{j=1}^\iny \Psi\pr{x, \frac r {2^{j}}; \abs{V}}
\le r^{2-n} \sum_{j=1}^\iny C_V \brac{\frac{r \, \ovm\pr{x, V}}{2^j}}^{2 - \frac n p} \Psi\pr{x, \frac 1 {\ovm\pr{x, V}}; \abs{V}},
\end{align*}
where we have applied Lemma \ref{BasicShenLem} to reach the last line.
(We remark that a version of this inequality was established in \cite[Remark 0.13]{She99} using a different argument.)
By \eqref{normRelationship} in the proof of Lemma \ref{omCompLem}, $\Psi\pr{x, \frac 1 {\ovm\pr{x, V}}; \abs{V}} \le d^2 \abs{\Psi\pr{x, \frac 1 {\ovm\pr{x, V}}; V}} = d^2$.
Since $p > \frac n 2$, the series converges and we see that
\begin{align}
\label{xBall1}
\int_{B(x, \frac r 2)} \frac{|V(z)|}{|z-x|^{n-2} |z-y|^{n-2}} dz
&\lesssim_{(d, n, p, C_V)} r^{2-n} \brac{r \, \ovm\pr{x, V}}^{2 - \frac n p}
\le r^{2-n} \brac{r \, \ovm\pr{x, V}}^{2 - \frac n q},
\end{align}
since $q \le p$.
An analogous argument shows that the second term in \eqref{FSdiffBound} satisfies
\begin{align}
\label{yBall1}
\int_{B(y, \frac r 2)} \frac{ |V(z)|}{|z-x|^{n-2} |z-y|^{n-2}} dz
&\lesssim_{(d, n, p, C_V)} r^{2-n} \brac{r \, \ovm\pr{y, V}}^{2 - \frac n q}
\lesssim_{(d, n, p, C_V)} r^{2-n} \brac{r \, \ovm\pr{x, V}}^{2 - \frac n q},
\end{align}
since Lemma \ref{muoBounds} and the assumption that $|x-y| \le \frac{1}{\ovm(x, V)}$ imply that $\ovm\pr{x, V} \simeq_{(d, n, p, C_V)} \ovm\pr{y, V}$.

We now turn to the fourth integral in \eqref{FSdiffBound}.
By the arguments in the proof of Lemma \ref{localBoundedIntegrals}, we combine \eqref{thirdIntegral} with \eqref{I31} and \eqref{outerBallBd} to get
\begin{equation}
\label{fourthIntegral}
\begin{aligned}
\int_{\R^n \setminus \pr{B(x, \frac r 2) \cup B(y, \frac r 2)}} \frac{e^{-\epsilon \ovd(x, z, V)}\, |V(z)| \, dz}{|z-x|^{n-2} |z-y|^{n-2}}
&\lesssim \int_{B(x,R) \setminus B(x, \frac r 2)} \frac{|V(z)| \, dz}{|z-x|^{2n-4}}
+ \int_{\R^n \setminus B(x, R)} \frac{e^{-\epsilon \ovd(z, x, V)}\, |V(z)| \, dz}{|z-x|^{2n-4}}  \\
&\lesssim_{(d, n, p, C_V)} r^{2 - n} \brac{r \ovm\pr{x, V}}^{2 - \frac n q}
+ \ovm\pr{x, V}^{n-2}.
\end{aligned}
\end{equation}
An analogous argument applied to the third integral in \eqref{FSdiffBound} shows that
\begin{equation}
\label{thirdIntegral1}
\begin{aligned}
\int_{\R^n \setminus \pr{B(x, \frac r 2) \cup B(y, \frac r 2)}} \frac{e^{-\epsilon \ovd(z, y, V)}\, |V(z)| \, dz}{|z-x|^{n-2} |z-y|^{n-2}}
&\lesssim_{(d, n, p, C_V)} r^{2 - n} \brac{r \ovm\pr{y, V}}^{2 - \frac n q} + \ovm\pr{y, V}^{n-2} \\
&\lesssim_{(d, n, p, C_V)} r^{2 - n} \brac{r \ovm\pr{x, V}}^{2 - \frac n q} + \ovm\pr{x, V}^{n-2},
\end{aligned}
\end{equation}
where we have again applied Lemma \ref{muoBounds} to conclude that $\ovm\pr{x, V} \simeq_{(d, n, p, C_V)} \ovm\pr{y, V}$.

Substituting \eqref{xBall1} -- \eqref{thirdIntegral1} into \eqref{FSdiffBound} shows that
\begin{equation*}
\label{differenceBound}
\begin{aligned}
|\Ga^V (x, y) - \Ga^0(x, y)|
&\lesssim_{\pr{\LV, p, C_V}} r^{2 - n} \brac{r \ovm\pr{x, V}}^{2 - \frac n q} + \ovm\pr{x, V}^{n-2}
\lesssim r^{2 - n} \brac{r \ovm\pr{x, V}}^{2 - \frac n q},
\end{aligned}
\end{equation*}
where we have used that
\begin{align*}
\ovm\pr{x, V}^{n-2}
&= r^{2 - n} \brac{r \ovm\pr{x, V}}^{n - 2}
= r^{2-n} \brac{r \ovm\pr{x, V}}^{2 - \frac n q} \brac{r \ovm\pr{x, V}}^{n - 2 - \pr{2 - \frac n q}}
\le r^{2-n} \brac{r \ovm\pr{x, V}}^{2 - \frac n q},
\end{align*}
since $r \, \ovm\pr{x, V} \le 1$ and $n\pr{1 + \frac 1 {q}} \ge 4$ by definition.
The conclusion of the lemma follows.
\end{proof}

We now prove our lower bound.
To do this, we assume that the following scale-invariant Harnack inequality holds for matrix solutions to our equation.

\begin{itemize}
\item[(SIH)]
We say that \rm{(SIH)} holds if there exists a small constant $c_S$ so that whenever $x_0 \in \R^n$ and $r \le \frac {c_S} {\ovm(x_0, V)}$, with $B = B(x_0, r)$, the following holds:
If $U$ is a $d \times d$ matrix for which $\V{u}_i \in W^{1,2}_V(2B)$ is a weak solution to $\LV \V{u}_i = \V{0}$ for each $i = 1, 2 \ldots, d$, then for every $\V{e} \in \R^d$, it holds that
\begin{equation}
\label{siHi}
\sup_{x \in B} \abs{\innp{U \V{e}, \V{e}}} \le C_{\text{H}} \inf_{x \in B} \abs{\innp{U \V{e},\V{e}}},
\end{equation}
where the constant $C_{\text{H}}$ depends only on $d, n, \la, \La$, and $V$.
\end{itemize}

The standard Harnack inequality has a constant that typically grows with the size of the domain and the norm of $V$.
Since the constant here is independent of $r$, we refer to this as the ``scale-invariant" version of the inequality.

Of course, since we are working in a systems setting, there is no guarantee that this estimate, or any of the standard de Giorgi-Nash-Moser results, necessarily hold.
As such, we assume that $\LV$ is chosen so that {\rm{(IB)}, {\rm{(H)}, and {\rm{(SIH)} all hold.
To convince ourselves that these are reasonable assumptions to make, we refer the reader to \cite{MP19} and \cite{DHM18}, where the validity of these assumptions in the scalar setting is shown.

Finally, we also need to assume the following lower bound on the fundamental matrix of the homogeneous operator.
\begin{itemize}
\item[(LB)]
We say that \rm{(LB)} holds if there exists a constant $c_0$ so that for every $\V{e} \in \Sd$,
\begin{equation}
\label{LB}
\abs{\innp{\Ga^0\pr{x,y} \V{e}, \V{e}}} \ge \frac{c_0}{\abs{x-y}^{n-2}}.
\end{equation}
\end{itemize}

In \cite{HK07}, the fundamental and Green's matrices for homogeneous elliptic systems are extensively studied.
Although such a bound does not necessary follow from the collection of results presented in \cite{HK07}, this result is shown to hold in the scalar setting; see \cite[Theorem 1.1]{GW82}.

\begin{thm}[Exponential lower bound]
\label{LowerBoundThm}
Let $\LV$ be given by \eqref{elEqDef}, where $A$ satisfies \eqref{ellip} and \eqref{Abd}, and $V \in \Bp \cap \ND$ for some $p > \frac n 2$.
Assume that {\rm{(IB)}}, {\rm{(H)}}, \rm{(SIH)}, and \rm{(LB)} hold.
Let $\Ga^V(x, y)$ denote the fundamental matrix of $\LV$.
Then there exist constants $C = C\pr{\LV, p, C_V, C_{\text{H}}, c_S, c_0}$, $\eps_2 = \eps_2\pr{d, n, p, C_V, C_{\text{H}}, c_S}$ so that for every $\V{e} \in \Sd$,
\begin{equation}
\label{upperExpBound}
\abs{\innp{\Ga^V (x, y) \V{e}, \V{e}}} \geq C \frac{e^{-\epsilon_2 \ovd(x, y, V)}}{|x-y|^{n-2}}.
\end{equation}
\end{thm}

\begin{rem}
\label{differentDistance}
If we make the weaker assumption that $\abs{V} \in \sBp$ (instead of assuming that $V \in \Bp$), then all of the statements in this section still hold with $\ovm\pr{\cdot, V}$ replaced by $m\pr{\cdot, \abs{V}}$.
Accordingly, the conclusion described by \eqref{upperExpBound} still holds with $\ovd(x, y, V)$ replaced by $d(x, y, \abs{V})$.
\end{rem}

\begin{rem}\label{differentLowerBounds}
Versions of this result still hold with either $\abs{\Ga^V (x, y)}$ or $\abs{\Ga^V (x, y) \V{e}}$ on the left side of \eqref{upperExpBound} in place of $\abs{\innp{\Ga^V (x, y) \V{e}, \V{e}}}$ if we replace the assumptions \rm{(SIH)} and \rm{(LB)} accordingly.
\end{rem}

We follow the arguments from \cite[Theorem 4.15]{She99} and  \cite[Theorem 7.27]{MP19}, with appropriate modifications for our systems setting.

\begin{proof}
By Lemma \ref{muoBounds} and the proof of \cite[Proposition 3.25]{MP19}, there exists $A = A\pr{d, n, p, C_V}$ large enough so that
\begin{equation}
\label{Adefn}
x \not \in B\pr{y, \frac{2}{\ovm(y, V)}} \text{ whenever } |x-y|\geq \frac{A}{\ovm(x, V)}.
\end{equation}
Similarly, with $c_1 = \min\set{\pr{\frac {c_0}{2C_2}}^{1/\al}, 1}$, where $c_0$ is from \rm{(LB)} and $C_2\pr{\LV, p, C_V}$ and $\al\pr{n, p}$ are from Lemma \ref{SmallScaleLowerLem}, an analogous argument shows that there exists $c_2 = c_2\pr{d, n, p, C_V, c_1}$ sufficiently small so that
\begin{equation}
\label{c2defn}
y \not \in B\pr{z, \frac{2 c_2}{\ovm(z, V)}} \text{ whenever } |z-y|\geq \frac{c_1}{\ovm(y, V)}.
\end{equation}
Since $c_1 = c_1\pr{\LV, p, C_V, c_0}$, then $c_2 = c_2\pr{\LV, p, C_V, c_0}$ as well.

We prove our bound in three settings: when $\abs{x - y}$ is small, medium, and large.
The constant $A$ is used to distinguish between the medium and the large settings, while $c_1$ is used to distinguish between the small and medium settings.
The small setting is used as a tool to prove the medium setting, so we start there.

Assume that we are in the small-scale setting where $ |z-y| \le \frac{c_1}{\ovm(z, V)}$.
By \rm{(LB)}, the triangle inequality, and Lemma \ref{SmallScaleLowerLem}, since $|z-y| \le \frac 1{\ovm(z, V)}$, then for any $\V{e} \in \Sd$,
\begin{align*}
\frac{c_0}{\abs{z-y}^{n-2}}
&\le \abs{\innp{\Ga^0\pr{z,y} \V{e}, \V{e}}}
\le \abs{\innp{\pr{\Ga^0\pr{z,y} - \Ga^V\pr{z,y}} \V{e}, \V{e}}} + \abs{\innp{\Ga^V\pr{z,y} \V{e}, \V{e}}} \\
&\le \abs{\Ga^0\pr{z,y} - \Ga^V\pr{z,y}} + \abs{\innp{\Ga^V\pr{z,y} \V{e}, \V{e}}}
\le C_2 \frac{\brac{|z-y| \ovm(z, V)}^{\al}}{|z-y|^{n-2}} + \abs{\innp{\Ga^V\pr{z,y} \V{e}, \V{e}}}.
\end{align*}
Since $c_1$ is defined so we may absorb the first term into the left, it follows that for any $\V{e} \in \Sd$,
\begin{equation}
\label{smallLowBd}
\abs{\innp{\Ga^V\pr{z,y} \V{e}, \V{e}}} \ge \frac{c_0}{2\abs{z-y}^{n-2}} \quad \text{ whenever } \, |z-y| \le \frac{c_1}{\ovm(z, V)}.
\end{equation}
Lemma \ref{muoBounds} implies that $\ovm(z, V) \simeq_{(d, n, p, C_V)} \ovm(y, V)$, so after redefining $c_1\pr{\LV, p, C_V, c_0}$ if we need to, we also have that for any $\V{e} \in \Sd$,
\begin{equation}
\label{smallLowBdFlipped}
\abs{\innp{\Ga^V\pr{z,y} \V{e}, \V{e}}} \ge \frac{c_0}{2\abs{z-y}^{n-2}} \quad \text{ whenever } \, |z-y| \le \frac{c_1}{\ovm(y, V)}.
\end{equation}

We now consider the midrange setting where $|x-y| \in \brac{\frac{c_1}{\ovm(x, V)}, \frac A {\ovm(x, V)}}$.
There is no loss in assuming that $c_2 \le c_S$, where $c_S$ is the small constant from \rm{(SIH)}.
Construct a chain $\set{z_i}_{i=1}^N$ of $N$ elements along the straight line connecting $x$ and $y$ so that $\abs{y - z_1} = \frac{c_1}{\ovm(y, V)}$, $\abs{z_{i+1} - z_i} = \frac{c_2}{\ovm(z_i, V)}$ for $i = 1, \ldots, N$, and $\abs{x - z_N} \le \frac{c_2}{\ovm(z_N, V)}$.
Since Lemma \ref{muoBounds} implies that $\ovm(z, V) \simeq_{(d, n, p, C_V)} \ovm(x, V)$ for any point $z$ along the line between $x$ and $y$, then $N \lesssim_{(d, n, p, C_V, c_1, c_2)} A$.
Since $\abs{z_j - y} \ge \abs{y - z_1} = \frac{c_1}{\ovm(y, V)}$ for all $j = 1, \ldots, N$, then \eqref{c2defn} shows that $y \not \in B\pr{z_j, \frac{2 c_2}{\ovm(z_j, V)}}$.
In particular, with $U = \Ga^V(\cdot, y)$ then $\LV \V{u}_i = 0$ weakly on each $B\pr{z_j, \frac{2 c_2}{\ovm(z_j, V)}}$.
Then we see by repeatedly applying the scale-invariant Harnack inequality \rm{(SIH)} that for any $\V{e} \in \Sd$,
\begin{equation*}
\begin{aligned}
\abs{\innp{\Ga^V (x, y) \V{e}, \V{e}}}
&\ge \inf_{B(z_N, \frac{c_2}{\ovm(z_N, V)})} \abs{\innp{\Ga^V (\cdot, y) \V{e}, \V{e}}}
\ge C_{\text{H}}^{-1} \sup_{B(z_N, \frac{c_2}{\ovm(z_N, V)})} \abs{\innp{\Ga^V (\cdot, y) \V{e}, \V{e}}}
\ge C_{\text{H}}^{-1} \abs{\innp{\Ga^V (z_N, y) \V{e}, \V{e}}} \\
&\ge C_{\text{H}}^{-1} \inf_{B(z_{N-1}, \frac{c_2}{\ovm(z_{N-1}, V)})} \abs{\innp{\Ga^V (\cdot, y) \V{e}, \V{e}}}
\ge C_{\text{H}}^{-2} \sup_{B(z_{N-1}, \frac{c_2}{\ovm(z_{N-1}, V)})} \abs{\innp{\Ga^V (\cdot, y) \V{e}, \V{e}}} \\
&\ge C_{\text{H}}^{-2} \abs{\innp{\Ga^V (z_{N-1}, y) \V{e}, \V{e}}}
\ge \ldots
\ge C_{\text{H}}^{-N} \abs{\innp{\Ga^V (z_{1}, y) \V{e}, \V{e}}}
\ge \frac{C_{\text{H}}^{-N} c_0}{2 \abs{z_1-y}^{n-2}},
\end{aligned}
\end{equation*}
where the last bound follows from \eqref{smallLowBdFlipped}.
However, $\abs{z_{1} - y} = \frac{c_1}{\ovm(y, V)} \le  C_A \frac{c_1}{\ovm(x, V)} \le C_A \abs{x-y}$.
Therefore, for any $\V{e} \in \Sd$,
$$\abs{\innp{\Ga^V (x, y) \V{e}, \V{e}}} \ge \frac{C_{\text{H}}^{-N} c_0}{2 \pr{C_A \abs{x-y}}^{n-2}} \quad \text{whenever} \, |x-y| \in \brac{\frac{c_1}{\ovm(x, V)}, \frac A {\ovm(x, V)}}.$$
Combining this observation with \eqref{smallLowBd} shows that for any $\V{e} \in \Sd$,
\begin{equation}
\label{medLowerBd}
\abs{\innp{\Ga^V (x, y) \V{e}, \V{e}}} \ge \frac{C_3}{\abs{x-y}^{n-2}} \quad \text{ whenever } \, |x-y| \le  \frac{A}{\ovm(x, V)}.
\end{equation}
An application of Lemma \ref{muoBounds} implies that for any $\V{e} \in \Sd$,
\begin{equation}
\label{medLowerBdFlipped}
\abs{\innp{\Ga^V (x, y) \V{e}, \V{e}}} \ge \frac{C_3}{\abs{x-y}^{n-2}} \quad \text{ whenever } \, |x-y| \le  \frac{A}{\ovm(y, V)},
\end{equation}
where $C_3 = C_3\pr{\LV, p, C_V, C_{\text{H}}, c_0}$ is possibly redefined.
By the proof of Lemma \ref{closeRemark}, if $|x-y| \le \frac A{\ovm(x, V)}$, then $\ovd(x, y, V) \lesssim_{(d, n, p, C_V)} 1$.
In particular, this observation combined with \eqref{medLowerBd} gives the result \eqref{upperExpBound} in the setting where $|x-y| \le \frac A{\ovm(x, V)}$.

Now consider the final (large-scale) setting where $|x-y| > \frac A{\ovm(x, V)}$.
Choose $\gamma : \brac{0, 1} \to \R^n$ with $\gamma(0) = x, \gamma(1) = y$, and
$$\int_0^1 \ovm(\ga(t), V) |\ga'(t)| \, dt \leq 2 \ovd(x, y, V).$$
Let
$$t_0 = \sup\set{ t \in [0, 1] : |x - \ga(t)|\leq \frac{A}{\ovm(x, V)}} < 1.$$
If $|\ga(t_0) - y| \leq \frac 1{\ovm(\ga(t_0), V)}$, then
$$\abs{x - y} \le \abs{x - \ga(t_0)} + \abs{\ga(t_0) - y} \le \frac{A}{\ovm(x, V)} + \frac 1{\ovm(\ga(t_0), V)} \le \frac{\tilde A}{\ovm(x, V)},$$
since Lemma \ref{muoBounds} implies that $\ovm(\ga(t_0), V) \simeq_{(d, n, p, C_V)} \ovm(x, V)$.
In this case, we may repeat the arguments from the previous paragraph to reach the conclusion of the theorem.

To proceed, we assume that $|x-y| > \frac A{\ovm(x, V)}$ and $|\ga(t_0) - y| > \frac 1{\ovm(\ga(t_0), V)}$.
Since $\ovm(\cdot, V)$ is locally bounded above and below, we can recursively define a finite sequence $0 < t_0 < t_1 < \ldots < t_\ell \le 1$ as follows.
For $j = 0, 1, \ldots, \ell$, let
$$t_j = \inf \set{ t \in [t_{j-1}, 1] : |\ga(t) - \ga(t_{j-1}) | \geq \frac{1}{\ovm(\ga(t_{j-1}), V)} }.$$
Then set $B_{j} = B\pr{\ga(t_{j}), \frac{1}{\ovm(\ga(t_{j}), V)} }$.
Define $I_{j} = [t_{j}, t_{j+1})$ for $j = 0, 1, \ldots, \ell-1$, and set $I_{\ell} = \brac{t_\ell, 1}$.
Observe that for $j = 0, 1, \ldots, \ell$,
$$\ga(t) \in B_j \text{ for all } t \in I_{j}.$$
In particular, Lemma \ref{muoBounds} implies that $\ovm\pr{\ga\pr{t}, V} \simeq_{(d, n, p, C_V)} \ovm\pr{\ga\pr{t_{j}}, V}$ whenever $t \in I_{j}$.
Moreover, for $j = 0, 1, \ldots, \ell-1$,
$$|\ga(t_{j+1}) - \ga(t_{j})| = \frac{1}{\ovm(\ga(t_{j}), V)}.$$
Thus,
\begin{align*}
\int_0^1 \ovm(\ga(t), V) |\ga'(t)| \, dt
&\ge \sum_{j = 0}^{\ell-1} \int_{I_j} \ovm(\ga(t), V) |\ga'(t)|\, dt
\gtrsim_{(d, n, p, C_V)} \sum_{j = 0}^{\ell-1} \ovm(\ga(t_{j}), V) \int_{I_j}  |\ga'(t)|\, dt \\
&\geq \sum_{j = 0}^{\ell-1} \ovm(\ga(t_{j}), V) |\ga(t_{j+1}) - \ga(t_{j})|
= \ell.
\end{align*}
Recalling how we defined $\ga$, this shows that
\begin{equation}
\label{ellBound}
\ell \le C_4 \, \ovd(x, y, V),
\end{equation}
where $C_4 = C_4\pr{d, n, p, C_V}$.

We defined $t_0$ so that whenever $t \ge t_0$, $|x - \ga(t)| \ge \frac{A}{\ovm(x, V)}$.
Therefore, by the choice of $A$ from \eqref{Adefn}, for each $j = 0, \ldots, \ell$, $x \not \in 2B_j$.
This means that if $U = \Ga^V(\cdot, x)$ then $\LV \V{u}_i = 0$ weakly on each $2B_j$.
Thus, repeated applications of the scale-invariant Harnack inequality from {\rm(SIH)} show that for any $\V{e} \in \Sd$,
\begin{align*}
\abs{\innp{U\pr{\ga\pr{t_0}} \V{e}, \V{e}}}
&\le \widetilde C_{\text{H}} \abs{\innp{U\pr{\ga\pr{t_1}} \V{e}, \V{e}}}
\le \ldots
\le \widetilde C_{\text{H}}^\ell \abs{\innp{U\pr{\ga\pr{t_\ell}} \V{e}, \V{e}}}
\le \widetilde C_{\text{H}}^{\ell+1} \abs{\innp{U\pr{\ga\pr{1}} \V{e}, \V{e}}},
\end{align*}
where $\widetilde C_{\text{H}} = C_{\text{H}}^\be$ and $\be$ depends on $c_S$ from {\rm(SIH)}.
Since $\ga(1) = y$, then
\begin{align*}
\abs{\innp{\Ga^V\pr{y, x} \V{e}, \V{e}}}
&\ge \widetilde C_{\text{H}}^{-\pr{\ell+1}} \abs{\innp{\Ga^V\pr{\ga(t_0), x} \V{e}, \V{e}}}
\ge \widetilde C_{\text{H}}^{-\pr{\ell+1}} \frac{C_3}{\abs{\ga(t_0) -x}^{n-2}},
\end{align*}
where \eqref{medLowerBdFlipped} was applicable since $|\ga(t_0) - x| \le \frac{A}{\ovm(x, V)}$.
Continuing on, since $|\ga(t_0) - x| < \abs{x-y}$, we get that for any $\V{e} \in \Sd$,
\begin{align*}
\abs{\innp{\Ga^V\pr{y, x} \V{e}, \V{e}}}
&\ge \frac{C_3 \exp\pr{-\ell \log \widetilde C_{\text{H}}}}{\widetilde C_{\text{H}} \abs{x-y}^{n-2}}
\ge \frac{C_3}{\widetilde C_{\text{H}} } \frac{\exp\pr{- C_4 \log \widetilde C_{\text{H}} \, \ovd(x, y, V) }}{\abs{x-y}^{n-2}},
\end{align*}
where we have applied \eqref{ellBound} in the final step.
As this bound is symmetric in $x$ and $y$, the conclusion \eqref{upperExpBound} follows.
\end{proof}

Finally, let us briefly discuss the connection between our upper and lower auxiliary functions and the Landscape functions that were mentioned in the introduction.

\begin{rem} 
\label{LandscapeRem}
For all $x \in \Rn$, define 
$$u(x) = \int_{\Rn} \abs{\Ga_V(x, y)} \, dy.$$
We decompose $\Rn$ into the disjoint union of the ball $B\pr{x, \frac{1}{\um\pr{x,V}}}$ and the annuli $B\pr{x, \frac{2^{j}}{\um\pr{x,V}}} \backslash B\pr{x, \frac{2^{j-1}}{\um\pr{x,V}}}$ for $j \in \N$, then assuming the conditions of Theorem \ref{UppBoundThm}, we argue as in Lemma \ref{localBoundedIntegrals} to show that $u(x) \lesssim  \um\pr{x,V}^{-2} $ for all $x \in \Rn$.  
On the other hand, for all $x \in \Rn$, Remark \ref{differentLowerBounds} tells us that (under appropriate conditions) 
$$u(x) \geq \int_{B\pr{x, \frac{1}{\ovm\pr{x,V}}} } \abs{\Ga_V(x, y)} \, dy \geq \int_{B\pr{x, \frac{1}{\ovm\pr{x,V}}} } \frac{e^{-\eps \ovd(x, y, V)}}{|x-y|^{n-2}} \, dy \gtrsim \ovm\pr{x,V}^{-2}.$$

As mentioned in the introduction, this connection was previously found in \cite{Po21} for scalar elliptic operators $\MC{L}_v$ with a nonnegative scalar potential $v$ on $\Rn$. 
In the scalar setting, it holds that $\ovm\pr{x,v}  = \um\pr{x,v}$ for all $x \in \Rn$.
If we denote this common function by $m(\cdot, v),$ it follows that $u(x) \simeq m(x, v)^{-2}$ for all $x \in \Rn$. 
Moreover, since the fundamental solution of such an operator is positive, we see that $u$ satisfies $\MC{L}_v u = 1$, which means that $u$ inherits desirable qualities that are not satisfied by $m(\cdot, v)$.     
We refer the reader to Theorems $1.18$ and $ 1.31$ in \cite{Po21} for additional details.
\end{rem}

\begin{appendix}

\section{The Noncommutativity Condition}
\label{Examples}

In this section, we are trying to further motivate the $\NC$ condition that was introduced in Section \ref{MWeights}.
To do this, we show that the set of matrix weights $\pr{\Bp \cap \ND} \setminus \NC$ is nonempty and that there is a matrix function in this space that fails to satisfy the Fefferman-Phong inequality described by Lemma \ref{FPml}.
As such, we hope to convince our reader that the additional assumption $V \in \NC$ is justified for our purposes.

We explore the properties of the matrix function that was introduced in Example \ref{notNCEx}.
With $x = \pr{x_1, \ldots, x_n} \in \R^n$, and $\abs{x} = \sqrt{x_1^2 + \ldots + x_n^2} \ge 0$, recall that $V : \R^n \to \R^{2 \times 2}$ is defined as
\begin{equation}
\label{VExDef}
V(x) = \begin{bmatrix}1 & \abs{x}^2 \\ \abs{x}^2 & \abs{x}^4 \end{bmatrix} = \begin{bmatrix}1 & x_1^2 + \ldots x_n^2 \\ x_1^2 + \ldots x_n^2 & \pr{x_1^2 + \ldots x_n^2}^2 \end{bmatrix}.
\end{equation}

We begin with a result regarding polynomial matrices.

\begin{prop}
\label{ExampleProp}
Let $V:\R^n \rightarrow \R^{d \times d}$ be any $d \times d$ matrix with polynomial entries.
Then $V^*V \in \Bp$ for every $p > 1$.
\end{prop}

\begin{proof}
First, if $P : \R^n \rightarrow \mathbb{\R}_{\ge 0}$ is any  nonnegative polynomial, then there exists $C > 0$, depending only on $n$ and the degree of $P$, so that for any cube $Q$
\begin{equation*}
\fint_Q |P(x)| \, dx \leq \sup_{x \in Q} |P(x)| \leq C \fint_Q |P(x)| \, dx.
\end{equation*}
The proof is from \cite{Fef83}: By the equivalence of norms on any finite dimensional vector space, the above equation is trivial if $Q$ is the unit cube centered at $0$.
The general case follows from dilation and translation.
Therefore, for any $p > 1$, $P$ is a scalar $\sBp$ function with a $\sBp$ constant that depends only on $n$ and the degree of $P$.

Let $k = \max\set{\deg \pr{V_{ij}}}_{i, j=1}^d$ and observe that for any $\V{e} \in \Rd$, $P(x; \V{e}) = \innp{V^* (x) V(x) \V{e}, \V{e}} $ is a nonnegative polynomial of degree at most $2k$.
By the conclusion of the previous paragraph, for any $p > 1$, $P(x; \V{e})$ belongs to $\sBp$ with a constant depending only on $n$ and the degree of $P(x; \V{e})$.
In particular, for any $p > 1$ and any $\V{e} \in \R^d$, $P(x; \V{e})$ belongs to $\sBp$ with a constant that is independent of $\V{e}$.
Since $V^*V$ is symmetric and positive semidefinite, then we conclude that $V^* V$ belongs to $\Bp$.
\end{proof}

We immediately have the following.

\begin{cor}
\label{ExampleCor}
Let $V:\R^n \rightarrow \R^{d \times d}$ be any $d \times d$, symmetric positive semidefinite matrix with polynomial entries.
Then $V \in \Bp$ for every $p > 1$.
\end{cor}

It follows that for $V$ as defined in \eqref{VExDef}, $V \in \Bp$.
Since $V$ also satisfies \eqref{NDCond}, then $V \in \ND$.
Now we show that $V$ is an example of positive definite polynomial matrix that doesn't satisfy our noncommutativity condition.

\begin{lem}[$\NC$ is a proper subset of $\Bp \cap \ND$]
\label{notNC}
For $V$ as defined in \eqref{VExDef}, $V \notin \NC$.
\end{lem}

\begin{proof}
A computation shows that
$$V^{\frac 1 2}(x) = \frac 1 {\sqrt{1 + \abs{x}^4}}\brac{\begin{array}{ll} 1 & \abs{x}^2 \\ \abs{x}^2 & \abs{x}^4 \end{array}}.$$
For any $x \in \R^n$,
\begin{align*}
\Psi\pr{x, r; V}
&= \frac 1 {r^{n-2}} \int_{Q(x, r)} V(y) dy
= \frac {1} {r^{n-2}} \int_{x_n - r}^{x_n+r} \ldots \int_{x_1 - r}^{x_1+r} \brac{\begin{array}{cc} 1 & y_1^2 + \ldots + y_n^2 \\ y_1^2 + \ldots + y_n^2 & \pr{y_1^2 + \ldots + y_n^2}^2 \end{array}} dy_1 \ldots dy_n.
\end{align*}
Computing, we have
\begin{align*}
\int_{Q(x, r)} \pr{y_1^2 + \ldots + y_n^2} dy
=& \sum_{j=1}^n \int_{x_n - r}^{x_n+r} \ldots \int_{x_1 - r}^{x_1+r} y_j^2 \, dy_1 \ldots dy_n
= \sum_{j=1}^n \pr{2r}^{n-1} \int_{x_j - r}^{x_j+r} y_j^2 \, dy_j \\
=& \sum_{j=1}^n \pr{2r}^{n-1} \pr{2 x_j^2 r + \frac 2 3 r^3}
= \pr{2r}^{n} \pr{\abs{x}^2 + \frac {n}3 r^2}.
\end{align*}
Define $\hat y_j = \left\{ \begin{array}{ll} \pr{y_2, \ldots, y_n} & j = 1 \\ \pr{y_1, \ldots, y_{j-1}, y_{j+1}, \ldots, y_n} & j = 2, \ldots, n-1 \\ \pr{y_1, \ldots, y_{n-1}} & j = n \end{array}\right. \in \R^{n-1}$ and let $Q_j = Q\pr{\hat x_j, r} \su \R^{n-1}$.
Then we have
\begin{align*}
 \int_{Q(x, r)} \pr{y_1^2 + \ldots + y_n^2}^2 dy
=& \sum_{j=1}^n \int_{x_n - r}^{x_n+r} \cdots \int_{x_1 - r}^{x_1+r} y_j^2 \pr{y_1^2 + \ldots + y_n^2} dy_1 \ldots dy_n \\
=& \sum_{j=1}^n \int_{x_n - r}^{x_n+r} \cdots \int_{x_1 - r}^{x_1+r} y_j^4 dy_1 \ldots dy_n
+ \sum_{j=1}^n \int_{x_n - r}^{x_n+r} \cdots \int_{x_1 - r}^{x_1+r} y_j^2 \abs{\hat y_j}^2 dy_1 \ldots dy_n \\
=&  \sum_{j=1}^n \pr{2r}^{n-1} \int_{x_j - r}^{x_j+r} y_j^4 \, dy_j
+ \sum_{j=1}^n \pr{\int_{x_j - r}^{x_j+r} y_j^2 \, dy_j} \brac{\int_{Q_j} \abs{\hat y_j}^2 d\hat y_j}
+ \ldots \\
=& \sum_{j=1}^n \pr{2r}^{n} \pr{x_j^4 + 2 x_j^2 r^2  + \frac 1 5 r^4}
+ \sum_{j=1}^n \pr{2r}^{n}  \pr{x_j^2 + \frac 1 3 r^2}  \pr{\abs{\hat x_j}^2 + \frac {n-1}3 r^2} \\
=& 2^n r^n \abs{x}^4
+ 2^{n+1} r^{n+2} \frac {n+ 2}3 \abs{x}^2
+ 2^n r^{n+4} \frac{\pr{5n+4}n}{45}.
\end{align*}
Therefore,
\begin{align*}
\Psi\pr{x, r; V}
&= \frac 1 {r^{n-2}} \int_{Q(x, r)} V(y) dy
= 2^n r^2 \brac{\begin{array}{cc} 1 &  \abs{x}^2 + \frac {n}3 r^2 \\ \abs{x}^2 + \frac {n}3 r^2 & \abs{x}^4 + 2 \frac {n+ 2}3 \abs{x}^2 r^{2} +  \frac{\pr{5n+4}n}{45}r^{4} \end{array}}.
\end{align*}
The characteristic polynomial of the inner matrix is
\begin{align*}
& \la^2 - \la \pr{1 + \abs{x}^4 + \frac {2n+ 4}3 \abs{x}^2 r^{2} + \frac{5n^2+4n}{45}r^{4}}
 + \frac {4}3 \abs{x}^2 r^{2} +  \frac{4n}{45}r^{4}
\end{align*}
so its eigenvalues are
\begin{align*}
\frac{\pr{1 + \abs{x}^4 + \frac {2n+ 4}3 \abs{x}^2 r^{2} + \frac{5n^2+4n}{45}r^{4}}  \pm \sqrt{\pr{1 + \abs{x}^4 + \frac {2n+ 4}3 \abs{x}^2 r^{2} + \frac{5n^2+4n}{45}r^{4}}^2 - 4\pr{\frac {4}3 \abs{x}^2 r^{2} +  \frac{4n}{45}r^{4}} }}{2}.
\end{align*}
We choose $\ur = \ur(x)$ optimally so that $\Psi\pr{x, \ur; V} \ge I$.
That is,
$$2^{n-1} \ur^2 \pr{1 + \abs{x}^4 + \frac {2n+ 4}3 \abs{x}^2 \ur^{2} + \frac{5n^2+4n}{45}\ur^{4}}   \brac{1 - \sqrt{1 - 4\frac{\frac {4}3 \abs{x}^2 \ur^{2} +  \frac{4n}{45}\ur^{4}}{\pr{1 + \abs{x}^4 + \frac {2n+ 4}3 \abs{x}^2 \ur^{2} + \frac{5n^2+4n}{45}\ur^{4}}^2} }} = 1.$$
When $\abs{x} \gg 1$, we can perform a Taylor expansion of the root term, and we then have to solve
$$2^{n} \ur^2 \pr{\frac {4}3 \abs{x}^2 \ur^{2} +  \frac{4n}{45}\ur^{4}} \approx \pr{1 + \abs{x}^4 + \frac {2n+ 4}3 \abs{x}^2 \ur^{2} + \frac{5n^2+4n}{45}\ur^{4}} $$
or
$$\abs{x}^4 - \pr{\frac {2^{n+2}}3 \ur^2 - \frac {2n+ 4}3} \ur^{2} \abs{x}^2 + \frac{5n^2+4n}{45}\ur^{4} - \frac{2^{n+2}n}{45}\ur^{6} + 1 \approx 0.$$
We see that a solution is given by
\begin{align*}
\abs{x}^2 &\approx \frac {2^{n+1}}3 \ur^4 - \frac {n+ 2}3 \ur^{2} + \sqrt{\pr{\frac {2^{n+1}}3 \ur^4 - \frac {n+ 2}3 \ur^{2}}^2 + \pr{\frac{2^{n+2}n}{45}\ur^{6} - \frac{5n^2+4n}{45}\ur^{4} - 1}}.
\end{align*}
In particular, $\ur \to \iny$ when $\abs{x} \to \iny$.
Therefore, we may construct an increasing sequence $\set{R_m}_{m=1}^\iny$ such that whenever $\abs{x_m} = R_m$, $\ur(x_m)^2 = m$.
That is,
\begin{align*}
& 2^{n-1} m \pr{1 + \abs{x_m}^4 + \frac {2n+ 4}3 \abs{x_m}^2 m + \frac{5n^2+4n}{45}m^2}   \brac{1 - \sqrt{1 - \tfrac{4 \pr{\frac {4}3 \abs{x_m}^2 m +  \frac{4n}{45}m^2}}{\pr{1 + \abs{x_m}^4 + \frac {2n+ 4}3 \abs{x_m}^2 m + \frac{5n^2+4n}{45}m^2}^2} }} = 1
\end{align*}
so that
\begin{align*}
\abs{x_m}^4 - \pr{ \frac {2^{n+2}}3 m^2 - \frac {2n+ 4}3 m}\abs{x_m}^2 - \pr{\frac{2^{n+2}n}{45}m^3 - \frac{5n^2+4n}{45}m^2 - 1}
&\approx 0
\end{align*}
and then
\begin{align*}
\abs{x_m}^2 & \approx \pr{\frac {2^{n+1}}3 m^2 - \frac {n+ 2}3 m} + \sqrt{\pr{\frac {2^{n+1}}3 m^2 - \frac {n+ 2}3 m}^2 + \pr{\frac{2^{n+2}n}{45}m^3 -  \frac{5n^2+4n}{45}m^2 - 1}}.
\end{align*}
Therefore, $\abs{x_m} = c_m m$, where $\set{c_m}_{m=1}^\iny$ is bounded.
With $r_m = \ur(x_m)$, we have $r_m = \sqrt{m}$.
In particular, $\frac{r_m}{\abs{x_m}} \to 0$ as $m \to \iny$.

Now we set $Q_m = Q(x_m, r_m)$ and calculate
\begin{align*}
V(Q_m)
&= \int_{Q_m} V(y) dy
= \pr{2 r_m}^n \brac{\begin{array}{cc} 1 & \abs{x_m}^2 + \frac {n}3 r_m^2 \\ \abs{x_m}^2 + \frac {n}3 r_m^2 & \abs{x_m}^4 + \frac {2n+ 4}3 \abs{x_m}^2 r_m^{2} + \frac{\pr{5n+4}n}{45} r_m^{4} \end{array}}
\end{align*}
so that
\begin{align*}
V(Q_m)^{-1}
&= \frac{3 r_m^{-n-2}}{2^{n+2}\pr{\abs{x_m}^2  + \frac{n}{15} r_m^2}} \brac{\begin{array}{cc} \abs{x_m}^4 + \frac {2n+ 4}3 \abs{x_m}^2 r_m^{2} + \frac{\pr{5n+4}n}{45} r_m^{4} & -\pr{\abs{x_m}^2 + \frac {n}3 r_m^2} \\ -\pr{\abs{x_m}^2 + \frac {n}3 r_m^2} &  1 \end{array}}.
\end{align*}
Then
\begin{align*}
& \frac{2^{n+2}}{3} r_m^{n+2} \pr{\abs{x_m}^2  + \frac{n}{15} r_m^2}\pr{1 + \abs{y}^4} V(y)^{\frac 1 2} V(Q_m)^{-1} V(y)^{\frac 1 2} \\
=& \brac{\begin{array}{ll} 1 & \abs{y}^2 \\ \abs{y}^2 & \abs{y}^4 \end{array}}
\brac{\begin{array}{cc} \abs{x_m}^4 + \frac {2n+ 4}3 \abs{x_m}^2 r_m^{2} + \frac{\pr{5n+4}n}{45} r_m^{4} & -\pr{\abs{x_m}^2 + \frac {n}3 r_m^2} \\ -\pr{\abs{x_m}^2 + \frac {n}3 r_m^2} &  1 \end{array}}
\brac{\begin{array}{ll} 1 & \abs{y}^2 \\ \abs{y}^2 & \abs{y}^4 \end{array}}
\\
=& \pr{\abs{x_m}^4 + \frac {2n+ 4}3 \abs{x_m}^2 r_m^{2} + \frac{\pr{5n+4}n}{45} r_m^{4} - 2 \pr{\abs{x_m}^2 + \frac {n}3 r_m^2} \abs{y}^2 + \abs{y}^4} \brac{\begin{array}{ll} 1 & \abs{y}^2 \\ \abs{y}^2 & \abs{y}^4 \end{array}}
\\
\end{align*}
and we see that
\begin{align*}
& \frac{2^{n+2}}{3} r_m^{n+2} \pr{\abs{x_m}^2  + \frac{n}{15} r_m^2} \innp{V(y)^{\frac 1 2} V(Q_m)^{-1} V(y)^{\frac 1 2} \V{e}_1, \V{e}_1} \\
=& \frac{\abs{x_m}^4 + \frac {2n+ 4}3 \abs{x_m}^2 r_m^{2} + \frac{\pr{5n+4}n}{45} r_m^{4} - 2 \pr{\abs{x_m}^2 + \frac {n}3 r_m^2} \abs{y}^2 + \abs{y}^4}{1 + \abs{y}^4}.
\end{align*}
Therefore,
\begin{align*}
&\frac{4 r_m^{2}}{3} \pr{\abs{x_m}^2  + \frac{n}{15} r_m^2} \int_{Q_m} \innp{V(y)^{\frac 1 2} V(Q_m)^{-1} V(y)^{\frac 1 2} \V{e}_1, \V{e}_1} dy \\
=& \pr{\abs{x_m}^4 + \frac {2n+ 4}3 \abs{x_m}^2 r_m^{2} + \frac{\pr{5n+4}n}{45} r_m^{4} - 1}\fint_{Q_m} \frac{1}{1 + \abs{y}^4} dy
- 2 \pr{\abs{x_m}^2 + \frac {n}3 r_m^2}  \fint_{Q_m} \frac{\abs{y}^2}{1 + \abs{y}^4} dy
+ 1
 \\
\le& \pr{\abs{x_m}^4 + \frac {2n+ 4}3 \abs{x_m}^2 r_m^{2} + \frac{\pr{5n+4}n}{45} r_m^{4} - 1} \frac{1}{1 + \pr{\abs{x_m} - r_m}^4}
+ 2 \pr{\abs{x_m}^2 + \frac {n}3 r_m^2} \frac{\pr{\abs{x_m} + r_m}^2}{1 + \pr{\abs{x_m} - r_m}^4}
+ 1 \\
=& \frac{4 + \pr{\frac {4n+ 28}3} \pr{\frac{r_m}{\abs{x_m}}}^{2} + \pr{\frac{4n-12}{3}} \pr{\frac{r_m}{\abs{x_m}}}^3 + \brac{\frac{\pr{5n+34}n}{45}+ 1} \pr{\frac{r_m}{\abs{x_m}}}^{4} }{1 - 4 \pr{\frac{r_m}{\abs{x_m}}} + 6 \pr{\frac{r_m}{\abs{x_m}}}^2 - 4 \pr{\frac{r_m}{\abs{x_m}}}^3 + \pr{\frac{r_m}{\abs{x_m}}}^4 + \abs{x_m}^{-4}}.
\end{align*}
It follows that
\begin{align*}
&\lim_{m \to \iny}\int_{Q_m} \innp{V(y)^{\frac 1 2} V(Q_m)^{-1} V(y)^{\frac 1 2} \V{e}_1, \V{e}_1} dy \\
\le& \lim_{m \to \iny} \frac{3}{4 r_m^{2}\pr{\abs{x_m}^2  + \frac{n}{15} r_m^2}} \pr{ \frac{4 + \pr{\frac {4n+ 28}3} \pr{\frac{r_m}{\abs{x_m}}}^{2} + \pr{\frac{4n-12}{3}} \pr{\frac{r_m}{\abs{x_m}}}^3 + \brac{\frac{\pr{5n+34}n}{45}+ 1} \pr{\frac{r_m}{\abs{x_m}}}^{4} }{1 - 4 \pr{\frac{r_m}{\abs{x_m}}} + 6 \pr{\frac{r_m}{\abs{x_m}}}^2 - 4 \pr{\frac{r_m}{\abs{x_m}}}^3 + \pr{\frac{r_m}{\abs{x_m}}}^4 + \abs{x_m}^{-4}}} \\
=& 0.
\end{align*}
In particular, there is no constant $c > 0$ so that for all cubes $Q = Q(x, \frac 1 {\um(x, V)})$ and all $\V{e} \in \R^d$,
$$\int_{Q} \innp{V(y)^{\frac 1 2} V(Q)^{-1} V(y)^{\frac 1 2} \V{e}, \V{e}} dy \ge c \abs{\V{e}},$$
showing that this choice of $V \in \Bp$ does not belong to $\mathcal{NC}$.
\end{proof}

Next, we show that this choice of $V$ violates the Fefferman-Phong inequality described by Lemma \ref{FPml}.

\begin{lem}[Failure of the Fefferman-Phong inequality]
For $V$ as defined in \eqref{VExDef}, there is no choice of constant $C$ so that for every $\V{u} \in C^1_0\pr{\R^n}$, it holds that
\begin{align}
\label{FPInq}
\int_{\R^n} \um\pr{x, V}^2 \abs{\vec{u}}^2
\le C \pr{\int_{\R^n} \abs{D\vec{u}}^2  + \int_{\R^n} \innp{V \vec{u}, \vec{u}}}.
\end{align}
\end{lem}

\begin{proof}
We will construct a sequence $\set{\V{u}_R} \su C^1_0\pr{\R^n}$ that violates this inequality as $R \to \iny$.
With $\disp \vec{u} = \brac{- \abs{x}^2, 1}^T$, we see that  $V \vec{u} = \vec{0}$.
For any $R \gg 1$, define $\xi_R \in C^\iny_0\pr{\R^n}$ so that $\xi_R \equiv 1$ when $x \in B_{2R} \setminus B_R$ and $\supp \xi_R \su B_{3R} \setminus B_{R/2}$.
In particular, $\supp \gr \xi_R \su \pr{B_{3R} \setminus B_{2R}} \cup \pr{B_R \setminus B_{R/2}}$ with $\abs{\gr \xi_R} \lesssim \frac 1 R$.
Then if we define $\vec{u}_R = \vec{u} \xi_R$, we see that $\vec{u}_R \in C^1_0\pr{\R^n}$ and for any choice of $R > 0$, $V \vec{u}_R = \vec{0}$.
In particular,
$$\int_{\R^n} \innp{V \vec{u}_R, \vec{u}_R} = 0.$$
Now
\begin{align*}
D\vec{u}_R
= \brac{\begin{array}{c} - 2 \vec{x} \xi_R \\ 0\end{array}}
+  \brac{\begin{array}{c} - \abs{x}^2 \gr \xi_R \\ \gr \xi_R \end{array}}
\end{align*}
so that
\begin{align*}
\int_{\R^n} \abs{D\vec{u}_R}^2
&\le \int_{B_{3R} \setminus B_{R/2}} 4 \abs{x}^2
+ \int_{\pr{B_{3R} \setminus B_{2R}} \cup \pr{B_R \setminus B_{R/2}}} \pr{\abs{x}^4 +1}\abs{\gr \xi_R}^2
\lesssim \abs{B_{3R}} R^2 \simeq R^{n+2}
\end{align*}
and then
\begin{equation}
\label{RHS}
\int_{\R^n} \abs{D\vec{u}}^2  + \int_{\R^n} \innp{V \vec{u}, \vec{u}}
\lesssim_{(d, n)} R^{n+2}.
\end{equation}
Recall from the proof of Lemma \ref{notNC} that there is a bounded sequence $\set{c_m}_{m=1}^\iny \su \R$ so that if $\abs{x_m} = c_m m$ for all $m \in \N$, then $\ur(x_m) = \sqrt{m}$.
In other words, $\ur(x_m) = \sqrt{\frac{\abs{x_m}}{c_m}}$.
Since $\ur(x) = \frac{1}{\um(x)}$, then we conclude that $\um(x) \simeq \sqrt{ \frac{1}{\abs{x}}}$ whenever $\abs{x} \gg 1$.
Thus, we see that
\begin{align}
\label{LHS}
\int_{\R^n} \um\pr{x, V}^2 \abs{\vec{u}_R}^2
\ge \int_{B_{2R} \setminus B_R} \um\pr{x, V}^2 \abs{\vec{u}_R}^2
\gtrsim \int_{B_{2R} \setminus B_R} \um\pr{R, V}^2 R^4
\simeq R^{n+3}.
\end{align}
If \eqref{FPInq} were to hold, then there is a $C > 0$ so that $R^{n+3} \le C R^{n+2}$ for all $R \gg 1$.
As this is impossible, the proof is complete.
\end{proof}

\section{The $\Atwi$, $\Ai$, $\RBM,$ and $\Bp$ Classes of Matrices.}
\label{AiApp}

The goal of this appendix is to provide precise and concrete connections between the classes of matrix weights that were introduced in Section \ref{MWeights}: $\Atwi$ (and more generally $\Api$, which will be defined momentarily), $\RBM$, $\Ai,$  and $\Bp$.
Further, we will make our presentation almost entirely self-contained for the reader who is unfamiliar with the theory of $\Api$ matrix weights.

Throughout this section, unless otherwise stated, we assume that $1 < p < \infty$.
Let $V$ be a complex-valued matrix weight defined on $\Rn$; that is, $V$ is a Hermitian positive semidefinite $d \times d$ matrix function with $\abs{V} \in L_{\T{loc}}^1 (\Rn)$.
Note that in the body of this paper, we assume that $V$ is real-valued and symmetric.
As pointed out in the introduction, for our purposes, there is no loss of generality in replacing ``complex Hermitian" with ``real symmetric".
However, here within the appendix, we follow the standard convention in matrix weight theory and work with complex Hermitian matrix weights.
It should be noted that a matrix weight $V$ is (unless otherwise stated) not necessarily positive definite a.e.  on $\Rn$.

We begin by proving some useful lemmas regarding matrix weights.
First we have what is known as the ``matrix Jensen's inequality" from \cite{NT96, Vol97}.
For the sake of completeness, we include the proof.

\begin{lem}[Matrix Jensen inequality]
\label{MatrixJensen}
If $V$ is a positive definite matrix weight, then for any cube $Q \su \R^n$,
\begin{equation*}
\det \fint_Q V(x) \, dx \geq \exp \pr{\fint_Q \ln \det V (x) \, dx}.
\end{equation*}
\end{lem}

\begin{proof}
Let $A$ be a matrix with $|\det A| = 1$.
If $W$ is a Hermitian, positive definite $d \times d$ matrix, then the classical arithmetic mean-geometric mean inequality shows that
\begin{equation*}
(\det W)^\frac{1}{d} = \brac{\det (A ^* W A)}^\frac{1}{d} \leq \frac{1}{d} \tr (A ^* W A),
\end{equation*}
with equality when $A = (\det W)^\frac{1}{2d} W^{-\frac12}$.
In particular, when $W$ is positive definite,
\begin{equation*}
(\det W)^\frac{1}{d} = \inf \set{\frac{1}{d} \tr (A ^* W A) : |\det A| = 1\ }.
\end{equation*}

Thus, we have
\begin{equation}
\label{DetConvexIneq}
\begin{aligned}
\pr{\det  \fint_Q V(x) \, dx  }^\frac{1}{d}
&= \inf_{|\det A| = 1}  \frac{1}{d} \tr A ^* \pr{\fint_Q V(x) \, dx}  A
= \inf_{|\det A| = 1}  \frac{1}{d} \tr \fint_Q A ^*   V(x)  A \, dx \\
&\geq \fint_Q \pr{ \inf_{|\det A| = 1}  \frac{1}{d} \tr  A ^*   V(x) A } \, dx
= \fint_Q \brac{\det V(x)}^\frac{1}{d}  \, dx.
\end{aligned}
\end{equation}
Combining \eqref{DetConvexIneq} with Jensen's inequality then gives us
$$\pr{\det  \fint_Q V(x) \, dx  }^\frac{1}{d} \geq \exp \pr{\fint_Q \ln \brac{\det V(x)}^\frac{1}{d} \, dx},$$
which completes the proof.
\end{proof}

Next we state and prove a well-known result that is sometimes called the ``Hadamard determinant inequality".

\begin{lem}[Determinant lemma]
\label{DetProp}
If $A$ is a positive semidefinite Hermitian matrix and $\{\V{e}_j\}_{j=1}^d$ is any orthonormal basis of $\C^d$, then $\disp \det A \leq \prod_{j = 1}^d \innp{A \V{e}_j, \V{e}_j} \leq \prod_{j = 1}^d \abs{A \V{e}_j}$.
\end{lem}

\begin{proof}
The proof that we present is from \cite{Bow01}.
The second inequality follows immediately from the first, which we now prove.
Let $A$ have eigenvalues $\{\lambda_j\}_{j=1}^d$ with a corresponding orthonormal basis of eigenvectors $\{\V{f}_j\}_{j=1}^d$ so that
$$\innp{A \V{e}_i, \V{e}_i} = \sum_{ j = 1}^d  \lambda_j \abs{\innp{\V{e}_i, \V{f}_j}}^2  = \sum_{ j = 1}^d  \lambda_j (f_i ^j)^2,$$ where we have set $\abs{\innp{\V{e}_i, \V{f}_j}} = f_j ^i. $
Using the weighted arithmetic-geometric mean inequality, we have that
\begin{align*}
\det A & = \prod_{j = 1}^d \lambda_j
= \prod_{j = 1}^d \lambda_j ^{\sum_{i = 1}^d (f_j ^i)^2}
=  \prod_{i = 1}^d \pr{ \prod_{j = 1}^d  \lambda_j ^{(f_j ^i)^2}}
\leq \prod_{i = 1}^d \pr {\sum_{j = 1}^d \lambda_j (f_j ^i)^2 }
= \prod_{i = 1}^d \innp{A \V{e}_i, \V{e}_i} ,
\end{align*}
as required.
\end{proof}

\begin{defn}[$p$-nondegenerate]
We say that a matrix weight $V$ is {\bf $p$-nondegenerate} if for every $\V{e} \in \Cd$, it holds that
\begin{equation*}
\label{pNonDeg}
\abs{V^\frac{1}{p} (x) \V{e}} > 0 \quad \text{a.e. on} \;\; \Rn.
\end{equation*}
\end{defn}

In the setting where $V$ is $p$-nondegenerate, for any cube $Q$, the map $\disp \V{e} \mapsto \pr{\fint_Q \abs{V^\frac{1}{p} (x) \V{e}} ^p \, dx}^\frac{1}{p}$ defines a norm on $\Cd$.
Thus, the John ellipsoid theorem implies the existence of a ``reducing matrix", defined as follows.

\begin{defn}[Reducing matrix]
If $V$ is a $p$-nondegenerate matrix weight, then for every cube $Q \su \R^n$, there exists a positive definite, Hermitian $d \times d$ matrix $R_Q ^p (V)$, called a {\bf reducing matrix}.
This matrix $R_Q ^p (V)$ has the property that for any $\vec e \in \Cd$,
\begin{equation}
\label{reducingDef}
\pr{\fint_Q \abs{V^{\frac 1 p}(x) \vec e}^p dx}^{\frac 1 p} \leq  \abs{R_Q^p(V) \vec e}
\leq \sqrt{d} \pr{\fint_Q \abs{V^{\frac 1 p}(x) \vec e}^p dx}^{\frac 1 p}.
\end{equation}
\end{defn}

\noindent
See \cite{NT96}, p. 79 for a proof with the same lower bound and a slightly worse upper bound of $d$.
The reducing matrix need not be unique, but the choice of $R_Q ^p (V)$ is insignificant.

Note that if $p = 2$, then a computation shows that
\begin{equation*}
\abs{\pr{\fint_Q V(y) dy}^\frac12 \V{e}}^2
= \fint_Q \abs{V^\frac12 (x) \V{e}}^2 \, dx.
\end{equation*}
That is, if $p = 2$, then $\disp \pr{\fint_Q V}^\frac12$ is a reducing matrix for $V$.

Also, observe that
\begin{equation}
\label{reducingpCons}
\begin{aligned}
\pr{\fint_Q \innp{V(x) \vec e, \V{e}}^{p} dx}^{\frac 1{2p}}
& = \pr{\fint_Q \abs{V^{\frac 1 2}(x) \vec e}^{2p} dx}^{\frac 1{2p}}
\leq  \abs{R_Q^{2p} (V^p) \vec e} \\
& \leq \sqrt{d} \pr{\fint_Q \abs{V^{\frac 1 2}(x) \vec e}^{2p} dx}^{\frac 1{2p}}
= \sqrt d \pr{\fint_Q \innp{V(x)\vec e, \V{e}}^{p} dx}^{\frac 1{2p}},
\end{aligned}
\end{equation}
showing that $R_Q^{2p} (V^p)$ is a reducing matrix for the norm
$\disp e \mapsto \pr{\fint_Q \innp{V(x)\vec e, \V{e}}^{p} dx}^{\frac 1{2p}}$.

We now introduce the $\Api$ class from \cite{NT96, Vol97}.
In contrast to the scalar setting where there is a single class of $\sAi$ weights, in the matrix setting, there is such a class for each $p$.

\begin{defn}[$\Api$]
We say that $V \in \Api$\, if $V$ is a $p$-nondegenerate matrix weight and there exists a constant $A_{V} = A_{V,p} > 0$ so that for every cube $Q \su \R^n$, it holds that
\begin{equation}
\label{Apinfone}
\det R_Q ^p (V) \le A_{V} \exp \pr{ \fint_Q \ln \det V^\frac{1}{p} (x) \, dx}.
\end{equation}
\end{defn}

For example, when $p = 2$, it follows from the observation above that $V \in \Atwi$ if there exists a constant $A_V > 0$ so that for every cube $Q \su \R^n$, we have
\begin{equation}
\label{Apinfone2}
\det \fint_Q V  \le A_V \exp \pr{ \fint_Q \ln \det V (x) \, dx}.
\end{equation}

We now prove that a matrix weight $V \in \Api$ is in fact positive definite a.e., which appears to previously unknown.

\begin{prop}
Let $V \in \Api$ be a matrix weight.
If $p \geq 2$, then $\pr{\det V}^{\frac{2}{dp}} \in \text{A}_\infty$, while if $1 < p < 2$, then $\pr{\det V}^\frac{1}{dp} \in \text{A}_\infty$.
In particular, $V$ is positive definite a.e.
\end{prop}

\begin{proof}
For some $Q \su \R^n$, let $R_Q ^p(V)$ be a reducing matrix for $V$ and take $\set{\V{e}_j}_{j=1}^d$ to be an orthonormal basis of eigenvectors of $R_Q ^p(V)$.

If $p \geq 2$, then applying \eqref{DetConvexIneq} to $V^\frac{2}{p}$ and using Lemma \ref{DetProp} shows that
\begin{align*}
\fint_Q \brac{\det V(x)}^{\frac{2}{dp}} dx
&\leq \pr{\det \fint_Q V^\frac{2}{p}(x) dx}^\frac{1}{d}
\leq \brac{\prod_{j = 1}^d \innp{\pr{\fint_Q V^\frac{2}{p}(x) dx} \V{e}_j, \V{e}_j}}^\frac{1}{d}
= \brac{\prod_{j = 1}^d {\fint_Q \innp{V^\frac{2}{p}(x) \V{e}_j, \V{e}_j} dx} }^\frac{1}{d} \\
&= \brac{\prod_{j = 1}^d {\fint_Q \abs{V^{\frac{1}{p}}(x) \V{e}_j}^2 dx}}^\frac{1}{d}
\leq \brac{\prod_{j = 1}^d {\fint_Q \abs{V^{\frac{1}{p}}(x) \V{e}_j}^p dx}}^\frac{2}{d p},
\end{align*}
where the last inequality follows form H\"{o}lder's inequality.
Applications of \eqref{reducingDef}, that  $\{\V{e}_j\}$ is an orthonormal basis of eigenvectors of $R_Q ^p(V)$, and \eqref{Apinfone} then give us that
\begin{align*}
\fint_Q \brac{\det V(x)}^{\frac{2}{dp}} dx
&\leq  \brac{\prod_{j = 1}^d \pr{\fint_Q \abs{V^{\frac{1}{p}}(x) \V{e}_j}^p dx}^{\frac 1 p}}^\frac{2}{d}
\le \brac{\prod_{j = 1}^d \abs{R_Q ^p(V) \V{e}_j}}^\frac{2}{d}
= \brac{\det R_Q ^p(V)} ^\frac{2}{d} \\
&\lesssim \exp \pr{\fint_Q \ln \brac{\det V(x)}^\frac{2}{dp}  \, dx }.
\end{align*}
By the classical reverse Jensen characterization of scalar $\sAi$ weights (see \cite[Theorem 7.3.3]{Gra14}), we conclude that $\pr{\det V}^{\frac{2}{dp}} \in \sAi$.

If $p \in \pr{1, 2}$, then applying the argument above to $V^{\frac 1 p}$ shows that
\begin{align*}
\fint_Q \brac{\det V(x)}^{\frac{1}{dp}} dx
&\leq \brac{\prod_{j = 1}^d {\fint_Q \innp{V^\frac{1}{p}(x)\V{e}_j, \V{e}_j} dx} }^\frac{1}{d}
\leq \brac{\prod_{j = 1}^d {\fint_Q \abs{V^{\frac{1}{p}}(x) \V{e}_j} dx}}^\frac{1}{d}
\leq \brac{\prod_{j = 1}^d {\fint_Q \abs{V^{\frac{1}{p}}(x) \V{e}_j}^p dx}}^\frac{1}{d p}  \\
&\le \brac{\prod_{j = 1}^d \abs{R_Q ^p(V) \V{e}_j}}^\frac{1}{d}
 = \brac{\det R_Q ^p(V)} ^\frac{1}{d}
\lesssim \exp \pr{\fint_Q \ln \brac{\det V(x)}^\frac{1}{dp}  \, dx },
\end{align*}
from which it follows that $\pr{\det V}^\frac{1}{dp} \in \T{A}_\infty$.
\end{proof}

Now we give another characterization of the $\Api$ class of matrices from \cite{Vol97}.

\begin{lem}[$\Api$ characterization]
\label{ApiProperty}
Let $V$ be a $p$-nondegenerate matrix weight.
Then $V \in \Api$ iff $V$ is positive definite and there exists a constant $C > 0$ so that for every $Q \su \R^n$ and every $\V{e} \in \C^d$, it holds that
\begin{equation}
\label{Apinftwo}
\exp\pr{\fint_Q \ln |V^{-\frac{1}{p}} (x) \V{e} | \, dx}  \le C \abs{ (R_Q ^p(V))^{-1}  \V{e} }.
\end{equation}
\end{lem}

This result was originally proved in \cite[p. 451]{Vol97}.

\begin{proof}
Let $Q \su \R^n$ be arbitrary and assume that \eqref{Apinftwo} holds for every $\V{e} \in \C^d$.
Let $\set{\V{e}_i}_{i = 1}^d$ be an orthonormal basis of eigenvectors for $R_Q ^p(V)$, and consequently for $R_Q ^p(V)^{-1}$.
For each $i = 1, \ldots, d$, taking the logarithm of the assumption \eqref{Apinftwo} shows that
\begin{equation*}
\fint_Q \ln |V^{-\frac{1}{p}} (x) \V{e}_i | \, dx
\leq \ln C  + \ln \abs{(R_Q ^p(V))^{-1}  \V{e}_i }.
\end{equation*}
Applying Lemma \ref{DetProp}, the inequality on the previous line, and then summing gives
\begin{align*}
\fint_Q \ln \det V^{-\frac{1}{p}} (x) \, dx
&\leq \sum_{i = 1} ^d \fint_Q \ln |V^{-\frac{1}{p}} (x) \V{e}_i | dx
\leq d\ln C + \sum_{i = 1} ^d  \ln \abs{ (R_Q ^p(V))^{-1}  \V{e}_i } \\
&= d\ln C + \ln \pr{\prod_{i = 1}^d   \abs{ (R_Q ^p(V))^{-1}  \V{e}_i } }
= d \ln C + \ln \det (R_Q ^p(V))^{-1}.
\end{align*}
After rearrangement, this is equivalent to \eqref{Apinfone}.
Since $Q$ was arbitrary, it follows that $V \in \Api$.

Now assume that $V \in \Api$.
It follows from \eqref{Apinfone} that for any $Q \su \R^n$,
\begin{equation}
\label{ApinfoneA}
\fint_Q \ln \det \brac{V^{-\frac{1}{p}} (x) R_Q ^p(V)} \, dx \leq c.
\end{equation}
Define the matrix $B(x) = R_Q ^p(V) V^{-\frac{2}{p}} (x)  R_Q ^p(V)$ and let $0 < \lambda_1(x) \leq \cdots \leq \lambda_d(x)$ be the eigenvalues of $B(x)$ with corresponding normalized eigenvectors $\V{e}_1 (x), \ldots, \V{e}_d (x)$.
Thus, for any fixed unit vector $\V{e}$, we have
\begin{equation}
\label{AinfEstOne}
\begin{aligned}
\det  \brac{V^{-\frac{1}{p}} (x) R_Q ^p(V)}
& = \brac{\det B(x)} ^\frac12
= \pr{\prod_{i = 1}^d \innp{B(x) \V{e}_i (x), \V{e}_i (x)}}^\frac12 \\
& \geq  \pr{\innp{B(x) \V{e}, \V{e}} \prod_{i = 1}^{d-1}  \innp{B(x) \V{e}_i (x), \V{e}_i (x)}}^\frac12
= \prod_{i = 1}^{d}  \abs{V^{-\frac{1}{p}} (x)  R_Q ^p(V) \V{f}_i (x)},
\end{aligned}
\end{equation}
where we have set $\V{f}_i (x) = \V{e}_i (x)$ for $i = 1, \ldots, d-1$, and $\V{f}_d (x)= \V{e}$.

However, for any (constant) orthonormal basis $\set{\V{g}_j}_{j=1}^d$ of $\Cd$, we have
\begin{align*}
\fint_Q \abs{R_Q ^p(V) ^{-1} V^{\frac{1}{p}} (x) \V{f}_i (x)}^p \, dx
&\lesssim_{(p)} \sum_{j = 1}^d  \fint_Q \abs{ \innp{  \V{f}_i (x),  V^{\frac{1}{p}} (x) R_Q ^p(V) ^{-1} \V{g}_j }} ^p \, dx\\
&\leq  \sum_{j = 1}^d  \fint_Q \abs{  V^{\frac{1}{p}} (x) R_Q ^p(V) ^{-1} \V{g}_j } ^p \, dx
\leq \sum_{j = 1}^d   \abs{  R_Q ^p(V) R_Q ^p(V) ^{-1} \V{g}_j } ^p
= d.
 \end{align*}
Therefore,
\begin{equation}
\label{AinfEstTwo}
\begin{aligned}
\sum_{i = 1}^d \fint_Q \ln ^+ \abs{R_Q ^p(V) ^{-1} V^{\frac{1}{p}} (x) \V{f}_i (x)} \, dx
&= \frac{1}{p} \sum_{i = 1}^d \fint_Q \ln ^+ \abs{R_Q ^p(V) ^{-1} V^{\frac{1}{p}} (x) \V{f}_i (x)}^p \, dx  \\
&\leq \frac{1}{p} \sum_{i = 1}^d \fint_Q  \abs{R_Q ^p(V) ^{-1} V^{\frac{1}{p}} (x) \V{f}_i (x)}^p \, dx \leq c(d, p).
\end{aligned}
\end{equation}

Note that for any invertible matrix $A$ and any unit vector $\V{c}$ we have $1 = \innp{A \V{c}, A^{-1} \V{c}}  \leq  \abs{A \V{c}} \abs{A^{-1} \V{c}}$.
In particular,
\begin{equation}
\label{InvIneq}
\abs{A \V{c}}^{-1} \leq \abs{A^{-1} \V{c}} .
\end{equation}

Using that $\V{f}_d = \V{e}$ and $\ln \le \ln^+$, the fact that $\ln x = \ln ^+ x - \ln ^+ x^{-1}$ followed by applications of \eqref{AinfEstOne}, \eqref{InvIneq}, \eqref{ApinfoneA}, and \eqref{AinfEstTwo} shows that
\begin{align*}
\fint_Q \ln \abs{V^{-\frac{1}{p}} (x) R_Q ^p(V)   \V{e} } \, dx
&\leq \sum_{i = 1}^d \fint_Q \ln^+  \abs{V^{-\frac{1}{p}} (x) R_Q ^p(V)   \V{f}_i (x)} \, dx \\
&= \sum_{i = 1}^d \fint_Q \ln  \abs{V^{-\frac{1}{p}} (x) R_Q ^p(V)   \V{f}_i (x)} \, dx
+ \sum_{i = 1}^d \fint_Q \ln^+  \abs{V^{-\frac{1}{p}} (x) R_Q ^p(V)   \V{f}_i (x)}^{-1}  \, dx \\
&\leq \fint_Q \ln \det \brac{V^{-\frac{1}{p}} (x) R_Q ^p(V)} \, dx
+ \sum_{i = 1}^d \fint_Q \ln^+  \abs{R_Q ^p(V)^{-1}  V^{\frac{1}{p}} (x)    \V{f}_i (x)} \, dx  \\
& \leq c + c(d, p) = C'.
\end{align*}
Now we replace $\V{e}$ with $\frac{R_Q ^p(V)^{-1} \V{e}}{\abs{R_Q ^p(V)^{-1}  \V{e}}}$ for an arbitrary vector $\V{e} \in \C^d$ to get that
$$\fint_Q \ln \abs{V^{-\frac{1}{p}} (x)   \V{e} } \, dx \leq C' + \ln \abs{R_Q ^p(V)^{-1}  \V{e}}.$$
After exponentiating both sides, this gives \eqref{Apinftwo}, as required.
\end{proof}

As was shown in \cite{NT96,Vol97}, for any matrix weight, the reverse inequalities to both \eqref{Apinfone} and \eqref{Apinftwo} hold with constant $C = 1$.

\begin{lem}[Reverse inequalities]
\label{OtherMatrixJensen}
Let $V$ be a $p$-nondegenerate matrix weight.
Then for any cube $Q \su \R^n$ and any $\V{e} \in \C^d$, it holds that
\begin{equation*}
\det R_Q ^p (V) \geq   \exp  \fint_Q \ln \det V^\frac{1}{p} (x) \, dx
\end{equation*}
and
\begin{equation*}
\exp\pr{ \fint_Q \ln |V^{-\frac{1}{p}} (x) \V{e} | \, dx} \geq \abs{ (R_Q ^p (V))^{-1}  \V{e} }.
\end{equation*}
\end{lem}

\begin{proof}
To prove the first inequality, let $\set{\V{e}_i}_{i=1}^d$ be an orthonormal basis of eigenvectors for $R_Q ^p(V)$.
Applications of Lemma \ref{MatrixJensen}, Lemma \ref{DetProp}, H\"{o}lder's inequality, and \eqref{reducingDef} show that
\begin{align*}
\exp \pr{\fint_Q \ln \det V^\frac{1}{p} (x) \, dx}
&\leq \det \pr{\fint_Q V^\frac{1}{p}(x) dx}
\leq \prod_{i = 1}^d \abs{\pr{\fint_Q V^\frac{1}{p}(x)  \, dx} \V{e}_i}
\le \prod_{i = 1}^d \fint_Q \abs{V^\frac{1}{p} (x) \V{e}_i} \, dx  \\
&\leq \prod_{i = 1}^d \pr{\fint_Q \abs{V^\frac{1}{p} (x) \V{e}_i}^p \, dx}^\frac{1}{p}
\leq  \prod_{i = 1}^d \abs{R_Q ^p (V) \V{e}_i} = \det R_Q  ^p (V),
\end{align*}
as required.

For the second inequality, observe that for any $\V{e}, \V{f} \in \Cd$
\begin{equation*}
\abs{\innp{\V{e}, \V{f}}} \leq |V^{-\frac{1}{p}}(x) \V{e}| |V^\frac{1}{p} (x)\V{f}|.
\end{equation*}
Taking logarithms and averages then shows that
\begin{equation*}
\ln \abs{\innp{\V{e}, \V{f}}}
\leq \fint_Q \ln |V^{-\frac{1}{p}} (x) \V{e} | \, dx
+ \fint_Q \ln |V^\frac{1}{p} (x) \V{f}| \, dx
\leq \fint_Q \ln |V^{-\frac{1}{p}} (x) \V{e} | \, dx
+ \pr{\fint_Q \ln |V^\frac{1}{p} (x) \V{f}|^p \, dx}^{\frac 1 p},
\end{equation*}
where we have applied H\"older's inequality to the second term on the right.
Thus, Jensen's inequality implies that
\begin{align*}
\abs{\innp{\V{e}, \V{f}}}
&\leq \exp\pr{\fint_Q \ln |V^{-\frac{1}{p}} (x) \V{e} | \, dx } \brac{ \exp\pr{ \fint_Q \ln |V^\frac{1}{p} (x) \V{f} |^p \, dx }}^\frac{1}{p} \\
&\leq \exp\pr{\fint_Q \ln |V^{-\frac{1}{p}} (x) \V{e} | \, dx } \pr{ \fint_Q  |V^\frac{1}{p} (x) \V{f} |^p \, dx }^\frac{1}{p}
\leq  \exp\pr{\fint_Q \ln |V^{-\frac{1}{p}} (x) \V{e} | \, dx } \abs{R_Q ^p(V) \V{f}},
\end{align*}
where we used \eqref{reducingDef} to reach the last line.
Replacing $\V{f}$ with $(R_Q  ^p(V))^{-1} \V{f}$ and using duality shows that
\begin{equation*}
\abs{(R_Q ^p(V))^{-1} \V{e}}
\leq \exp\pr{\fint_Q \ln |V^{-\frac{1}{p}} (x) \V{e} | \, dx},
\end{equation*}
as desired.
\end{proof}

We also need the following elementary Lemma from \cite{NT96}.

\begin{lem}[Determinant to norm lemma]
\label{DetToNormLem}
Let $A$ be a $d \times d$ matrix for which $|\det A| \leq C < \infty$ and $|A\V{e} | \geq |\V{e}|$ for any $\V{e} \in \Cd$.
Then $\|A\| \leq C$.
\end{lem}

\begin{proof}
Let $\lambda_1, \ldots, \lambda_d$ denote the eigenvalues of $A^*A$.
The second condition implies that $\innp{A^*A \V{e}, \V{e}} \geq |\V{e}|^2$ for any $\V{e} \in \Cd$, so it holds that $\disp \min_j |\lambda_j| \geq 1$.
Since the first condition implies that
$$ \prod_j |\lambda_j| = \det A^* A = \abs{\det A}^2 \leq C^2,$$
we must have that
$$\abs{A}^2 = \abs{A ^* A} = \max_j \lambda_j \leq C^2$$
and the conclusion follows.
\end{proof}

If $V \in \ND$, where $\ND$ is the ``nondegenerate" class of matrix weights introduced and discussed in Section \ref{MWeights}, then for any measurable set $E$ with $|E| > 0$, it holds that $\disp \int_E V > 0$.
That is, for any $\V{e} \in \Cd$, we have
\begin{equation*}
0 < \innp{\pr{\int_E V(x) \, dx} \V{e}, \V{e}} = \int_E \innp {V(x) \V{e}, \V{e}} \, dx = \int_E \abs{V^\frac12 (x) \V{e}}^2 \, dx.
\end{equation*}
It follows that $V^p$ is $2p$-nondegenerate.
In particular, for each cube $Q \su \R^n$, there exists a reducing matrix $R_Q ^{2p} (V^p)$.
We now state and prove a determinant characterization of the matrix class $\Bp$.

\begin{lem}[$\Bp$ determinant characterization]
\label{BpDef}
If $V \in \ND$, then the following are equivalent:
\begin{itemize}
\item[(i)] There exists a constant $C > 0$ so that for every cube $Q \su \R^n$,
\begin{equation}
\label{BpDefOne}
\det \brac{R_Q^{2p} (V^p)} \le C \det \left(\fint_Q V\pr{x} dx \right)^\frac12.
\end{equation}
\item[(ii)] There exists a constant $C > 0$ so that for every cube $Q \su \R^n$ and every $\vec e \in \C^d$,
\begin{equation}
\label{BpDefTwo2}
\pr{\fint_Q \innp{V\pr{x} \vec e, \vec e}^p dx}^{1/p} \le C \innp{\pr{ \fint_Q V\pr{x} dx } \vec e, \vec e} .
\end{equation}
\end{itemize}
\end{lem}

\begin{rem}
Notice that the condition described by \eqref{BpDefTwo2} is our classicial definition of $V \in \Bp$ as presented in Section \ref{MWeights}.
Therefore, this proposition gives an alternative definition in terms of determinants and reducing matrices.
\end{rem}

\begin{proof}
We first prove that \eqref{BpDefOne} implies \eqref{BpDefTwo2}.
Observe that for any $\V{e} \in \Cd$, we have by H\"older's inequality and the property of the reducing matrix in \eqref{reducingDef} that
\begin{equation*}
\abs{\pr{\fint_Q V(x) \, dx }^\frac12 \V{e}}^2
= \fint_Q |V^\frac12(x) \V{e}|^2 \, dx
\leq \pr{\fint_Q \abs{\brac{V^p(x)}^\frac{1}{2p} \V{e}}^{2p}  \, dx }^\frac{1}{p}
\leq \abs{\brac{R_Q^{2p} (V^p)}\V{e} }^2.
\end{equation*}
Thus, for any $\V{e} \in \C^d$,
\begin{equation*}
\abs{\brac{R_Q^{2p} (V^p)}  \pr{\fint_Q V(x) \, dx }^{-\frac12} \V{e} } \geq |\V{e}|,
\end{equation*}
while the assumption of \eqref{BpDefOne} implies that
\begin{equation*}
\det \set{\brac{R_Q^{2p} (V^p)}  \pr{\fint_Q V(x) \, dx }^{-\frac12}}  \leq C.
\end{equation*}
An application of Lemma \ref{DetToNormLem} shows that
$$\norm{\brac{R_Q^{2p} (V^p)}  \pr{\fint_Q V(x) \, dx }^{-\frac12} } \leq C.$$
Therefore, it follows from \eqref{reducingpCons} that
\begin{equation*}
\pr{\fint_Q \innp{V(x) \V{e}, \V{e}}^p \, dx} ^\frac{1}{p}
\leq \abs{R_Q ^{2p} (V^p) \V{e}}^2
\leq C^2 \abs{\pr{\fint_Q V(x) dx} ^\frac12 \V{e}}^2
= C^2 \innp{\pr{\fint_Q V(x) \, dx} \V{e}, \V{e}},
\end{equation*}
showing that \eqref{BpDefTwo2} holds.

For the converse, assume that \eqref{BpDefTwo2} holds.
As demonstrated above, this assumption is equivalent to
\begin{align*}
& \abs{R_Q ^{2p} (V^p) \V{e}} ^2 \leq C \abs{\pr{\fint_Q V(x) dx}^\frac12 \V{e}}^2 \qquad \forall \V{e} \in \Cd \\
\Leftrightarrow & \abs{R_Q ^{2p} (V^p) \pr{\fint_Q V(x) dx} ^{-\frac12} \V{e}} ^2 \leq C \abs{ \V{e}}^2 \qquad \forall \V{e} \in \Cd \\
\Leftrightarrow & \norm{\pr{\fint_Q V(x) dx} ^{-\frac12}  \brac{R_Q^{2p} (V^p)} ^2 \pr{\fint_Q V(x) dx} ^{-\frac12} } = \norm{ R_Q ^{2p} (V^p) \pr{\fint_Q V(x) dx} ^{-\frac12}}^2 \leq C^2.
\end{align*}
It follows that
\begin{equation*}
\det \brac{ \pr{\fint_Q V(x) dx} ^{-\frac12}  \brac{R_Q^{2p} (V^p)} ^2 \pr{\fint_Q V(x) dx} ^{-\frac12}} \leq C^{2d}
\end{equation*}
which implies \eqref{BpDefOne}, as required.
\end{proof}

Recall the classical assertion that for $p > 1$, a scalar weight $v \in \sBp$ iff $v^p  \in \sAi$.
The following proposition connects the classes $\Bp$ with $\Api$ and provides a matrix analogue to the aforementioned scalar result.
See also \cite[Corollary 3.8]{Ros16} for a related result.

\begin{prop}
\label{AinfBpLem}
If $V \in \ND$, then $V^p \in \mathcal{A}_{2p, \infty}$ iff $V \in \Atwi \cap \Bp$.
\end{prop}

\begin{proof}
Assume that $V \in \ND \cap \Atwi \cap \Bp$.
Since $V \in \ND \cap \Bp$, then the conclusions from Lemma \ref{BpDef} hold.
As $\disp \pr{\fint_Q V(x) dx}^\frac12$ is a $p = 2$ reducing matrix, then \eqref{Apinfone} holds with $\disp R_Q^2(V) = \pr{\fint_Q V(x) dx}^\frac12$ since $V \in \Atwi$.
Combining \eqref{BpDefOne} and \eqref{Apinfone} shows that for any cube $Q \su \R^n$, we have
\begin{equation*}
\det \brac{R_Q ^{2p} (V^p)}
\le C \det \pr{\fint_Q V(x) dx} ^\frac12
\le C \exp \pr{\fint_Q \ln \det V^\frac12 (x) \, dx}.
\end{equation*}
Comparing with \eqref{Apinfone}, this shows that $V^p \in \mathcal{A}_{2p, \infty}$.

Conversely, assume that $V^p \in \mathcal{A}_{2p, \infty}$ and $V \in \ND$.
By the definition of $\mathcal{A}_{2p, \infty}$ as in \eqref{Apinfone}, then by an application Lemma \ref{MatrixJensen}, we see that for any cube $Q \su \R^n$,
\begin{align*}
\det \brac{R_Q ^{2p} (V^p)}
&\le C \exp \pr{\fint_Q \ln \det   V^\frac12(x) \, dx}
= C \brac{\exp \pr{\fint_Q \ln \det   V(x) \, dx}}^\frac12
\leq \det \left( \fint_Q V (x) \, dx\right)^\frac12.
\end{align*}
Since $V \in \ND$, Lemma \ref{BpDef} implies that $V \in \Bp$.

For any $\V{e} \in \Cd$, we have by H\"older's inequality and \eqref{reducingDef} that
$$\abs{\pr{\fint_Q V(x) dx}^\frac12 \V{e}}
= \pr{\fint_Q \abs{V^\frac12 (x) \V{e}} ^2 \, dx}^\frac12
\leq \pr{\fint_Q \abs{V^\frac12 (x) \V{e}} ^{2p} \, dx}^\frac{1}{2p}
\leq \abs{R_Q ^{2p} \pr{V^p} \V{e}} $$
so that
\begin{align*}
\abs{\pr{R_Q ^{2p} (V^p) }^{-1} \pr{\fint_Q V(x) dx} \pr{R_Q ^{2p} (V^p)} ^{-1}}
&= \abs{\pr{R_Q ^{2p} (V^p)} ^{-1} \pr{\fint_Q V(x) dx}^\frac12}^2 \\
&=  \abs{\pr{\fint_Q V(x) dx}^\frac12 \pr{R_Q ^{2p} (V^p)}^{-1} }^2
\leq 1.
\end{align*}
Therefore, for any cube $Q \su \R^n$,
$$\det \pr{\fint_Q V(x) dx}^\frac12
\le \det  R_Q ^{2p} \pr{V^p}
\le C \exp \pr{\fint_Q \ln \det V^\frac12 (x) \, dx },$$
where the second inequality uses \eqref{Apinfone} since $V^p \in \mathcal{A}_{2p, \infty}$.
In particular, since $\disp \pr{\fint_Q V(x) dx}^\frac12 $ is a reducing matrix for $p = 2$, it follows from \eqref{Apinfone} that $V \in \Atwi$.
\end{proof}

Now that we have discussed the $\Api$ classes of matrices, we seek the connections between the classes $\Api$ and $\Ai$.
We first recall the definition of $\Ai$ from \cite{Dall15}.

\begin{defn}[$\Ai$]
We say that $V \in \Ai$ if for any $\epsilon > 0$, there exists $\delta > 0$ so that for any cube $Q \su \R^n$, it holds that
\begin{equation}
\label{AinfIneq}
\abs{\set{x \in Q: V\pr{x} \geq \delta \fint_Q V(y) dy}} \geq (1-\epsilon) |Q|.
\end{equation}
\end{defn}

For the proofs below, we will use an alternative version of this definition which appears in \cite{Dall15}, described by the next lemma.

\begin{lem}[$\Ai$ characterization]
\label{AiChar}
$V \in \Ai$ iff for any $\epsilon > 0$ there exists $\gamma > 0$ such that for any cube $Q \su \R^n$, it holds that
\begin{equation}
\label{AinfIneq2}
\abs{\set{x \in Q : \norm{\pr{\fint_Q V\pr{y} \, dy}^\frac12  V^{-\frac12} (x)} >  \gamma}} < \epsilon |Q|.
\end{equation}
\end{lem}

\begin{proof}
Examining the $\Ai$ condition described by \eqref{AinfIneq}, we see that
\begin{align*}
V(x) \geq \delta \fint_Q V(y) dy
& \Leftrightarrow \innp{V(x) \V{e}, \V{e}} \geq \delta  \innp{\pr{\fint_Q V(y) \, dy}\V{e}, \V{e}}, \quad \forall \V{e} \in \Cd \\
& \Leftrightarrow  \delta^{-1} |\V{e}|^2 \geq   \innp{\pr{\fint_Q V(y) \, dy} V^{-\frac12} (x)\V{e}, V^{-\frac12} (x)\V{e}}, \quad \forall \V{e} \in \Cd \\
& \Leftrightarrow  \delta^{-1} |\V{e}|^2 \geq   \abs{\pr{\fint_Q V(y) \, dy}^\frac12  V^{-\frac12} (x) \V{e} }^2, \quad \forall \V{e} \in \Cd \\
& \Leftrightarrow  \delta^{-1}\geq   \norm{\pr{\fint_Q V(y) \, dy}^\frac12  V^{-\frac12} (x)}^2.
\end{align*}
The conclusion follows from setting $\ga = \de^{-1}$.
\end{proof}

Our next pair of results examine the relationship between $\Atwi$, $\RBM$, and $\Ai$.
We first prove the following more general result which implies the first inclusion.

\begin{prop}
\label{AinfLem}
If $p \geq 1$ and  $V^p \in \mathcal{A}_{2p, \infty}$, then $V \in \Ai$.
In particular, $\Atwi \subseteq \Ai$.
\end{prop}

\begin{proof}
Let $p \geq 1$ and assume that $V^p \in \mathcal{A}_{2p, \infty}$.
Note that that $V^p$ is $2p$-nondegenerate by definition.
Let $Q \su \R^n$ and let $\set{\V{e}_i}_{i = 1}^d$ be the standard orthonormal basis of $\Cd$.
For any $\lambda > 0$, let $\MC{J}_i (Q)$ denote the collection of maximal dyadic subcubes $J$ of $Q$ satisfying
\begin{equation}
\label{subcubeDefEqn}
\abs{\brac{R_J ^{2p} (V^p)}^{-1} \brac{R_Q ^{2p} (V^p)} \V{e}_i} >  e^{C\lambda},
\end{equation}
where $C > 0$ is independent of $Q$ and $J$ (and will be determined below).
Then it is enough to prove the following claim:
\begin{clm}
\label{SubcubeClaim}
If $C> 0$ is sufficiently large (and independent of $Q$ and $J$), then
 \begin{equation}
 \label{DecayingStopEst}
 \sum_{i = 1}^d \sum_{J \in \MC{J}_i (Q)} \abs{J} < \frac{1}{\lambda} |Q| .
 \end{equation}
\end{clm}
Before proving the claim, we show how it leads to the conclusion that $V \in \Ai$.
If $\disp x \in Q \backslash \pr{\bigcup_{i = 1}^d \bigcup_{J \in \MC{J}_i (Q)} J}$, then for any dyadic subcube $L$ of $Q$  containing $x$, we must have that
\begin{equation*}
\norm{R_Q ^{2p} (V^p)\brac{R_L ^{2p} (V^p)} ^{-1} }
= \norm{\brac{R_L ^{2p} (V^p)} ^{-1} R_Q ^{2p} (V^p)}
\leq  e^{C \lambda}.
\end{equation*}
It follows that for any $\V{e} \in \Cd$,
\begin{equation}
\label{LcubeObs}
\abs{R_Q ^{2p} (V^p) \V{e}}
\leq e^{C \lambda} \abs{R_L ^{2p} (V^p) \V{e}}
\leq \sqrt{d} e^{C \lambda} \pr{\fint_L \abs{V^\frac12 (y) \V{e} }^{2p} \, dy}^\frac{1}{2p},
\end{equation}
where we have applied \eqref{reducingpCons} in the last inequality.
Applications of H\"older's inequality, \eqref{reducingpCons}, then \eqref{LcubeObs} combined with the Lebesgue differentiation theorem show that for any $\V{e} \in \Cd$,
\begin{equation*}
\abs{\pr{\fint_Q V(y) \, dy }^\frac12 \V{e}}
= \pr{\fint_Q \abs{V^\frac12(y) \V{e}}^2 \, dy}^\frac12 \leq \pr{\fint_Q \abs{V^\frac12(y) \V{e}}^{2p} \, dy}^\frac{1}{2p}
\leq \abs{R_Q ^{2p} (V^p) \V{e}}
\leq \sqrt{d}  e^{C \lambda} \abs{V^\frac12 (x) \V{e}}.
\end{equation*}
However, this implies that, modulo a set of measure zero,
\begin{equation*}
Q \backslash \pr{\bigcup_{i = 1}^d \bigcup_{J \in \MC{J}_i (Q)} J}
\subseteq \set{x \in Q : \norm{\pr{\fint_Q V\pr{y} \, dy}^\frac12  V^{-\frac12} (x)} \leq \sqrt{d} e^{C\lambda}}.
\end{equation*}
In particular, an application of  \eqref{DecayingStopEst} shows that
\begin{equation*}
\abs{\set{x \in Q : \norm{\pr{\fint_Q V\pr{y} \, dy}^\frac12  V^{-\frac12} (x)} >  \sqrt{d} e^{C\lambda }}}
\leq \sum_{i = 1}^d \sum_{J \in \MC{J}_i (Q)} \abs{J} < \frac{1}{\lambda} |Q|.
\end{equation*}
Since \eqref{AinfIneq2} holds with $\lambda = \frac{1}{\epsilon}$ and $\gamma = \sqrt{d} e^{C \lambda}$, then it follows from Lemma \ref{AiChar} that $V \in \Ai$.

To complete the proof, we now establish Claim \ref{SubcubeClaim}.
Note that this claim was implicitly proved in the proof of Lemma $3.1$ in \cite{Vol97}, but we include the details for the sake of completion.
Let $J \in \MC{J}_i (Q)$.
Taking logarithms in \eqref{subcubeDefEqn}, then applying Lemma \ref{OtherMatrixJensen} shows that
 \begin{equation*}
 C \lambda
 < \ln \abs{\brac{R_J ^{2p} (V^p)}^{-1} R_Q ^{2p} (V^p) \V{e}_i}
 \leq \fint_J \ln \abs{V^{-\frac12} (x) R_Q ^{2p} (V^p) \V{e}_i} \, dx
 \leq \fint_J \ln^+ \abs{V^{-\frac12} (x) R_Q ^{2p} (V^p) \V{e}_i} \, dx,
 \end{equation*}
 where the last inequality ensures that the integrand is nonnegative.
 Since each collection $\MC{J}_i (Q)$ is disjoint by maximality, then we may sum to get
 \begin{align*}
 \sum_{i = 1}^d \sum_{J \in \MC{J}_i (Q)} \abs{J}
 &\leq  \frac{1}{C\lambda} \sum_{i = 1}^d \sum_{J \in \MC{J}_i (Q)} \int_J \ln^+ \abs{V^{-\frac12} (x) R_Q ^{2p} (V^p) \V{e}_i} \, dx
\leq \frac{1}{C\lambda} \sum_{i = 1}^d \int_Q \ln^+ \abs{V^{-\frac12} (x) R_Q ^{2p} (V^p) \V{e}_i} \, dx \\
&= \frac{1}{C\lambda} \sum_{i = 1}^d \int_Q \ln \abs{V^{-\frac12} (x) R_Q ^{2p} (V^p) \V{e}_i} \, dx
+ \frac{1}{C\lambda} \sum_{i = 1}^d \int_Q \ln^+ \abs{V^{-\frac12} (x) R_Q ^{2p} (V^p) \V{e}_i}^{-1} \, dx,
\end{align*}
since $\ln ^+ x = \ln x + \ln ^+ x^{-1}$.
As $V^p \in \mathcal{A}_{2p, \infty}$, then it follows from Lemma \ref{ApiProperty} that for some $C' > 0$ independent of $Q$,
\begin{equation*}
\int_Q \ln \abs{V^{-\frac12} (x) R_Q ^{2p} (V^p) \V{e}_i} \, dx
\leq C' |Q|.
\end{equation*}
An application of \eqref{InvIneq} with $A = V^{-\frac12} (x) R_{Q}^{2p} (V^p)$ and $\V{c} = \V{e}_i$ gives
\begin{align*}
\int_Q \ln^+ \abs{V^{-\frac12} (x) R_Q ^{2p} (V^p) \V{e}_i}^{-1} \, dx
&\leq \int_Q \ln^+ \abs{\brac{R_Q ^{2p} (V^p)}^{-1} V^{\frac12} (x)  \V{e}_i} \, dx
\leq |Q| \fint_Q  \abs{\brac{R_Q ^{2p} (V^p)}^{-1} V^{\frac12} (x)  \V{e}_i} \, dx \\
&\leq |Q| \pr{\fint_Q  \abs{\brac{R_Q ^{2p} (V^p)}^{-1} V^{\frac12} (x)  \V{e}_i}^{2p} \, dx}^\frac{1}{2p}
\le C(d, p) |Q|,
\end{align*}
where we have applied H\"older's and the same argument used to prove \eqref{AinfEstTwo} for the last two inequalities, respectively.
Combining the previous three observations shows that
 \begin{align*}
 \sum_{i = 1}^d \sum_{J \in \MC{J}_i (Q)} \abs{J}
\leq  \frac{d\pr{C' + C(d, p)}}{C\lambda}  |Q|.
 \end{align*}
In particular, Claim \ref{SubcubeClaim} holds whenever $C > d\brac{C' + C(d, p)}$.
\end{proof}

Next we show that the inclusion may be reversed.
In fact, for the final result of this appendix, we prove three equivalent conditions for nondegenerate matrices.
But first, we recall the following definition of the reverse Brunn-Minkowski class of matrices.

\begin{defn}[$\RBM$]
We say that a matrix weight $V$ belongs to the {\bf reverse Brunn-Minkowski class}, $V \in \RBM$, if there exists a constant $B_V > 0$ so that for any cube $Q \su \Rn$, it holds that
$$\pr{\det \fint_Q V(x) dx}^\frac{1}{d} \leq B_V \fint_Q \brac{\det V(x)}^\frac{1}{d} dx.$$
\end{defn}

\begin{prop}
\label{AtwoinfAinfProp}
If $V \in \ND$, then the following are equivalent:
\begin{itemize}
\item[a)] $V \in \Atwi$,
\item[b)] $V \in  \Ai$,
\item[c)] $V \in \RBM$ and $(\det V)^\frac{1}{d} \in \T{A}_\infty$.
\end{itemize}
\end{prop}

\begin{proof}
That $a) \Rightarrow b)$ was proved in Proposition \ref{AinfLem}.

We now prove $b) \Rightarrow c)$.
Let $V \in \ND \cap \Ai$.
For any $\eps > 0$, let $\de = \de(\eps)$ be as given in the definition of $\Ai$.
Let $\disp \set{\V{e}_k(x)}_{k = 1}^d$ be an orthonormal basis of eigenvectors for $V(x)$ and let $S\subseteq Q$ be the set on the lefthand side of \eqref{AinfIneq} so that $|S| \geq (1 - \epsilon) |Q|$. Then for any $x \in S$, we have that
\begin{equation}
\label{DetAinfone}
\begin{aligned}
\det V(x)
&= \prod_{k = 1}^d \innp{V(x) \V{e}_k(x), \V{e}_k(x)}
\geq \delta \prod_{k = 1}^d   \innp{\pr{\fint_Q V(y) dy} \V{e}_k(x), \V{e}_k(x)}
\geq \delta \det \pr{\fint_Q V(y) dy} \\
&\geq \delta \pr{ \fint_Q \brac{\det V(y)}^\frac{1}{d}  \, dy}^d,
\end{aligned}
\end{equation}
where here we have used both Lemma \ref{DetProp} and \eqref{DetConvexIneq}.
In other words,
\begin{equation*}
\abs{\set{ x \in Q : \brac{\det V (x) }^\frac{1}{d}  \ge \delta^\frac{1}{d}  \fint_Q  \brac{\det V(y)}^\frac{1}{d} dy }}
\geq |S|
\geq (1 - \epsilon) |Q| ,
\end{equation*}
which shows that $\pr{\det V}^\frac{1}{d} \in \T{A}_\infty$.

Next, to show that $V \in \RBM,$ we use the first line of \eqref{DetAinfone} with $\epsilon = \frac12$  and $\delta = \delta\pr{\frac12}$ to get
\begin{align*}
\fint_Q \brac{\det V(y)}^\frac{1}{d}  \, dy
&\ge \frac 1 {|Q|} \int_S \brac{\det V(x)} ^\frac{1}{d} \, dx
\ge \frac 1 {|Q|} \int_S \delta^{\frac 1 d} \brac{\det \pr{\fint_Q V(y) dy}}^{\frac 1 d}
\ge \frac{\delta^{\frac 1 d}}{2} \brac{\det \pr{\fint_Q V(y) dy}}^{\frac 1 d},
\end{align*}
as required.

Finally we prove that $c) \Rightarrow a)$.
If $(\det V) ^\frac{1}{d} \in \T{A}_\infty$, then by the classical reverse Jensen characterization of scalar $\sAi$ weights (again see \cite[Theorem 7.3.3]{Gra14}),  there exists $C > 0$ so that for any $Q \su \R^n$, we have
\begin{equation*}
\fint_Q \brac{\det V(x)} ^\frac{1}{d} dx
\le C \exp \pr{\fint_Q \ln \brac{\det V(x)}^\frac{1}{d} dx }
= C \brac{\exp \pr{\fint_Q \ln \det V(x) dx}}^\frac{1}{d}.
\end{equation*}
However, combining this bound with $V \in \RBM$ gives us
$$\pr{\det \fint_Q V(x) dx}^\frac{1}{d} \leq B_V \fint_Q \brac{\det V(x)}^\frac{1}{d} dx \leq B_V C \brac{\exp \pr{\fint_Q \ln \det V(x) dx}}^\frac{1}{d},$$
which by \eqref{Apinfone2} shows that $V \in \Atwi$ and completes the proof.
\end{proof}

\section{Technical Proofs}
\label{TechProofs}

This final appendix provides the technical proofs that were skipped in the body of the paper.
We first prove Proposition \ref{la1Prop}, which states that if $V \in \Bp \cap \Ai$, then $\la_1$, the smallest eigenvalue of $V$, belongs to $\sBp$.

\begin{proof}[Proof of Proposition \ref{la1Prop}]
Let $\eps > 0$.
Since $V \in \Ai$, we may choose $\de > 0$ so that \eqref{Ainf} holds.
For some $Q \su \R^n$, define
$$S = \set{x \in Q: V\pr{x} \geq \delta \fint_Q V\pr{y} dy}.$$
Let $\la_1 \le \la_2 \le \ldots \le \la_d$ denote the eigenvalue functions of $V$ with associated orthonormal eigenvectors $\set{\V{v}_i}_{i=1}^d$.
That is, for each $i = 1, \ldots, d$, $V \V{v}_i = \la_i \V{v}_i$.
Observe that for any $x \in S$,
\begin{align*}
\la_1^p\pr{x}
&= \innp{V\pr{x} \V{v}_1\pr{x}, \V{v}_1\pr{x}}^p
\ge \brac{ \innp{ \pr{ \de \fint_Q V\pr{y} dy} \V{v}_1\pr{x}, \V{v}_1\pr{x}}}^p \\
&= \de^p \brac{ \fint_Q \innp{ V\pr{y} \V{v}_1\pr{x}, \V{v}_1\pr{x} dy}}^p
\ge \pr{\frac \de {C_V}}^p \fint_Q \innp{ V\pr{y} \V{v}_1\pr{x}, \V{v}_1\pr{x}}^p dy,
\end{align*}
where the last inequality follows from the assumption that $V \in \Bp$.
Now by diagonalization,
\begin{align*}
\innp{V\pr{y} \V{v}_1\pr{x}, \V{v}_1\pr{x}}
&= \sum_{j=1}^d \la_j\pr{y} \abs{\innp{\V{v}_1\pr{x}, \V{v}_j\pr{y}}}^2
\ge \la_1\pr{y} \sum_{j=1}^d\abs{\innp{\V{v}_1\pr{x}, \V{v}_j\pr{y}}}^2
= \la_1\pr{y} \abs{\V{v}_1\pr{x}}^2 \\
&= \la_1\pr{y}.
\end{align*}
Combining these inequalities shows that $\disp \la_1^p\pr{x} \ge \pr{\frac \de {C_V}}^p \fint_Q \la_1^p\pr{y} dy$.
If we define $\de' = \pr{\frac \de {C_V}}^p$ and $\disp S' = \set{x \in Q: \la_1^p\pr{x} \geq \de' \fint_Q \la_1^p\pr{y} dy}$, then we see that $S \su S'$.
In particular, $\abs{S'} \ge \abs{S} \ge \pr{1 - \eps} \abs{Q}$.
Therefore, $\la_1^p \in \sAi$.
It follows from a scalar result that $\la_1 \in \sBp$, as required.
\end{proof}

Now we prove Proposition \ref{BickelToProveProp} which states that a matrix weight of the form $A = \pr{a_{ij} \abs{x}^{\ga_{ij}}}_{i, j = 1}^d$ belongs to $\Atwi \cap \Bp$ if $A = \pr{a_{ij}}_{i, j = 1}^d$ is a Hermitian, positive definite matrix and $\ga_{ij} = \frac 1 2 \pr{\ga_i + \ga_j}$ for some $\V{\ga} \in \R^d$ with $\gamma_{i} > - \frac{n}{p}$.

\begin{proof}[Proof of Proposition \ref{BickelToProveProp}]
First observe that since $V$ is positive definite, then $V \in \ND$.
By Proposition \ref{AinfBpLem}, $V \in \Atwi \cap \Bp$ iff $V^p \in \mathcal{A}_{2p, \infty}$.
Therefore, we will show that $V^p \in \mathcal{A}_{2p, \infty}$.
By Lemma \ref{ApiProperty}, $V^p \in \mathcal{A}_{2p, \infty}$ iff there exists a constant $C > 0$ so that for every $\V{e} \in \Cd$ and every cube $Q \su \R^n$, it holds that
\begin{equation*}
\exp\pr{\fint_Q \ln |V^{-\frac{1}{2}} (x) \V{e} | \, dx}  \le C \abs{ \brac{R_Q ^{2p}(V^p)}^{-1}  \V{e} },
\end{equation*}
where $R_Q^{2p}$ is a reducing matrix of $V^p$; see \eqref{reducingDef}.
This condition is equivalent to the existence of $C > 0$ so that for every unit vector $\V{e} \in \Cd$ and every cube $Q \su \R^n$, it holds that
\begin{equation}
\label{BickelToProve}
\exp\pr{\fint_Q \ln |V^{-\frac{1}{2}} (x) R_Q ^{2p}(V^p) \V{e} | \, dx}  \le C .
\end{equation}
Therefore, to prove this proposition, we will show that there exists a constant $C > 0$ so that \eqref{BickelToProve} holds for every unit vector $\V{e} \in \Cd$ and every cube $Q \su \R^n$.

First, using the facts that $\ln x \leq \abs{\ln x} = \abs{ \ln ^+ x - \ln ^+ x^{-1}} \leq \ln ^+ x + \ln ^+ x^{-1}$ and $\abs{A\V{e}}^{-1} \leq \abs{A^{-1} \V{e}}$ for any invertible Hermitian matrix $A$ and any unit vector $\V{e}$, we get
\begin{equation}
\label{BickelEst0}
\begin{aligned}
&\brac{\exp\pr{\fint_Q \ln |V^{-\frac{1}{2}} (x) R_Q ^{2p}(V^p) \V{e} | \, dx} }^2
= \exp\brac{\fint_Q \ln \pr{|V^{-\frac{1}{2}} (x) R_Q ^{2p}(V^p) \V{e} |^2} \, dx} \\
\le& \exp\brac{\fint_Q \set{\ln^+ \pr{|V^{-\frac{1}{2}} (x) R_Q ^{2p}(V^p) \V{e} |^2} + \ln^+ \pr{|V^{-\frac{1}{2}} (x) R_Q ^{2p}(V^p) \V{e} |^{-2}}} \, dx} \\
=& \exp\brac{\fint_Q \ln^+ \pr{|V^{-\frac{1}{2}} (x) R_Q ^{2p}(V^p) \V{e} |^2} \, dx}
 \exp \brac{\fint_Q \ln^+ \pr{|V^{-\frac{1}{2}} (x) R_Q ^{2p}(V^p) \V{e} |^{-2}} \, dx} \\
\le& \exp\brac{\fint_Q \ln^+ \pr{|V^{-\frac{1}{2}} (x) R_Q ^{2p}(V^p) \V{e} |^2} \, dx}
 \exp\brac{\fint_Q \ln ^+ \pr{\abs{\brac{R_Q ^{2p}(V^p)}^{-1} V^{\frac{1}{2}} (x) \V{e} }^2} \, dx} \\
=:& E_1 \times E_2.
\end{aligned}
\end{equation}

We estimate $E_2$.
If $\{e_j\}_{j = 1}^d$ is any orthonormal basis of $\Cd$, then since $\ln ^+x \leq x$ for $x > 0$, we get
\begin{equation}
\label{BickelEst1}
\begin{aligned}
E_2
&= \exp\brac{\fint_Q \ln ^+ \pr{\abs{\brac{R_Q ^{2p}(V^p)}^{-1} V^{\frac{1}{2}} (x) \V{e} }^2} \, dx}
\leq \exp\pr{\fint_Q \abs{V^{\frac{1}{2}} (x) \brac{R_Q ^{2p}(V^p)}^{-1} }^{2} \, dx} \\
&\leq \prod_{j = 1}^d \exp\pr{\fint_Q \abs{V^{\frac{1}{2}} (x) \brac{R_Q ^{2p}(V^p)}^{-1} \V{e}_j }^{2} \, dx} \\
&\leq \prod_{j = 1}^d\brac{\exp\pr{\fint_Q   \abs{V^{\frac{1}{2}} (x) \brac{R_Q ^{2p}(V^p)}^{-1} \V{e}_j }^{2p} \, dx}}^\frac{1}{p}  \\
&\leq \prod_{j = 1}^d\brac{\exp \pr{\abs{R_Q^{2p} (V^p) \brac{R_Q ^{2p}(V^p)}^{-1} \V{e}_j }^{2p}} }^\frac{1}{p}
\leq \exp \pr{\frac{d}{p}},
\end{aligned}
\end{equation}
where we have applied H\"older's inequality followed by the reducing matrix property from \eqref{reducingDef}.

Next, we estimate $E_1$.
We start with some preliminary estimates.
For any unit vector $\V{e} \in \C^d$ and any $x \in Q$, observe that by another application of \eqref{reducingDef},
\begin{equation}
\label{BickelEst2}
\begin{aligned}
|V^{-\frac{1}{2}} (x) R_Q ^{2p}(V^p) \V{e} |^2
&\leq  |R_Q ^{2p}(V^p)  V^{-\frac{1}{2}} (x)  |^2
\leq \sum_{j = 1}^d |R_Q ^{2p}(V^p)  V^{-\frac{1}{2}} (x) \V{e}_j  |^2 \\
&\leq d \sum_{j = 1}^d \pr{\fint_Q \abs{V^\frac12 (y) V^{-\frac{1}{2}} (x) \V{e}_j }^{2p} dy}^{\frac 1{p}}
\lesssim_{(d)}  \pr{\fint_Q \abs{V^\frac12 (y) V^{-\frac{1}{2}} (x)  }^{2p} dy}^{\frac 1{p}} \\
&\simeq_{(d)} \pr{\fint_Q \abs{\tr \pr{V (y) V^{-1} (x)  }}^{p} \, dy}^{\frac 1{p}},
\end{aligned}
\end{equation}
where in the final line we have used the fact that whenever $B$ and $C$ are Hermitian, positive semidefinite matrices, it holds that  $\disp \abs{B^\frac12 C^\frac 12}^2 = \abs{B ^\frac12 C B^\frac12}
\simeq_{(d)} \tr \pr{B^\frac12 C B^\frac12}  = \tr \pr{BC}$.
Using the explicit presentation of $V$ and $V^{-1}$ from \eqref{mpower} and \eqref{mpowerIn}, respectively, we see that
\begin{equation}
\label{BickelEst3}
\begin{aligned}
\pr{\fint_Q \abs{ \tr \brac{V (y) V^{-1} (x)  }}^{p} \,  dy}^{\frac 1{p}}
&= \pr{\fint_Q \abs{ \sum_{i, j=1}^d a_{ij} a^{ji} |y|^{\gamma_{ij}} |x|^{-\gamma_{ji}}  }^{p} \, dy}^{\frac 1{p}}
\\
& \lesssim_{(A)} \sum_{i, j=1}^d |x|^{-\gamma_{ij}} \pr{ \fint_Q \abs{f_{ij}(y)}^{p} dy}^\frac{1}{p}  ,
\end{aligned}
\end{equation}
where we have introduced the notation $f_{ij}(y) = \abs{y}^{\ga_{ij}}$.
Since $p \gamma_{ij} = \frac{p}{2} (\gamma_{i} + \gamma_{j}) > - n$, then $\disp f_{ij}^{p} \in \T{A}_{\infty} = \bigcup_{q \geq 1} \T{A}_q$  (see \cite[p. 506]{Gra14},) which implies that $f_{ij} \in \T{B}_p$.
In particular, it holds that
$$\fint_Q \abs{f_{ij}(y)}^{p} dy \lesssim_{(n, p, \ga_{ij})} \pr{\fint_Q {f_{ij}(y)} dy}^p.$$
Thus, combining this final observation with \eqref{BickelEst2} and \eqref{BickelEst3} shows that there exists $C = C(d, n, p, A, \V{\ga})$ so that
\begin{align}
\label{E1EstTool}
|V^{-\frac{1}{2}} (x) R_Q ^{2p}(V^p) \V{e} |^2
&\le C \sum_{i, j=1}^d |x|^{-\gamma_{ij}} \pr{\fint_Q f_{ij}(y) dy}.
\end{align}
To estimate $E_1$, we use that for any $C, x_1, x_2, \ldots, x_d > 0$, it holds that $\disp \ln^+\pr{\sum_{i =1}^d x_i} \leq d + \sum_{i = 1}^d \ln ^+x_i$ and $\ln^+(C x) \le \ln^+ C + \ln^+ x$.
(The proofs of these results follow from induction and case analysis.)
Therefore, from \eqref{E1EstTool} we see that
\begin{equation*}
\begin{aligned}
E_1
&= \exp\brac{\fint_Q \ln^+ \pr{|V^{-\frac{1}{2}} (x) R_Q ^{2p}(V^p) \V{e} |^2} \, dx}
\le \exp\set{\fint_Q \ln^+ \brac{C \sum_{i, j=1}^d |x|^{-\gamma_{ij}}  \pr{\fint_Q {f_{ij}(y)} dy}} \, dx} \\
&\le \exp\brac{\fint_Q \set{d\pr{1 + \ln^+ C} + \sum_{i, j=1}^d \ln^+ \brac{ |x|^{-\gamma_{ij}} \pr{\fint_Q {f_{ij}(y)} dy}}} \, dx} \\
&\lesssim_{(d, n, p, A, \V{\ga})} \prod_{i, j=1}^d \exp\set{\fint_Q \ln^+ \brac{|x|^{-\gamma_{ij}} \pr{\fint_Q {f_{ij}(y)} dy}}  \, dx} \\
&\le \prod_{i, j =1}^d \exp\set{\fint_Q \ln \brac{|x|^{-\gamma_{ij}} \pr{\fint_Q {f_{ij}(y)} dy}}  \, dx} \exp\set{\fint_Q \ln^+ \brac{|x|^{\gamma_{ij}} \pr{\fint_Q {f_{ij}(y)} dy}^{-1}}  \, dx},
\end{aligned}
\end{equation*}
where in the last line we used that $\ln ^+ x = \ln x + \ln^+ x^{-1}$.
Since $\ln^+ x \le x$, then
\begin{align*}
\fint_Q \ln^+ \brac{|x|^{\gamma_{ij}} \pr{\fint_Q {f_{ij}(y)} dy}^{-1}}  \, dx
&\leq \fint_Q |x|^{\gamma_{ij}} \pr{\fint_Q {f_{ij}(y)} dy}^{-1}  \, dx
= \fint_Q {f_{ij}(x)} dx  \pr{\fint_Q {f_{ij}(y)} dy}^{-1}
= 1.
\end{align*}
On the other hand, since $\gamma_{ij} = \frac 1 2 \pr{\ga_i + \ga_j} > -\frac{n}{p} \ge -n$, then $f_{ij} \in \T{A}_\infty$.
An application of the reverse Jensen inequality (see \cite[p. 525]{Gra14}) shows that there exists $C(n, p, \ga_{ij}) > 0$ so that for any $Q \su \R^n$, it holds that
\begin{align*}
\exp\set{\fint_Q \ln \brac{|x|^{-\gamma_{ij}} \pr{\fint_Q {f_{ij}(y)} dy}}  \, dx}
&= \pr{\fint_Q {f_{ij}(y)} dy} \exp\brac{\fint_Q \ln \pr{|x|^{-\gamma_{ij}}}  \, dx}
\leq C .
\end{align*}
It follows that $E_1 \lesssim_{(d, n, p, A, \V{\ga})} 1$, which, when combined with \eqref{BickelEst0} and \eqref{BickelEst1} shows that \eqref{BickelToProve} holds, as required.
\end{proof}

Next we prove Proposition \ref{RBrunnMinProp}, which states that if $V \in \ND$ and there exists a constant $B_V > 0$ so that $\disp \pr{\det \fint_Q V}^\frac{1}{d}  \leq B_V \fint_Q \pr{ \det V} ^\frac{1}{d}$ for every cube $Q \su \R^n$, then $V \in \NC$.

\begin{proof}[Proof of Proposition \ref{RBrunnMinProp}]
Observe that if $\{e_j\}_{j=1}^d$ is any orthonormal basis of $\Cd$, then
\begin{align*}
\abs{V^\frac12 (x) \pr{\fint_Q V}^{-1} V^\frac12 (x)}
& = \abs{V^\frac12 (x) \pr{\fint_Q V}^{-\frac12 } \brac{V^\frac12 (x) \pr{\fint_Q V}^{-\frac12 }}^*} \\
&= \abs{\brac{V^\frac12 (x) \pr{\fint_Q V}^{-\frac12 }}^* V^\frac12 (x) \pr{\fint_Q V}^{-\frac12 }}
 = \abs{\pr{\fint_Q V}^{-\frac12 } V(x) \pr{\fint_Q V}^{-\frac12 }} \\
&\leq  \sum_{j = 1}^d \innp {V(x)  \pr{\fint_Q V}^{-\frac12 } \V{e}_j, \pr{\fint_Q V}^{-\frac12 } \V{e}_j}.
\end{align*}
Thus, we have
$$\abs{\fint_Q V^\frac12 (x) \pr{\fint_Q V}^{-1} V^\frac12 (x) dx}
\leq \sum_{j = 1}^d \innp {\pr{\fint_Q V(x) dx}  \pr{\fint_Q V}^{-\frac12 } \V{e}_j, \pr{\fint_Q V}^{-\frac12 } \V{e}_j}
\leq d,$$
which implies that the largest eigenvalue of $\disp \fint_Q {V^\frac12 (x) \pr{\fint_Q V}^{-1} V^\frac12 (x)} dx$ is bounded above by $d$ for every cube $Q \su \R^n$.

Assume that $V \notin \NC$.
Looking at \eqref{NCCond}, this means that there exists a sequence of cubes $\set{Q_k}_{k=1}^\iny \su \R^n$ so that if we define $\disp V_k := \fint_{Q_k} {V^\frac12 (x) \pr{\fint_{Q_k} V}^{-1} V^\frac12 (x)} dx$, then each $V_k$ has a smallest eigenvalue $\la_{k,1} := \la_1(V_k)$ with the property that $\la_{k,1} \to 0$ as $k \to \iny$.
For $j = 1, \ldots, d$, let $\la_{k,j} := \la_j(V_k)$, the $j^{\text{th}}$ eigenvalue of $V_k$, and note that $\la_{k, j} \le d$ for $j = 2, \ldots, d$.
Then
\begin{equation}
\label{notND}
\begin{aligned}
&\inf \set{\det \brac{\fint_Q {V^\frac12(x) \pr{\fint_Q V}^{-1} V^\frac12(x)} dx} : Q \su \R^n} \\
\le& \inf\set{\det V_k : k \in \N}
= \inf\set{\prod_{j = 1}^d \la_{k, j} : k \in \N}
\le \inf\set{ \la_{k, 1} d^{d-1} : k \in \N}
= 0.
\end{aligned}
\end{equation}

However, for any $Q \su \R^n$, an application of \eqref{DetConvexIneq} to $\disp V^\frac12 \pr{\fint_Q V}^{-1} V^\frac12$ shows that
\begin{align*}
\det \brac{\fint_Q {V^\frac12(x) \pr{\fint_Q V}^{-1} V^\frac12(x)} dx}
&\ge \set{\fint_Q\brac{\det \pr{V^\frac12(x) \pr{\fint_Q V}^{-1} V^\frac12(x)}}^{\frac 1 d} dx}^d \\
&= \pr{\fint_Q\brac{\det V(x)}^{\frac 1 d} dx}^d  \det\pr{\fint_Q V}^{-1}
\ge B_V^{-d},
\end{align*}
where we have applied the assumption in the last inequality.
This contradicts \eqref{notND}, and therefore gives the desired conclusion.
\end{proof}

Finally, we provide the proof of Proposition \ref{umCompLem}.
Recall that Proposition \ref{umCompLem} states that if $V \in \Bp \cap \ND \cap \Ai$ for some $p > \frac n 2$, then $m(x, \la_1)  \le \um(x, V) \lesssim m(x, \la_1)$.

\begin{proof}[Proof of Proposition \ref{umCompLem}]
Let $r = \frac 1 {\um\pr{x, V}}$.
Choose $\V{e} \in \Sd$ so that
\begin{align*}
1 &= \innp{\Psi(x, r;V) \V{e}, \V{e}}
= \innp{\pr{\frac{1}{r^{n-2}} \int_{Q\pr{x,r}} V\pr{y}dy } \V{e}, \V{e}}
= \frac{1}{r^{n-2}} \int_{Q\pr{x,r}} \innp{V\pr{y} \V{e}, \V{e}} dy.
\end{align*}
Since $\innp{V\pr{y} \V{e}, \V{e}} \ge \la_1\pr{y}$, then it follows that $r \le \frac{1}{m\pr{x, \la_1}}$ so that
$$m(x, \la_1) \le \um(x, V).$$

Since $V \in \Ai$, then there exists $\de > 0$ so that if we define
$$S\pr{x, r} = \set{y \in Q\pr{x, r} : V\pr{y} \ge \de \fint_{Q\pr{x, r}} V\pr{z} dz},$$
then $\abs{S\pr{x, r}} \ge \frac 1 2 \abs{Q\pr{x,r}}$.
Then with $r = \frac 1 {\um\pr{x, V}}$ as above,
\begin{align*}
\Psi\pr{x, r; \la_1}
&= r^{2-n} \int_{Q\pr{x, r}} \innp{V\pr{y} \V{v}_1\pr{y}, \V{v}_1\pr{y}} dy
\ge r^{2-n} \int_{S\pr{x, r}} \innp{V\pr{y} \V{v}_1\pr{y}, \V{v}_1\pr{y}} dy \\
&\ge \de r^{2-n} \int_{S\pr{x, r}} \innp{\pr{\fint_{Q\pr{x, r}} V\pr{z} dz } \, \V{v}_1\pr{y}, \V{v}_1\pr{y}} dy \\
&\ge \frac \de 2 \fint_{S\pr{x, r}} \innp{\pr{r^{2-n}  \int_{Q\pr{x, r}} V\pr{z} dz} \, \V{v}_1\pr{y}, \V{v}_1\pr{y}} dy
= \frac \de 2 \fint_{S\pr{x, r}} \innp{ \underline{\Psi}\pr{x} \V{v}_1\pr{y}, \V{v}_1\pr{y}} dy \\
&\ge \frac \de 2 \fint_{S\pr{x, r}} 1 dy
= \frac \de 2.
\end{align*}
Applying the previous observation, then Lemma \ref{BasicShenLem} with the fact that $r \le \frac 1 {m\pr{x, \la_1}}$, we see that
\begin{align*}
\frac \de 2
&\le \Psi\pr{x, r; \la_1}
\le C_V \pr{r m\pr{x, \la_1}}^{2 - \frac n p} \Psi\pr{x, \frac 1 {m\pr{x, \la_1}}; \la_1}
=C_V \brac{\frac{ m\pr{x, \la_1}}{\um\pr{x, V}}}^{2 - \frac n p},
\end{align*}
where the last equality uses that $\Psi\pr{x, \frac{1}{m\pr{x, \la_1}}; \la_1} = 1$.
After rearranging, we see that $m\pr{x, \la_1} \gtrsim \um\pr{x, V}$, completing the proof.
\end{proof}

\end{appendix}


\end{document}